\documentclass[10pt]{article}

\usepackage[margin=1in]{geometry} 
\usepackage{graphicx}

\usepackage{booktabs} 
\usepackage{array} 
\usepackage{paralist} 
\usepackage{verbatim}
\usepackage{mathrsfs}
\usepackage{amssymb}
\usepackage{amsthm}
\usepackage{amsmath,amsfonts,amssymb}
\usepackage{esint}
\usepackage{graphics}
\usepackage{enumerate}
\usepackage{mathtools}
\usepackage{xfrac}
\usepackage{bbm}
\usepackage{caption}
\usepackage{subcaption}
\usepackage{tikz-cd}
\usepackage{adjustbox}

 \usepackage[T1]{fontenc} 
\usepackage[latin1]{inputenc}
\usepackage[english]{babel}

\usepackage[colorlinks=true, pdfstartview=FitV, linkcolor=blue, citecolor=blue, urlcolor=blue]{hyperref}

\numberwithin{equation}{section}
\numberwithin{figure}{section}

\newtheorem{theorem}{Theorem}[section]

\newtheorem{corollary}[theorem]{Corollary}
\newtheorem{proposition}[theorem]{Proposition}
\newtheorem{lemma}[theorem]{Lemma}
\newtheorem{OQ}{Open Question}

\theoremstyle{definition}
\newtheorem{definition}[theorem]{Definition}

\newtheorem{remark}{Remark}

\newcommand{\md}[1]{\ (\text{mod}\ #1)}

\newcommand*{\N}{\ensuremath{\mathbb{N}}}

\newcommand*{\Z}{\ensuremath{\mathbb{Z}}}

\newcommand*{\R}{\ensuremath{\mathbb{R}}}
\newcommand*{\Zd}{\ensuremath{\mathbb{Z}^d}}
\newcommand*{\Rd}{\ensuremath{\mathbb{R}^d}}

\newcommand{\eps}{\varepsilon}

\renewcommand*{\tilde}{\widetilde}
\renewcommand*{\hat}{\widehat}

\renewcommand{\S}{\mathbf{S}}

\newcommand{\E}{\mathbb{E}}

\DeclareSymbolFont{boldoperators}{OT1}{cmr}{bx}{n}
\SetSymbolFont{boldoperators}{bold}{OT1}{cmr}{bx}{n}
\edef\bar{\unexpanded{\protect\mathaccentV{bar}}\number\symboldoperators16}

\definecolor{darkgreen}{rgb}{0,0.6,0.05}

\definecolor{labelkey}{rgb}{0,0,1}

\newcommand{\indc}{{\boldsymbol{1}}}

\renewcommand{\P}{\mathbb{P}}

\def\Bb#1#2{{\def\md{\bigm| }#1\bigl[#2\bigr]}}

\def\Pb{\Bb\P}

\def\<#1{\langle #1\rangle}

\def\bi{\begin{itemize}}  
\def\ei{\end{itemize}}
\def\bnum{\begin{enumerate}} 
\def\enum{\end{enumerate}}
\def\ni{\noindent}

\def\Vor{\mathrm{Vor}}
\def\calE{\mathcal{E}}

\def\calF{\mathcal{F}}
\def\calG{\mathcal{G}}

\def\p{{\partial}}

\usepackage{titlesec}

\newcommand{\addperiod}[1]{#1.}
\titleformat{\section}
   {\centering\normalfont\Large}{\thesection.}{0.5em}{}
\titleformat*{\subsection}{\bfseries}
\titleformat{\subsubsection}[runin]
  {\normalfont\bfseries}
  {\thesubsubsection.}
  {0.5em}
  {\addperiod}
\titleformat*{\subsubsection}{\bfseries}
\titleformat*{\paragraph}{\bfseries}
\titleformat*{\subparagraph}{\large\bfseries}

\title{Phase transitions for the $XY$ model in non-uniformly elliptic and Poisson-Voronoi environments}

\author{Paul Dario
\thanks{LAMA, Universit\'e Paris-Est Cr\'eteil, CNRS UMR 8050, and CNRS, Cr\'eteil, France.
{\footnotesize paul.dario@u-pec.fr.}
}
\and 
Christophe Garban
\thanks{Universit\'e Claude Bernard Lyon 1, CNRS UMR 5208, Institut Camille Jordan, 69622 Villeurbanne, France, Institut Universitaire de France (IUF)
{\footnotesize garban@math.univ-lyon1.fr.}
}
}
\date{ }

\usepackage[nottoc,notlot,notlof]{tocbibind}

\begin{document}

\maketitle

\begin{abstract}

The goal of this paper is to analyze how the celebrated phase transitions of the $XY$ model are affected by the presence of a non-elliptic quenched disorder.

In dimension $d=2$, we prove that if one considers an $XY$ model on the infinite cluster of a supercritical percolation configuration,  the Berezinskii-Kosterlitz-Thouless (BKT) phase transition still occurs despite the presence of quenched disorder. The proof works for all $p>p_c$ (site or edge). We also show  that the $XY$ model defined on a planar Poisson-Voronoi graph also undergoes a BKT phase transition. 

When $d\geq 3$, we show in a similar fashion that the continuous symmetry breaking of the $XY$ model at low enough temperature is not affected by the presence of  quenched disorder such as supercritical percolation (in $\mathbb{Z}^d$) or Poisson-Voronoi (in $\mathbb{R}^d$).

Adapting either Fr\"{o}hlich-Spencer's proof of existence of a BKT phase transition \cite{frohlich1981kosterlitz} or the more recent proofs \cite{lammers2022height, van2023elementary,  aizenman2021depinning, van2023duality} to such non-uniformly elliptic disorders appears to be non-trivial. Instead, our proofs rely on Wells' correlation inequality~\cite{Wellsthesis}. \end{abstract}

\section{Introduction}

\subsection{Context.}  We consider the classical $XY$ (or rotator) model on the hypercubic lattice $\Zd$ with $d \geq 2$. The model is formally defined as follows: given a finite subset $\Lambda \subseteq \Zd$, and an inverse temperature $\beta \geq 0$, we define the probability measure $\mu_{\Lambda, \beta}$ on the set of configurations $[0 , 2\pi)^{\Lambda} := \left\{ \theta : \Lambda \to [0 , 2 \pi) \right\}$ according to the formula
\begin{equation} \label{eq:XYmodel}
    \mu_{\Lambda , \beta}(d \theta) := \frac{1}{Z_{\Lambda, \beta}} \exp \left( \beta \sum_{\substack{x , y \in \Lambda \\ x \sim y}} \cos(\theta_x - \theta_y )\right) \prod_{x \in \Lambda} d \theta_x,
\end{equation}
where $d \theta_x$ denotes the uniform measure on the interval $[0 , 2\pi)$, $Z_{\Lambda, \beta}$ is the normalizing constant which ensures that $\mu_{\Lambda , \beta}$ is a probability measure and the notation $x \sim y$ means that $x$ and $y$ are nearest neighbour in $\Zd$. By setting $S := e^{i \theta}$, the model can be equivalently considered as a spin system valued in the circle $\S^1$. The existence of a thermodynamic limit for the $XY$ model (i.e., an infinite-volume limit as $\Lambda \to \infty$) is guaranteed by compactness arguments (as the spins take values in the compact space $\S^1$) and the Ginibre correlation inequality \cite{ginibre1970general} (N.B. the present setup corresponds to the so-called free boundary conditions. A limit also exists for Dirichlet boundary conditions. Furthermore, when $d=2$,  both limits are known to agree 
 as shown in~\cite{MMP}).
 We will denote the free infinite volume limit by $\mu_\beta$ (see~\cite{MMP}).
These results imply in particular the convergence of the two-point function (and in fact the Ginibre correlation inequality~\cite{ginibre1970general} implies its monotonicity in the domain): for each $x , y \in \Zd$,
\begin{equation*}
    \left\langle \cos(\theta_x - \theta_y) \right\rangle_{\mu_{\Lambda, \beta}} \underset{\Lambda \uparrow \Zd}{\longrightarrow} \left\langle \cos(\theta_x - \theta_y) \right\rangle_{\mu_\beta}.
\end{equation*}

The $XY$ model is known to undergo phase transitions which are of different nature depending on the dimension. In dimension $d \geq 3$, the model exhibits an order/disorder phase transition characterised by long-range order at low temperature, i.e., $\left\langle \cos(\theta_x - \theta_y) \right\rangle_{\mu_{\beta}} \geq c > 0$ for $\beta \gg 1$, and exponential decay of the two-point function at high temperature, i.e., $\left\langle \cos(\theta_x - \theta_y) \right\rangle_{\mu_{\beta}} \leq \exp(-c |x - y|)$ for $\beta \ll 1$ (N.B. This phase transition is conjectured to be sharp in the sense that there should exist a critical inverse temperature $\beta_c(d) > 0$ below which the model exhibits exponential decay and above which the model exhibits long-range order). The existence of this phase transition was originally established by Fr\"{o}hlich, Simon and Spencer~\cite{frohlich1976infrared}, and we also mention the alternative successful approaches to the question of Fr\"{o}hlich and Spencer~\cite{frohlich1982massless}, Kennedy and King~\cite{kennedy1986spontaneous}, and the second author and Spencer~\cite{garban2022continuous}.

In two dimensions, the phenomenology of the model is different and the Mermin-Wagner theorem~\cite{mermin1966absence} rules out the existence of long-range order at any positive temperature. Moreover, McBryan and Spencer~\cite{mcbryan1977decay} showed that the two-point function decays at least polynomially fast at any positive temperature (see also \cite{shlosman1978decrease}). The model still undergoes a phase transition known as the Berezinskii-Kosterlitz-Thouless (BKT) phase transition and characterised by a different asymptotic behaviour of the two-point function: this function decays polynomially fast in the distance between the two points at low temperature and exponentially fast at high temperature. The existence of this phase transition was established in a celebrated work by Fr\"{o}hlich and Spencer~\cite{frohlich1981kosterlitz}, and has been the subject of recent new developments by Lammers~\cite{lammers2022height, lammers2022dichotomy, lammers2023bijecting}, van Engelenburg and Lis~\cite{van2023duality, van2023elementary}, and Aizenman, Harel, Peled and Shapiro~\cite{aizenman2021depinning} (see also the survey \cite{kharash2017fr}). All these proofs rely on a duality argument and require to establish the delocalization of a model of integer-valued random height function in the high temperature regime. In this line of research, we mention the recent important contribution of Bauerschmidt, Park and Rodriguez~\cite{bauerschmidt2022discrete, bauerschmidt2022discreteII} who identified the scaling limit of the integer-valued Gaussian free field (in the high temperature regime).

\begin{figure}[!htp]
\begin{center}
\includegraphics[width=0.48\textwidth]{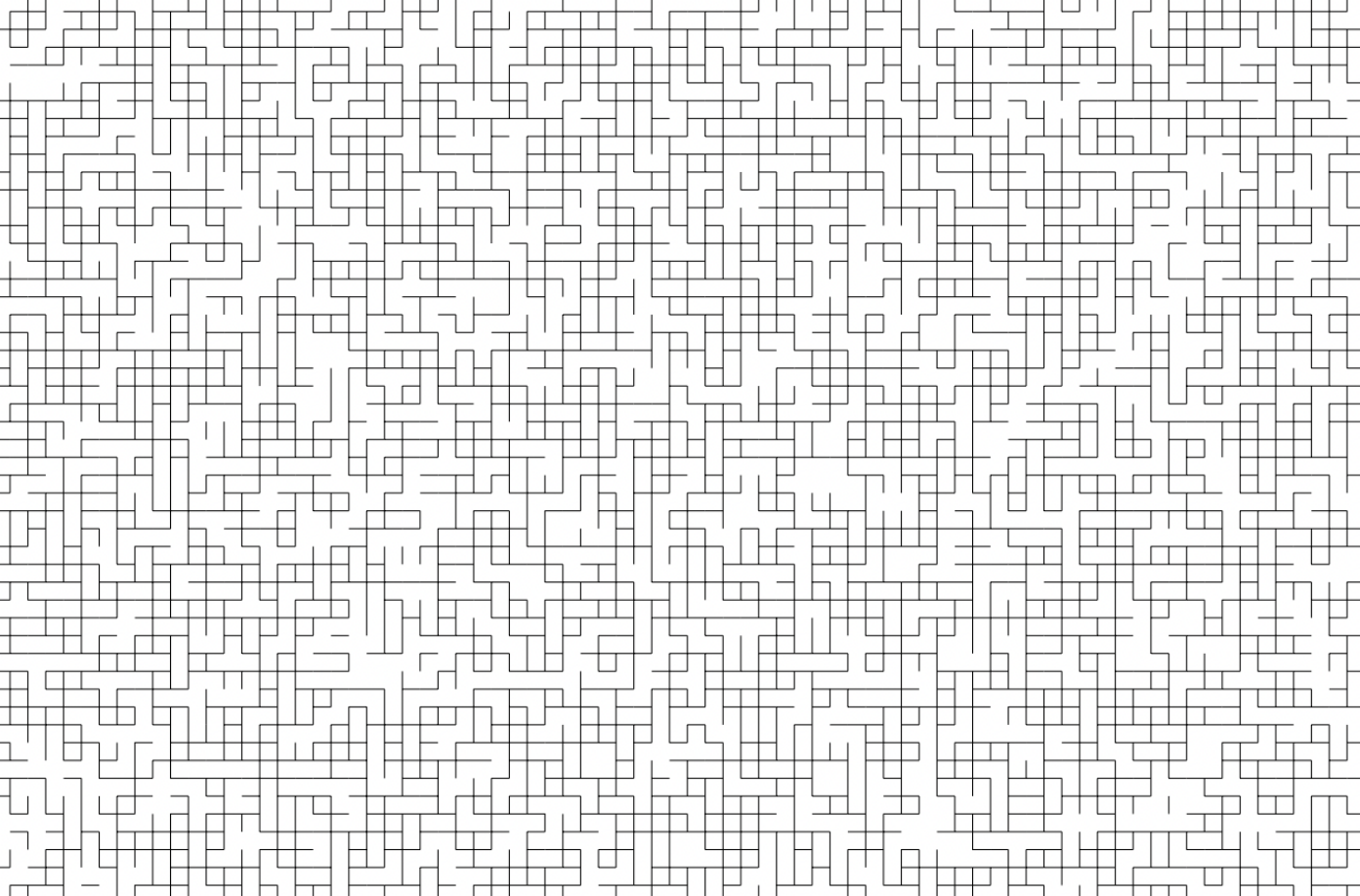}
\hskip 0.04\textwidth
\includegraphics[width=0.44\textwidth]{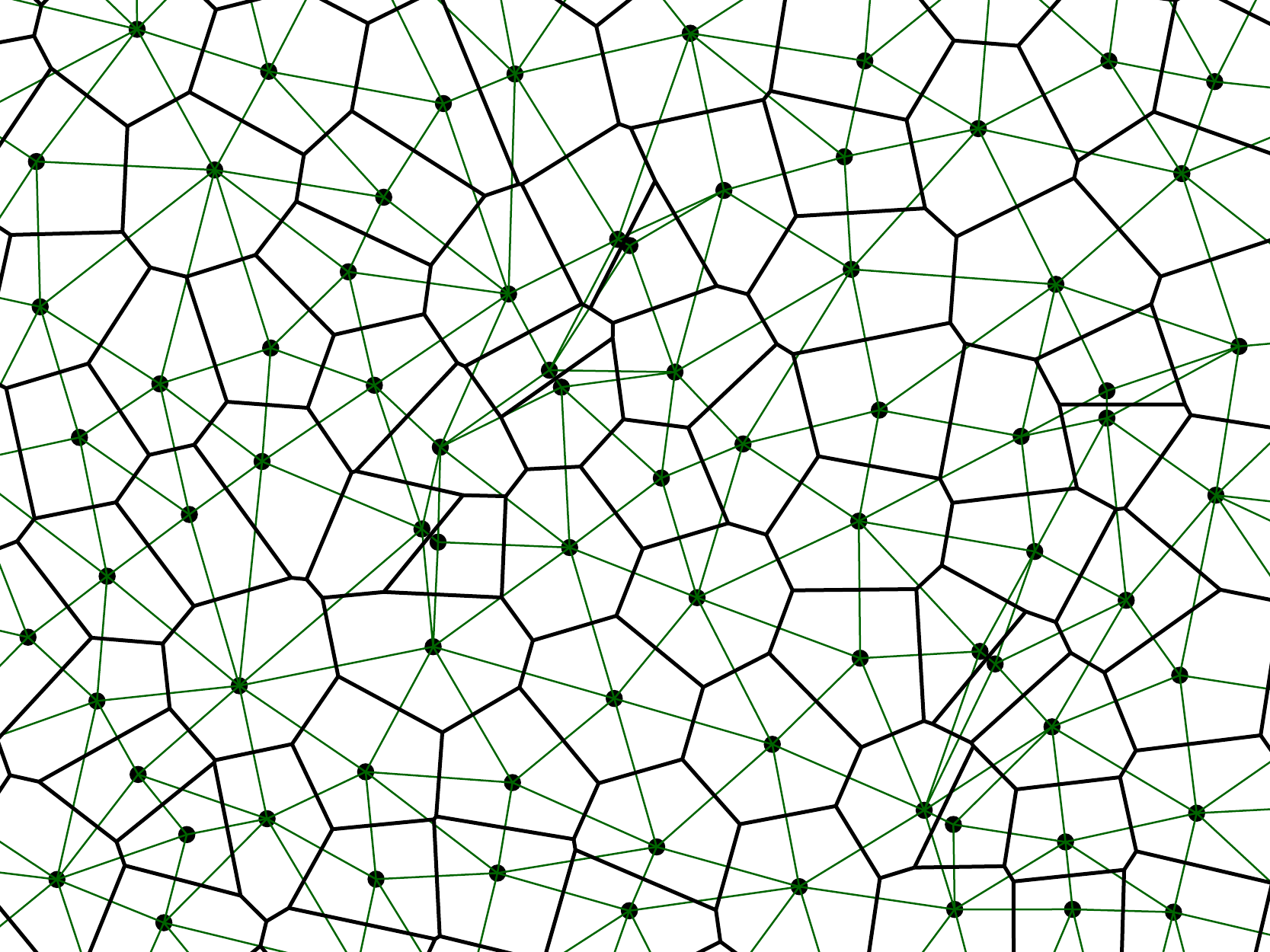}
\end{center}
\caption{The two types of quenched disorders for which we prove that the BKT phase is non-trivial: on the left is pictured a supercritical edge-percolation configuration $\omega_p$ (with $p=0.66> p_{c, \mathrm{edge}}(2)=\tfrac 1 2$).  On the right, the sites represent a Poisson point process $P$ in the plane $\R^2$. Each point $x\in P$ carries a spin $\sigma_x\in \S^1$. Two such spins interact if and only if their Voronoi cells (pictured with black edges) intersect. Or equivalently, if  they are connected by an edge in the dual Delaunay triangulation (pictured in green). }\label{}
\end{figure}

In this article, we are interested in the existence of these phase transitions for the $XY$ model in the presence of a random quenched disorder in the underlying graph.
When $d\geq 2$, the stability of the existence of such phase transitions under small random perturbations of the graph is a natural question on its own. We also have a second motivation in the specific $d=2$ case, namely its relevance for the analysis of the $2d$ classical {\em Heisenberg model}. This model  is similar to the $XY$ model except it assigns spins $\sigma_x\in \S^2$ to each $x\in \Z^2$.
Polyakov conjectured in 1975 (\cite{polyakov1975interaction}) that this model should never enter a BKT phase at large $\beta$, as opposed to the case of the $XY$ model. 
It is not hard to see that the Heisenberg model ($\S^2$ valued) and the $XY$ model ($\S^1$ valued) are related to each other via the following observation:  by conditioning on the vertical directions of the $\S^2$-valued spin field, one ends up with a $XY$ model on $\Z^2$ with random conductances $\{J_e\}_e$. One may then rephrase Polyakov's celebrated  conjecture by saying that this quenched disorder (which is $\beta$-dependent) will destroy the power-law decay of the induced $XY$ model at all $\beta>0$. We refer to the work \cite{aru2022percolation} for a detailed discussion and links to the literature. The work \cite{aru2022percolation} describes locally what this field of random conductances is and produces laws of environments which are in some sense consistent with Polyakov's conjecture. Of course the true field $\{J_e\}_e$ of such induced random conductances is far from being an i.i.d field of Bernoulli variables as considered in the present work. Yet, viewed from this angle,  we believe it is an important question  to know whether the BKT phase is robust or not under such simpler models of quenched disorder. 

\medskip

We will consider in this paper two natural examples of quenched disorder in $d$-dimension on which we will analyse the $XY$ model: 
\bnum
\item  The infinite cluster of a supercritical  i.i.d Bernoulli percolation disorder (either site or edge) on the lattice $\Z^d$.
\item The Poisson-Voronoi graph induced by a Poisson Point Process $P$ of intensity 1 on $\R^d$.
\enum
We now describe both  settings. More details, definitions will be given in Section \ref{Section2}.

\medskip

\ni
\textbf{Percolation disorder setup.}  We consider the $XY$ model in a random environment given by a Bernoulli site percolation and defined as follows. We let $(r_x)_{x \in \mathbb{Z}^d}$ be a collection of i.i.d. Bernoulli random variables of probability $p \in (0,1)$ indexed on the vertices of $\Zd$. We refer to the collection $(r_x)_{x \in \mathbb{Z}^d}$ as the environment and denote by $\E_p$ the corresponding expectation. For each realization of the environment $(r_x)_{x \in \mathbb{Z}^d}$, we consider the finite-volume Gibbs measure defined by the formula
\begin{equation} \label{eq:18221905}
    \mu_{\Lambda, \beta,  r}(d S) := \frac{1}{Z_{\Lambda, \beta, r}} \exp \left( \beta \sum_{\substack{x , y \in \Lambda \\ x \sim y}} r_{x} r_y \cos(\theta_x - \theta_y) \right) \prod_{x \in \Lambda} d \theta_x.
\end{equation}
We denote by $ \mu_{\beta,  r}$ the thermodynamic limit obtained by taking the limit $\Lambda \uparrow \Zd$ (whose existence and uniqueness is once again guaranteed by compactness arguments and the Ginibre inequality).
We prove that, when the probability $p$ is larger than $p_{c , \mathrm{site}}(d)$, where $p_{c , \mathrm{site}}(d)$ is the critical site percolation threshold on $\Zd$, and the inverse temperature $\beta$ is sufficiently high (depending on $p$), then the expectation of the two-point function with respect to the environment decays polynomially fast in two dimensions. In dimension $d \geq 3$, we show that the model exhibits long-range order when  $p > p_{c , \mathrm{site}}(d)$ and $\beta$ is sufficiently large (depending on $p$ and $d$). Similar results hold when the underlying environment is given by a Bernoulli edge percolation (the model is formally defined by replacing the product $r_x r_y$ by $\omega_{xy}$ where $\omega \in \{ 0,1 \}^{E(\Zd)}$ is an edge percolation configuration in the definition~\eqref{eq:18221905}).

\medskip
\ni
\textbf{Poisson-Voronoi setup.} Let $P$ be a Poisson Point Process of unit intensity on $\R^d$. To each point $x\in P$, we associate its Voronoi cell:
\begin{align}\label{e.Vorx}
\Vor(x)=\Vor(x,P):=\{ z\in \R^d, |z-x| \leq \min_{y\in P\setminus \{x\}}  |z-y| \}\,.
\end{align}
It is easy to check that a.s. $\bigcup_{x\in P} \Vor(x,P)$ induces a covering of $\R^d$ with random compact sets (which may intersect along their boundaries).  

Furthemore, on may also check that this covering induces the following a.s. locally finite graph $G=(V,E)$ which we shall call the {\em Voronoi graph of $P$}: 
\bi
\item Its vertex set $V$ is given by the points of the Poisson point process $P$.
\item For the edge set $E$, we draw an edge $e=\{x,y\}$ between two distinct points if and only if $\Vor(x,P)\cap \Vor(y,P) \neq \emptyset$. 
\ei

 Statistical physics models on the Voronoi graph of a Poisson point process have been the focus of many important works in particular in the context of ``Voronoi percolation'', see for example \cite{benjamini1998conformal,bollobas2006critical,tassion2016crossing,ahlberg2016quenched,vanneuville2019annealed,vanneuville2021annealed} and references therein. 
 In this work we consider the $XY$ model on the Voronoi graph $G$ of a Poisson point process $P$ in $\R^d$.  More precisely, for any finite square domain $\Lambda \subset \R^d$, we consider the finite-volume Gibbs measure on $(\S^1)^{\Lambda \cap P}$ defined by 
\begin{equation} \label{e.VorF}
    \mu_{\Lambda, \beta,  P}(d S) := \frac{1}{Z_{\Lambda, \beta, P}} \exp \left( \beta \sum_{\substack{x \neq y \in P \cap \Lambda  \\ \Vor(x,P)\cap \Vor(y,P)\neq \emptyset }} F(\Vor(x,P),\Vor(y,P)) \cos(\theta_x - \theta_y) \right) \prod_{x \in \Lambda} d \theta_x,
\end{equation}
where we allow the interaction strength between neighbouring cells  $x$ and $y$ to depend on the joint shape of $\Vor(x,P)$ and $\Vor(y,P)$.
We will analyse the following three cases\footnote{Our techniques would cover many other natural choices as well}:
\begin{align}\label{e.choices}
\begin{cases}
F_1(V,\tilde V) = 1 & \text{ i.e. the interaction strength does not depend on the geometry} \\
F_2(V,\tilde V) =  f(|V|) f(|V'|)  & \text{ for some continuous strictly increasing function $f: \R_+ \to \R_+$} \\
F_3(V, \tilde V) = \lambda^{d-1}(V \cap \tilde V) &\text{ where $\lambda^{d-1}$ denotes the volume of hyperplanes.}
\end{cases}
\end{align}
The third interaction models, say, a {\em friction} between two neighbouring cells which would be proportional to the surface separating the two cells. 
Notice that our choice of finite-volume Gibbs measure corresponds here to {\em free boundary conditions} and is built in such a way that as $\Lambda$ ``increases'', new points and new edges arise in an increasing manner. As such, Ginibre inequality can be used again. It gives us  the monotony which implies the existence of an infinite volume limit Gibbs measure as $\Lambda \nearrow \R^d$ denoted by $\mu_{\beta, P}$.

\medskip

Our main results are collected below and come in two flavours. Theorem~\ref{mainthm} and Corollary~\ref{proofKTbonds} state the existence of the phase transitions in the case of a highly supercritical percolation (for the site and edge percolation models). These results assume that the percolation probability is very close to $1$ but allow for (relatively) small values of the inverse temperature $\beta$. Theorem~\ref{proofKTsitegeneral} states the existence of the phase transitions for the $XY$ model for a supercritical percolation cluster. These results only assume that the percolation probability is larger than the site (resp. edge) percolation threshold but require the inverse temperature $\beta$ to be large.  In the case of Poisson-Voronoi, our main result is Theorem \ref{th.Vor} which states that the $XY$ model undergoes the same phase transition on the Voronoi graph as on $\Z^d$. Note that we do not have two types of statements in this case, as one does not have a parameter $p>p_c$ to play with. (The intensity is set to be 1 as law of the induced combinatorial graph is insensitive to rescaling space).

\subsection{Main results.}

Before stating the results, we introduce some definitions and notation. In two dimensions, we denote by $\beta_{BKT} > 0$ the critical inverse temperature for the BKT phase transition, i.e.
\begin{equation*}
    \beta_{BKT} := \sup \left\{ \beta \geq 0 \, : \, \left\langle \cos(\theta_0 - \theta_x) \right\rangle_{\mu_\beta} ~\mbox{decays exponentially fast}\right\}.
\end{equation*}
The results of~\cite{frohlich1981kosterlitz, aizenman2021depinning, van2023elementary} imply that $\beta_{BKT} < \infty$, and the results of~\cite[Section 6]{van2023elementary} imply that for any $\beta \geq \beta_{BKT}$, the two-point function decays polynomially fast.

In dimension $d \geq 3$, we denote by $\beta_c(d)$ the critical point associated with the order/disorder phase transition, and formally define it according to the formula
\begin{equation*}
    \beta_{c}(d) := \sup \left\{ \beta \geq 0 \, : \, \left\langle \cos(\theta_0 - \theta_x) \right\rangle_{\mu_\beta} \underset{|x| \to \infty}{\longrightarrow}  0\right\}.
\end{equation*}
The results of~\cite{frohlich1976infrared, frohlich1982massless, kennedy1986spontaneous, garban2022continuous} imply that $\beta_{c}(d) < \infty$ and the Messager-Miracle-Sole inequality~\cite{messager1977correlation} implies that for any $\beta > \beta_c(d)$, the two-point function remains bounded away from $0$ (i.e., we have $\inf_{x \in \Zd} \left\langle \cos(\theta_0 - \theta_x) \right\rangle_{\mu_\beta} >0$).

\begin{theorem}[Phase transitions for the $XY$ model on a high density Bernoulli site percolation cluster] \label{mainthm}
    For the $XY$ model in a random environment given by a Bernoulli site percolation, the following hold:
    \begin{itemize}
        \item In dimension $d = 2$, for any inverse temperature $\beta \geq 4 \beta_{BKT}$ and any probability $p \geq \frac 1 {1+e^{-16 \beta_{BKT}}}$, the function $x \mapsto \mathbb{E}_p [  \left\langle \cos(\theta_0 - \theta_x) \right\rangle_{\mu_{\beta, r}}]$ decays polynomially fast in $|x|$.
        \item In dimension $d \geq  3$, for any inverse temperature $\beta > 4 \beta_{c}(d)$ and any probability $ p >  \frac 1 {1+e^{-8d \beta_{c}(d)}}  $, the function $x \mapsto \mathbb{E}_p [  \left\langle \cos(\theta_0 - \theta_x) \right\rangle_{\mu_{\beta, r}}]$ is bounded away from $0$.
    \end{itemize}
\end{theorem}

\begin{remark} \label{remark1.2}
Let us make a few remarks about the previous theorem:
\begin{itemize}
    \item Using the results of~\cite{van2023elementary}, the following lower bound can be deduced from the proof: there exists a (universal) constant $c > 0$ such that
    \begin{equation*}
        \mathbb{E}_p [  \left\langle \cos(\theta_0 - \theta_x) \right\rangle_{\mu_{\beta, r}}] \geq \frac{c}{|x|} ~\mbox{in}~ d=2.
    \end{equation*}
    \item The Ginibre correlation inequality~\cite{ginibre1970general} implies that for any $\beta < \beta_{BKT}$ and any percolation configuration $(r_x)_{x \in \Z^2} \in \{0 , 1\}^{\Z^2}$, the two-point function $x \mapsto  \left\langle \cos(\theta_0 - \theta_x) \right\rangle_{\mu_{\beta, r}}$ decays exponentially fast in two dimension. Similarly, for any $\beta < \beta_{c}(d)$ and any percolation configuration $(r_x)_{x \in \Zd} \in \{0 , 1\}^{\Zd}$, the two-point function $x \mapsto \left\langle \cos(\theta_0 - \theta_x) \right\rangle_{\mu_{\beta, r}}$ converges to $0$ in dimension $d \geq 3$.
    \item 
    It is not the first result about the $XY$ model in non-uniformly elliptic environment: 
    in \cite{bricmont1985some,dunlop1985correlation}, it is shown that the BKT transition still holds for anisotropic versions of the $2d$ classical Heisenberg model (at low enough temperature). Their proof would not apply to the present percolation quenched disorder. Still they relied on Wells' correlation inequality and therefore their work was a source of inspiration for us.
    \item We believe that the assumption that the inverse temperature is larger than two times the critical parameters is not intrinsic to the problem and is a limitation of the proof (see Open Question 4).
    
\item 
For any $d\geq 2$, and for any $p\in (0,1)$, if one wants to obtain the similar result for  i.i.d edge percolation rather than site percolation, one can easily do it using Theorem \ref{mainthm}. To see why, notice that if $u\in (0,1)$ is fixed and at each site $x\in \Z^d$, we sample independent $u$-Bernoulli variables $\{Z_{x,y}\}_{y\sim x}$, then 
\bnum
\item the edge percolation $\omega=\{\omega_{xy}\}_{x\sim y}$ defined by $\omega_{x,y}:= Z_{x,y} Z_{y,x}$ is an i.i.d percolation with parameter $p= u^2$. 
\item the site percolation $r=\{r_x\}_{x\in \Z^d}$ defined by $r_x:=\prod_{y\sim x} Z_{x,y}$ is an i.i.d site percolation with parameter $\bar p:= u^{2d} \leq p$.
\item The edge percolation clusters of $\omega$ stochastically dominate the site percolation clusters of $r$ (i.e., $\omega_{xy} \geq r_x r_y$ for each pair of neighbouring vertices $x , y$). An application of the Ginibre correlation inequality thus implies the following corollary. 
\enum

\end{itemize}
\end{remark}

\begin{corollary}[Phase transitions for the $XY$ model on a high density Bernoulli edge percolation cluster]\label{proofKTbonds}
For the $XY$ model in a random environment given by a Bernoulli edge percolation $\omega$ of parameter~$p$, the following hold:
    \begin{itemize}
        \item In dimension $d = 2$, for any inverse temperature $\beta \geq 4 \beta_{BKT}$ and any probability $p\geq \big( \frac 1 {1+ e^{-16 \beta_{BKT}}} \big)^{1/2}$ the function $x \mapsto \mathbb{E}_p [  \left\langle \cos  ( \theta_0 - \theta_x ) \right\rangle_{\beta, \mathrm{edge}, \omega}]$ decays polynomially fast in $|x|$.
        \item In dimension $d \geq  3$, for any inverse temperature $\beta > 4 \beta_{c}(d)$
        and any probability $p > \big( \frac 1 {1+ e^{-8d \beta_{c}(d)}} \big)^{1/d}$ the function $x \mapsto \mathbb{E}_p [  \left\langle \cos  ( \theta_0 - \theta_x ) \right\rangle_{\beta, \mathrm{edge}, \omega}]$ is bounded away from $0$.
    \end{itemize}
\end{corollary}

Theorem~\ref{mainthm} and Corollary~\ref{proofKTbonds} are quantitative over the parameter $\beta$ but only apply to large values of the probability~$p$. In our second main result, we build upon Theorem~\ref{mainthm} and a renormalization argument, and show that, for any percolation probability $p > p_{c, \mathrm{site}}(d)$, the BKT phase transition occurs in dimension $d = 2$ and the model exhibits long-range order for sufficiently large $\beta$ in dimension $d \geq 3$.

\begin{theorem}[Phase transitions for the $XY$ model on a supercritical Bernoulli percolation cluster] \label{proofKTsitegeneral}
For the $XY$ model in a random environment given by a Bernoulli site percolation of parameter $p$, the following hold:
    \begin{itemize}
        \item In dimension $d = 2$, for any $p > p_{c , \mathrm{site}}(2)$, there exists an inverse temperature $\beta_{BKT}(p) < \infty$, such that, for any $\beta \geq \beta_{BKT}(p)$, the function $x \mapsto \mathbb{E}_p [  \left\langle \cos  ( \theta_0 - \theta_x ) \right\rangle_{\beta, r}]$ decays polynomially fast in $|x|$.
        \item In dimension $d \geq  3$, for any $p > p_{c , \mathrm{site}}(d)$, there exists an inverse temperature $\beta_{c}(d, p) < \infty$, such that, for any $\beta \geq \beta_{c}(d, p)$, the function $x \mapsto \mathbb{E}_p [  \left\langle \cos  ( \theta_0 - \theta_x ) \right\rangle_{\beta, r}]$ is bounded away from $0$.
    \end{itemize}
Furthermore, the same result holds for supercritical edge percolation clusters, i.e.  for any $p> p_{c, \mathrm{edge}}(d)$ (where $p_{c, \mathrm{edge}}(d)$ is the edge percolation threshold on $\Zd$) and for sufficiently large inverse temperature, the model enters either the BKT or the long-range-order phase.   
\end{theorem}

\begin{remark} \label{remark1.4}
Let us make a few remarks about the previous theorem:
\begin{itemize}
    \item The proofs written below (based on the argument of~\cite{van2023elementary, garban2022continuous}) gives the lower bounds: there exists a constant $c := c(p) > 0$,
    \begin{equation*}
        \mathbb{E}_p [  \left\langle \cos(\theta_0 - \theta_x) \right\rangle_{\mu_{\beta, r}}] \geq \frac{c}{|x|} ~\mbox{in}~ d=2 ~~\mbox{and}~~ \inf_{x \in \Zd} \mathbb{E}_p [  \left\langle \cos(\theta_0 - \theta_x) \right\rangle_{\mu_{\beta, r}}] \underset{\beta \to \infty}{\longrightarrow} 1 ~\mbox{in}~ d\geq 3.
    \end{equation*}
    \item Regarding the dependency of the parameters $\beta_{BKT}(p)$ and $\beta_c(d, p)$ in the variable $p$, an investigation of the proof shows that an explicit lower bound depending on the values of the constants in Proposition~\ref{proppre-good} could be obtained.
    
\item Note that this theorem (as well as Theorem \ref{mainthm}) is not an almost sure statement (in the disorder, i.e. in the percolation configuration  $r$) about the existence of a BKT phase transition (or the  existence of a  long-range order in $d\geq 3$). The next corollary provides such an almost sure statement. As explained in Section~\ref{subsection5.3}, it follows readily from the above theorem and the pointwise ergodic theorem.   
\end{itemize}
\end{remark}

\begin{corollary}\label{c.AlmostSure}
With the same setting as in Theorem \ref{proofKTsitegeneral}:
    \begin{itemize}
        \item In dimension $d = 2$, for any $p > p_{c , \mathrm{site}}(2)$, there exists an inverse temperature $\beta_{BKT}(p) < \infty$ such that, for any $\beta > \beta_{BKT}(p)$, there exists $c=c(p)>0$ such that  the following holds for almost every percolation configuration $r\in \{0,1\}^{\Z^2}\sim \P_p$:

\begin{align*}\label{}
\text{for every $m\geq 1$, }\liminf_{L\to \infty}  
\left\{ \frac 1 {L^2 m^2} \sum_{x\in \Lambda_L} \sum_{y\in x + \Lambda_m} 
\<{\cos(\theta_x - \theta_y)}_{\mu_{\beta, r}}  \right\} \geq \frac c m \,.
\end{align*}
        
        \item In dimension $d \geq  3$,  for any $p > p_{c , \mathrm{site}}(d)$, there exists an inverse temperature $\beta_{c}(d, p) < \infty$, such that, for any $\beta > \beta_{c}(d, p)$, there exists a constant $c=c(p,\beta,d) > 0$ such that  the following holds for almost every percolation configuration $r\in \{0,1\}^{\Z^d}\sim \P_p$: 
\begin{align*}\label{}
\text{for every $m\geq 1$, }\liminf_{L\to \infty}  
\left\{ \frac 1 {L^d m^d}  \sum_{x\in \Lambda_L} \sum_{y\in x + \Lambda_m}
\<{\cos(\theta_x - \theta_y)}_{\mu_{\beta, r}}  \right\} \geq c >0  \,.
\end{align*}     
\end{itemize}
\end{corollary}

As mentioned above, we obtain similar results when the underlying quenched disorder is generated by a Poisson Point process $P$ of unit intensity in $\R^d$. The following result establishes that if the inverse temperature $\beta$ is sufficiently large, then the two-point function decays polynomially fast in two dimensions and remains bounded away from~$0$ in dimension $d \geq 3$.

\begin{theorem}[Phase transitions for the $XY$ model on  Poisson-Voronoi]\label{th.Vor}
For the $XY$ model defined on the random Voronoi graph $G$ induced by a Poisson Point Process $P$ and with interactions strength $F\in \{F_1,F_2,F_3\}$ from~\eqref{e.VorF}, the following holds:
    \begin{itemize}
        \item In dimension $d = 2$, there exists an inverse temperature $\beta_{BKT,Vor}(F) < \infty$, such that, for any $\beta > \beta_{BKT,Vor}(F)$, the function $x \mapsto \mathbb{E} [  \left\langle \cos  ( \theta_{\hat 0_P} - \theta_{\hat x_P} ) \right\rangle_{\mu_{\beta, P}}]$ decays polynomially fast in~$|x|$, where $\hat 0_P$ (resp. $\hat x_P$)  denote the closest point in $P$ to $0$  (resp. to $x$) for the Euclidean norm.
        \item In dimension $d \geq  3$, there exists an inverse temperature $\beta_{c,Vor}(d,F) < \infty$, such that, for any $\beta > \beta_{c,Vor}(d,F)$, the function $x \mapsto \mathbb{E} [  \left\langle \cos  ( \theta_{\hat 0_P} - \theta_{\hat x_P} ) \right\rangle_{\mu_{\beta, P}}]$ is bounded away from $0$.
    \end{itemize}
\end{theorem}

\subsection{Technique of proof and comparison with previous works.}

\ni
\textbf{Known proofs of BKT phase transition and quenched disorder.}
Let us first emphasise why the current proofs of BKT phase transitions do not seem to work (without very substantial effort at least) with the types of disorders we are interested in. 
We will discuss the approaches of Fr\"{o}hlich-Spencer as well as Lammers, van Engelenburg-Lis and Aizenman-Harel-Peled-Shapiro:
\bnum
\item {\em Coulomb gas expansion} from \cite{frohlich1981kosterlitz}. 

\item {\em Delocalisation and loop expansions} from \cite{lammers2022height, van2023elementary,  aizenman2021depinning, van2023duality}. 
\enum

Both arguments are based on the important observation that the existence of the BKT phase transition can be proved by establishing the existence of a \emph{roughening} phase transition for a dual model of integer-valued height functions defined on the plaquettes of $\Z^2$. The first approach~\cite{frohlich1981kosterlitz} establishes the roughening phase transition through a remarkable and delicate multiscale analysis. The second approach relies on a recent breakthrough by Lammers~\cite{lammers2022height} and is based on a noncoexistence theorem for planar percolation from Sheffield~\cite{sheffield2005random} (see also~\cite{duminil2019sharp} for a simpler proof).

In order to adapt these techniques to the $XY$ model in a random environment, one needs to establish the existence of the roughening phase transition for the dual integer-valued height function in a dual random environment (several plaquettes of $\Z^2$ are conditioned to carry the same value). The addition of this disorder makes the adaptation of the multiscale analysis from~\cite{frohlich1981kosterlitz} difficult. Specifically when a large cluster of plaquettes are conditioned to carry the same value, they create a vertex of large degree in the resulting graph which seems to be an obstacle to the extension of the argument.

The second approach of Lammers~\cite{lammers2022height} makes important use of the noncoexistence result of Sheffield~\cite{sheffield2005random} which seems difficult to extend when the underlying environment is not translation-invariant (which is the case for almost every realization of the disorder).

\begin{remark}\label{r.EIT}
When $d\geq 3$, it would of course be impossible to adapt reflection positivity techniques (\cite{frohlich1976infrared}) to the presence of quenched disorder. Yet, let us mention another possible way to prove Theorem \ref{mainthm} in the case $d\geq 3$: the work \cite{garban2022continuous} provides a new proof of long-range-order of the $XY$ model in any $d\geq 2+\eps$ which requires the so-called {\em Exponential Intersection Property (EIT) property} (see \cite{benjamini1998unpredictable} or Theorem~\ref{thmunpredic} below). As pointed out to us by Yuval Peres, it can be shown that the EIT property a.s. still holds on a sufficiently supercritical percolation cluster in $\Z^d, d\geq 3$. 
This gives a different way to proceed for Theorem \ref{mainthm}. Yet in order to identify LRO for all  $p>p_{c, \mathrm{edge}}(d)$, one would still need the same renormalization argument as in the proof of Theorem \ref{proofKTsitegeneral}. 

Also, in $d=2$, none of these ideas apply: the EIT property is not true on $\Z^2$ and Well's inequality seems to be the only way out. 
\end{remark}

\ni
\textbf{Proof idea.}
The main idea in this work is to transfer the a priori difficult analysis of the $XY$ model on a lattice with {\em quenched} disorder to an {\em annealed} version of the $XY$ model using a remarkable inequality called Wells' inequality~\cite{Wellsthesis}. This inequality was recently advertised in the work of Madrid, Simon and Wells~\cite{madrid2022comparison} (see also~\cite{tale3coauthors}), but does not seem to have received much attention before that (except in \cite{aizenman1980comparison, bricmont1985some,dunlop1985correlation}). 

The specific form of Wells' inequality can be found in Section~\ref{sectinoWellsineq}. Its proof relies on techniques which are similar to the proof of the Ginibre correlation inequality~\cite{ginibre1970general} (use of duplicated variables, expansion of the exponentials etc.), and, as it is the case for the Ginibre inequality, it applies to a general class of $XY$ models. The proof is reproduced in Section~\ref{sectinoWellsineq} below in the same level of generality.

In the context of the nearest neighbour $XY$ model in a random environment considered in this article, an almost immediate consequence of Wells' inequality is the following result: given a finite set $\Lambda \subseteq \Zd$, and an inverse temperature $\beta$, there exists an explicit probability distribution $\nu$ on the set of percolation configurations $\{ 0,1\}^\Lambda$ such that, for any pair of vertices $x , y \in \Lambda$,
\begin{equation} \label{ineq.Wellsintro}
    \left\langle \cos(\theta_x - \theta_y) \right\rangle_{\mu_{\Lambda, \beta/2}} \leq \E_\nu \left[ \left\langle \cos(\theta_x - \theta_y) \right\rangle_{\mu_{\Lambda, \beta,r}} \right].
\end{equation}
The exact formula for the probability measure $\nu$ is obtained by multiplying the i.i.d. Bernoulli percolation measure of parameter $p = 1/2$ by the partition function $Z_{\Lambda, \beta, r}$ (see~\eqref{eq:18221905}) for the $XY$ model in a disordered environment (see~\eqref{measurenu.eq} for the explicit formula). Despite the apparent complexity of this definition, two observations can be deduced from the inequality~\eqref{ineq.Wellsintro}:
\begin{itemize}
    \item If the inverse temperature $\beta$ is selected larger than $2 \beta_{BKT}$ (resp. $2 \beta_c(d)$ for $d \geq 3$), then the term in the left-hand side of~\eqref{ineq.Wellsintro} decays polynomially fast with respect to the distance between $x$ and $y$ (resp. remains bounded away from $0$ for $d \geq 3$). 
    \item Using the explicit formula for the measure $\nu$ and a standard criterion for stochastic domination by product measures (see Lemma~\ref{Lemma1.5} below), one can prove that the measure $\nu$ is stochastically dominated by an i.i.d. Bernoulli percolation measure of probability $p_0 := p_0(\beta) < 1$ (but close to $1$). Combining this observation with the Ginibre correlation inequality (which implies that the function $r \mapsto \left\langle \cos(\theta_x - \theta_y) \right\rangle_{\mu_{\Lambda, \beta,r}}$ is increasing in $r$), we obtain
    \begin{equation*}
         \E_\nu \left[ \left\langle \cos(\theta_x - \theta_y) \right\rangle_{\mu_{\Lambda, \beta,r}} \right] \leq \E_{p_0} \left[ \left\langle \cos(\theta_x - \theta_y) \right\rangle_{\mu_{\Lambda, \beta,r}} \right].
    \end{equation*}
\end{itemize}
A combination of these two results implies the existence of the BKT phase transition for the $XY$ model on a high density Bernoulli site percolation cluster (resp. the order/disorder phase transition for $d \geq 3$). Optimizing over the inverse temperature yields Theorem~\ref{mainthm}. The details of the argument can be found in Section~\ref{Section4.1}. 

In order to extend the result to every supercritical percolation probability $p > p_{c , \mathrm{site}}(d),$ we rely on a renormalization argument. To this end, we make use of the notion of good boxes introduced by Penrose and Pisztora~\cite{penrose-pisztora-1996, pisztora-percolation} (see Section~\ref{sectionrenormalization} for formal definitions). These good boxes satisfy the following properties:
\begin{itemize}
    \item The probability of a box of sidelength $L$ to be good converges to $1$ as $L$ goes to infinity.
    \item Two intersecting good boxes of the same sidelength are connected by a cluster which lies inside their union.
\end{itemize}
On a heuristic level, it is thus possible to select an integer $L_0$ sufficiently large that the probability for a box of sidelength $L_0$ to be good is larger than the probability $p_0 := p_0(\beta)$ and to partition the set $\Zd$ into boxes of sidelength $L_0$ to reduce the problem of the BKT transition for the $XY$ model on a supercritical percolation cluster to the one on a high density percolation cluster. This latter problem can be solved using Wells' inequality. While the overall strategy is relatively straightforward, many technicalities need to be addressed along the way. An extended version of the graph $\Zd$ is introduced as well as a specific version of the $XY$ model on this new graph (see Definitions~\ref{def.Zdn} and~\ref{def.heterogeneousXY}). The existence of the BKT phase transition for this version of the $XY$ model on the extended graph then needs to be proved in a quantitative way (see Appendix~\ref{AppendixA}, where we rely on \cite{van2023elementary} for $d=2$ and \cite{garban2022continuous} for $d\geq 3$). The Ginibre correlation inequality is applied multiple times in the renormalization arguments (see Section~\ref{subsection5.2}). The details of the proof can be found in Section~\ref{section4.2} and Appendix~\ref{AppendixA}.
\smallskip

Finally, we mention another fruitful application of Wells' inequality: it was noticed by Dunlop~\cite{dunlop1985correlation} that it implies the existence of a BKT phase transition for the two-component $\Phi^4$ model on $\Z^2$. We provide the details of his observation in Appendix~\ref{a.phi4}, which may be of independent interest.

\subsection{Open questions.} 
We now list some open questions left by this work: 

\begin{OQ}
Show that the BKT  phase transition of the {\em Villain model} is also  robust to the presence of non-uniformly elliptic quenched disorder.
\end{OQ}

For most properties, if a result is shown for the $XY$ model, it can then be extended to the so-called Villain model by noticing that the Villain model is a limit of the $XY$ model when breaking each edge of the graph into $N\gg 1$ edges with high inverse temperature ($\beta_N=N \beta$).  This procedure works very well for example with Ginibre inequality as it was observed in \cite{frohlich1981kosterlitz}. See also \cite{aizenman2021depinning, dubedat2022random}.  
It appears that the same limiting procedure is less useful when coupled with Wells' inequality. Whence the above open question. (Notice that when  $d\geq 3$, Remark \ref{r.EIT} and the use of EIT property would allow us to conclude for the Villain model as well).

\begin{OQ}
Consider the faces of $\Z^2$ and glue two faces $f\sim f'$ together, independently of the other neighbouring faces with probability $1-p$, with $p>p_{c,\mathrm{edge}}(2)=\tfrac 1 2 $.  The resulting graph is the dual planar graph to the graph of the infinite cluster ($p>\tfrac 1 2$). 
Show that if the temperature is high enough, the {\em integer-valued GFF} (also called {\em discrete Gaussian Chain}) as well as the height function model dual to $XY$ a.s. delocalises on this graph with quenched disorder. 
\end{OQ}

Proving such a delocalisation in the presence of quenched disorder was our initial strategy to analyse the $XY$ model on supercritical percolation clusters. The results from \cite{lammers2023bijecting} may be useful to extract such a result from our Theorems \ref{mainthm} and \ref{proofKTsitegeneral}. (N.B. the work \cite{garban2023statistical} established quantitative delocalisation in the presence of a different kind of quenched disorder).

\begin{OQ}
On the same graph as in the above question, show that the {\em Coulomb gas} still has a low temperature phase without {\em Debye screening} despite the presence of quenched disorder. (See   \cite{garban2023quantitative}). 
\end{OQ}

\begin{OQ}
Prove that for all $\beta > \beta_{BKT}$, there exists $p=p(\beta)<1$ such that the BKT phase transition still holds on a supercritical cluster of intensity $p$ (site or edge). 
\end{OQ}
Indeed, Theorem \ref{mainthm} implies this property only for all  $\beta \geq 2 \beta_{BKT}$. The factor 2 looks artificial to us, but we did not manage to get rid of it while relying on Well's inequality.

\section{Definitions and preliminaries} \label{Section2}

\subsection{General definitions} \label{sectiongeneraldef}

We first introduce a few standard definitions and notations which will be used in this article. 

\subsubsection{Lattice}

We consider the lattice $\Zd$ with $d \geq 2$, denote by $E(\Zd)$ the set of undirected edges of $\Zd$, and by $\vec{E}(\Zd)$ the set of directed edges. Given a subset $\Lambda \subseteq \Zd$, we denote by $E(\Lambda)$ the set of undirected edges of $\Lambda$ (i.e., the edges of $\Zd$ whose endpoints are both in $\Lambda$). Similarly, we denote by $\vec{E}(\Lambda)$ the set of directed edges of $\Lambda$.

We say that two vertices $x , y \in \Zd$ are neighbours, and denote it by $x\sim y$, if $\{ x , y\} \in E(\Zd)$. We denote by $|\cdot|$ and $|\cdot |_1$ the Euclidean and $1$-norm on $\Zd$ (or $\Rd$), i.e., $|x| := (\sum_{i=1}^d |x_i|^2)^{1/2}$ and $| x |_1 = \sum_{i=1}^d |x_i|$ for any $x = (x_1 , \ldots, x_d) \in \Zd$.

A path from $x$ to $y$ is a finite collection of distinct vertices $ \{ x_1 , x_2 , \ldots, x_n \} \subseteq \Zd$ such that $x_1 = x$, $x_n = y$, and $x_{k+1} \sim x_k$, for any $k \in \{1 , \ldots, n - 1\}$. We call the integer $n$ the length of the path.

A box $ \Lambda \subseteq \Zd$ is a subset of the form $x + \{ - L , \ldots, L\}^d$ with $x \in \Zd$ and $L \in \N$. We refer to $x$ as the center of the box $\Lambda$ and to the integer $(2L+1)$ as its sidelength. We may scale boxes by defining, for any integer $k \in \N$, $k \Lambda := x + \{ - k L , \ldots, k L\}^d$ (note that this operation does not modify the center of the box and only acts on its sidelength). When the center of the box is the vertex $0 \in \Zd$, we will simply write $\Lambda_L := \{ - L , \ldots, L\}^d$. The cardinality of a box $\Lambda$ is denoted by $\left| \Lambda \right|$ (in particular $\left| \Lambda_L \right| = (2L+1)^d$). The (inner) boundary of subset $\Lambda \subseteq \Zd$ is the collection of the vertices $x \in \Lambda$ which are neighbours to a vertex outside $\Lambda$. It is denoted by $\partial \Lambda$.

\subsubsection{Spin configurations}

Given a set $\Lambda \subseteq \Zd$, a finitely supported function $m \in \Z^\Lambda$ and a spin configuration $\theta \in [0 , 2\pi)^\Lambda$, we denote by $m \theta := \sum_{x \in \Lambda} m_x \theta_x.$ For any pair of functions $A \in \Z^\Lambda$ and $r \in \R^\Lambda$ with $A$ finitely supported, we denote by $r_A = \prod_{x \in \Lambda} r_x^{A_x}$. For $x \in \Zd$, we denote by $\textbf{1}_x \in \Z^\Lambda$ the function equal to $1$ at $x$ and $0$ everywhere else. 

\subsubsection{Functions defined on subsets of $\Zd$} \label{sec2.1.3.def}

A function $f$ defined on the finite subsets of $\Zd$ and valued in $\R$ is said to be increasing if, for any pair of finite subsets $\Lambda \subseteq \Lambda' \subseteq \Zd$, one has $f(\Lambda) \leq f(\Lambda')$. A typical example of such function is the two-point function $f(\Lambda) = \left\langle \cos(\theta_x - \theta_y) \right\rangle_{\mu_{\Lambda, \beta}}$ where $\mu_{\Lambda, \beta}$ is the measure introduced in~\eqref{eq:XYmodel} (and we use the convention $f(\Lambda) = 0$ if either $x \notin \Lambda$ or $y \notin \Lambda$). 

We say that a function $f$ converges as $\Lambda \uparrow \Zd$ if the sequence $(f(\Lambda_n))_{n \in \N}$ converges as $n \to \infty$ for any sequence of finite sets $(\Lambda_n)_{n \in \N}$ such that $\Lambda_n \subseteq \Lambda_{n+1}$ and $\cup_{n \in \N} \Lambda_n = \Zd$ and if the limit does not depend on the specific sequence of sets chosen. In particular, any bounded and increasing function $f$ converges as $\Lambda \uparrow \Zd$, and the limit is equal to $\sup f(\Lambda),$ where the supremum is considered over all the finite subsets of $\Zd$.

\subsubsection{Probability spaces}

We denote by $\mathcal{P}(\R)$ the set of probability distributions on $(\R, \mathcal{B}(\R))$, where $\mathcal{B}(\R)$ is the Borel $\sigma$-algebra on $\R$. For $a \in \R$, we denote by $\delta_a$ the Dirac measure at $a$.

\subsection{The $XY$ model and the Ginibre inequality} \label{sectiondef.XY}

In this section, we introduce a general version of the $XY$ model (extending the spin space and allowing more general coupling constants) which will be used to state and prove Wells' inequality in Section~\ref{sectinoWellsineq}. We additionally state the Ginibre inequality in Theorem~\ref{theoremGinibreineq} below.

\begin{definition}[$XY$ model with general coupling constants] \label{generalXY}
Let $\Lambda \subseteq \Zd$ be a finite set, $\mathcal{M} \subseteq \Z^\Lambda$ be a finite collection of integer-valued functions, and let $\left\{ J(m) \, : \, m \in \mathcal{M} \right\} \subseteq [0, \infty)^{\mathcal{M} } $ be a collection of non-negative coupling constants. We define the finite-volume Gibbs measure on the space $[0 , 2\pi)^{\Lambda}$ according to the formula
    \begin{equation} \label{eq:17131905}
        \mu_{\Lambda, J} (d\theta) := \frac{1}{Z_{\Lambda, J}}\exp \left( \sum_{m \in \mathcal{M}} J(m ) \cos ( m \theta) \right) \prod_{x \in \Lambda} d \theta_x,
    \end{equation}
where $Z_{\Lambda, J}$ is the normalizing constant. We denote by $\left\langle \cdot \right\rangle_{\mu_{\Lambda, J}}$ the expectation with respect to the measure~$\mu_{\Lambda, J}$.
\end{definition}

\begin{definition}[$XY$ model with annealed disorder] \label{def.finite-volHamiltonian}
Let $\Lambda \subseteq \Zd$ be a finite set, $\mathcal{M} \in \Z^\Lambda$  and $\mathcal{A} \in \N^\Lambda$ two finite collections of integer-valued functions, and let $\left\{ J(m,A) \, : \, m \in \mathcal{M}, A \in \mathcal{A} \right\} \subseteq [0, \infty)^{\mathcal{M} \times \mathcal{A}} $ be a collection of non-negative coupling constants. We let $(\kappa_x)_{x \in \Lambda} \in (\mathcal{P}(\R))^\Lambda$ be a collection of compactly supported probability distributions on $[0 , \infty)$ and define the finite-volume Gibbs measure on the space $( [0, \infty) \times [0 , 2\pi])^{\Lambda}$ according to the formula
    \begin{equation} \label{eq:17141905}
        \mu_{\Lambda, J, \kappa} (dr d\theta) := \frac{1}{Z_{\Lambda, \kappa}}\exp \left( \sum_{\substack{m \in \mathcal{M} \\ A \in \mathcal{A}}} J(m , A) r_A \cos ( m \theta) \right) \prod_{x \in \Lambda} d \kappa_x(r_x) d \theta_x,
    \end{equation}
where $Z_{\Lambda, J, \kappa}$ is the normalizing constant. We denote by~$\left\langle \cdot \right\rangle_{\mu_{\Lambda, J, \kappa}}$ the expectation with respect to the measure $\mu_{\Lambda, J, \kappa}$.
\end{definition}

\begin{remark}
Let us make a few remarks about the previous definitions:
\begin{itemize}
    \item The definitions~\eqref{eq:17131905} and~\eqref{eq:17141905} are quite general (and in particular, the results of this article are established for the much more restrictive class of nearest-neighbour interactions). The reason motivating these definitions is twofold: first the Wells and Ginibre inequalities apply to this level of generality (without any modification of the argument), and second these definitions are sufficiently general to contain almost all the models which appear in the article (and in particular, the $XY$ model defined on the extended graph with heterogeneous temperatures introduced in Definition~\ref{def.heterogeneousXY} below). The only exception is the $XY$ model in a Nishimori disorder introduced in Appendix~\ref{AppendixA}.
    \item For the application to nearest-neighbour $XY$  model, we may relax the condition that $\kappa$ is compactly supported to the condition that $\kappa$ has sub-Gaussian tails, i.e. $\int_{\R_+} e^{\lambda t^2} d\kappa(t) <\infty$ for all $\lambda \geq 0$. This will be used in Appendix~\ref{a.phi4} to define and study a related model: the two-component $\Phi^4$ model.
    \item The $XY$ model with general coupling constants~\eqref{eq:17131905} can be obtained from the model~\eqref{eq:17141905} by considering the specific measure $\kappa = \delta_1$ and the coupling constants $J(m , A) = 0$ if $A \neq 0$ and $J(m,A) = J(m)$ if $A = 0$.
    \item The model~\eqref{eq:XYmodel} is obtained from the one introduced in~\eqref{generalXY} by considering the coupling constants $J(m) = \beta$ if $m = \indc_x - \indc_y$ for a pair of neighbouring vertices $x , y \in \Lambda$ and $J(m) = 0$ otherwise.
\end{itemize}
\end{remark}

\begin{figure} 
\begin{center}
\includegraphics[width=9cm]{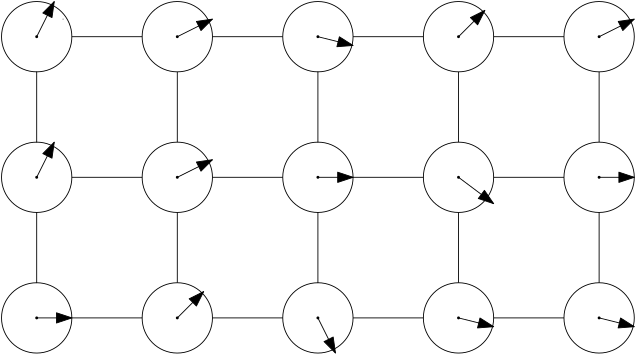}
\caption{A realization of the $XY$ model on $\Z^2$} \label{XYmodel.picture}
\end{center}
\end{figure}

We next state the Ginibre correlation inequality~\cite{ginibre1970general} which implies the monotonicity of the two-point function in the coupling constants. We specifically state two results from~\cite{ginibre1970general}, the first inequality~\eqref{theoremGinibreineq} is an important ingredient in the proof of the monotonicity~\eqref{eq:Ginibreineq}, and will also be used to establish the Wells' inequality in Section~\ref{sectinoWellsineq}.

\begin{theorem}[Ginibre inequality~\cite{ginibre1970general}] \label{theoremGinibreineq}
 One has the following inequalities:
 \begin{itemize}
     \item For any integer $k \in \mathbb{N}$, any collection of functions $m_1 , \ldots, m_k \in \mathbb{Z}^\Lambda$ and any sequence of plus and minus signs,
 \begin{equation*}
     \int \prod_{i = 1}^k \left( \cos ( m_i \theta) \pm  \cos ( m_i \theta') \right) \prod_{x \in \Lambda} d\theta_x d \theta'_x \geq 0.
 \end{equation*}
    \item In the setting of Definition~\ref{generalXY}, for any function $m \in \Z^{\Lambda}$,
\begin{equation} \label{eq:Ginibreineq}
    \left\langle \cos(m \theta) \right\rangle_{\mu_{\Lambda, J}} \geq 0 \hspace{5mm} \mbox{and} \hspace{5mm} \frac{\partial}{\partial J_{xy}} \left\langle \cos(m \theta) \right\rangle_{\mu_{\Lambda, J}} \geq 0.
\end{equation}
 \end{itemize}
\end{theorem}

\begin{remark}
    Consider a collection of coupling constants of the form $J := \left\{ J_{xy} \, : \, x, y \in \Zd, x \sim y  \right\}$. For any finite set $\Lambda \subseteq \Zd$, if we denote by $J_{\Lambda} := \left\{ J_{xy} \, : \, x, y \in \Lambda, x \sim y  \right\}$, then the inequality~\eqref{eq:Ginibreineq} implies that the function $\Lambda \mapsto \left\langle \cos(m \theta) \right\rangle_{\mu_{\Lambda, J_{\Lambda}}}$ is increasing, and thus converges as $\Lambda \uparrow \Zd$.
\end{remark}

\subsection{Percolation} \label{subsecpercolation}

In this section, we introduce some definitions and notations about percolation configurations and stochastic domination which are used in the proofs of the main theorems in Section~\ref{Section4.1} and Section~\ref{section4.2}.

 \begin{definition}[Percolation configuration on $\Zd$]
    A site percolation configuration on a subset $\Lambda \subseteq \Z^d$ is a function $r \in \{ 0 , 1\}^\Lambda$. We say that a site is open in the percolation configuration $r$ if $r_x = 1$ and closed if $r_x = 0$.
\end{definition}

\begin{definition}[Clusters and paths]
     Given a set $\Lambda \subseteq \Zd$ and a percolation configuration $r \in \{ 0 , 1\}^{\Lambda}$, we say that a subset $\mathcal{C} \subseteq \Lambda$ is a cluster for the percolation configuration $r$ if the set $\mathcal{C}$ is connected in $\Zd$ and if, for any $x \in \mathcal{C}$, $r_x = 1$. The diameter of a finite cluster $\mathcal{C}$ is the maximal distance between any pair of vertices in the cluster.

     Given a percolation configuration $r \in \{0 , 1\}^\Lambda$, an open path is a path whose vertices are all open in the percolation configuration $r$.
\end{definition}

\begin{definition}[Bernoulli site percolation measure] \label{DEf.bernoulliperc}
Given a set $\Lambda \subseteq \Zd$ and a probability $p \in [0 , 1]$, we equip the space of percolation configurations $\{ 0 , 1 \}^{\Lambda}$ with the $\sigma$-algebra generated by the projections, and denote by $\P_{\Lambda, p}$ the i.i.d. site percolation measure of probability $p$ on $\Lambda$. 

In the case when $\Lambda = \Zd$, the set of percolation configurations is $\{ 0 , 1 \}^{\Zd}$. For any subset $\Lambda \subseteq \Zd$, we denote by $\mathcal{F}(\Lambda)$ the smallest $\sigma$-algebra that makes the maps $r \mapsto r_x$ measurable for any $x \in \Lambda$. We write $\mathcal{F}:= \mathcal{F}(\Zd)$ and $\P_{p} := \P_{\Zd , p}$.
\end{definition}

We next introduce the notion of stochastic domination, and will make use of the following notation: given a subset $\Lambda \subseteq \Zd$ and a probability measure $\nu$ on $\{ 0 , 1 \}^{\Lambda}$, we denote by $\E_\nu$ the expectation with respect to $\nu$. We simply write $\E_p$ instead of $\E_{\P_{\Lambda,p}}$ or $\E_{\P_p}$.

\begin{definition}[Partial order, increasing functions and stochastic domination]
We introduce the three following definitions:
\begin{itemize}
\item[(i)] \textit{Partial order:} Given a subset $\Lambda \subseteq \Zd$, we define a partial order on the space $\{ 0 , 1 \}^{\Lambda}$ by writing, for any pair $r , r' \in \{ 0 , 1 \}^{\Lambda}$, $r \preceq r'$ if and only if, for all $x \in \Lambda, \, r_x \leq r_x'.$
\item[(ii)] \textit{Increasing function:} A function $f : \{ 0 , 1 \}^\Lambda \to \R$ is called increasing if $f(r) \leq f(r')$ for all pairs of configurations $r , r' \in \{ 0 , 1 \}^{\Lambda}$ with $r \preceq r'$.
\item[(iii)] \textit{Stochastic domination:} Given two probability measures $\nu , \nu'$ on $\{ 0 , 1 \}^{\Lambda}$, we say that $\nu$ stochastically dominates $\nu'$, and denote it by $\nu \preceq \nu'$, if $\E_\nu \left[  f \right] \leq \E_{\nu'} \left[  f \right]$ for any increasing function $f : \{ 0 , 1\}^{\Lambda} \to [0, \infty)$.
\end{itemize}
\end{definition}

\begin{remark}
    Let us make a few remarks about the previous definition:
\begin{itemize}
    \item For any $p , p' \in (0,1)$ with $p \leq p'$, the measure $\P_p$ is stochastically dominated by the measure $\P_{p'}$.
    \item The Ginibre inequality implies that, for any finite set $\Lambda \subset \Zd$, the function $r \in \{ 0 , 1 \}^\Lambda \mapsto \left\langle \cos (m \theta) \right\rangle_{\mu_{\Lambda, \beta,  r}}$ is increasing (using the Definition~\ref{eq:18221905}).
    \item We may extend all these definitions and properties to edge percolation instead of site percolation by replacing the set $\Lambda$ by the edges of $\Lambda$.
\end{itemize} 
\end{remark}

The following lemma provides a convenient criterion to prove stochastic domination. The statement and proof of the result can be found in~\cite[Lemma 1.1]{liggett1997domination} or in~\cite[Lemma 1.5]{duminil2019lectures} (specified to the case where the measure $\nu$ is the i.i.d. Bernoulli percolation measure $\P_{\Lambda,p}$ in the latter case). We say that a probability measure $\nu$ on $\{ 0,1\}^\Lambda$ is strictly positive if $\nu( r) > 0$ for any $r \in \{0,1\}^\Lambda$.

\begin{lemma}[Lemma 1.1 of~\cite{liggett1997domination} or Lemma 1.5 of~\cite{duminil2019lectures}] \label{Lemma1.5}
        Let $\Lambda \subseteq \Zd$ be a finite set and let $\nu$ be a strictly positive probability measure on $\left\{ 0 , 1 \right\}^{\Lambda}$. We assume that for any vertex $x \in \Lambda$ and any configuration $r_1 \in \Lambda \setminus \{ x \}$, one has the inequality
        \begin{equation*}
            \nu \left( r_x = 1 \, | \, r = r_1 ~ \mbox{in} ~ \Lambda \setminus \{ x \} \right) \leq p.
        \end{equation*}
        Then $\P_{\Lambda,p}$ stochastically dominates $\nu$.
\end{lemma}

We will need the following key result which is an important particular case of Theorem 0.0 in \cite{liggett1997domination}.
For any $k\geq 1$, we will say that a field of random variables $\{X_i\}_{i\in \Lambda}$ indexed by a subset $\Lambda \subseteq \Z^d$ is \textbf{$k$-dependent} if and only if for any subsets $A,B \subset \Lambda$ s.t. $\mathrm{dist}(A,B)>k$ $\sigma(X_i, i\in A)$ is independent of $\sigma(X_j, j\in B)$ (where $\mathrm{dist}$ is the graph distance on $\Z^d$).
 
\begin{theorem}[Theorem 0.0 from \cite{liggett1997domination}]\label{th.Liggett}
For any $d\geq 1$, any $k\geq 1$ and any $p\in (0,1)$, there exists 
$\hat p=\hat p(d,k, p) < 1$ such that the following hold. For any subset $\Lambda \subseteq \Zd$, if $\{X_i\}_{i\in \Lambda}$ is a $k$-dependent field of Bernoulli variables satisfying $\Pb{X_i=1}\geq \hat p$, for each $ i \in \Lambda$, then if $\{Y_i\}_{i\in \Lambda}$ is an i.i.d field of Bernoulli variables with parameter $p$, one has 
\begin{align*}\label{}
\{X_i\}_{i\in \Lambda}  \text{ Stoch. dominates } \{Y_i\}_{i\in \Lambda}\,.
\end{align*}
\end{theorem}

\subsection{Renormalization for the supercritical Bernoulli percolation cluster} \label{sectionrenormalization}

Theorem~\ref{proofKTsitegeneral} is obtained by combining the results of Theorem~\ref{mainthm} and Corollary~\ref{proofKTbonds} (for highly supercritical percolation) with renormalization arguments for the supercritical percolation cluster. Specifically, we rely on the results and techniques developed by Penrose and Pisztora~\cite{penrose-pisztora-1996, pisztora-percolation}, and, following these articles, we first introduce a notion of good boxes. The definition is stated in two steps: we first introduce the notion of pre-good box in Definition~\ref{pregoodbox}, and enhance it in the definition of good box in Definition~\ref{def.goodbox} below.

\begin{definition}[Pre-good box] \label{pregoodbox}
Let $d \geq 2$. For any box $\Lambda \subseteq \mathbb{Z}^d$ and any site percolation $r \in \{0 , 1\}^\Lambda$, we say that the box $\Lambda$ is pre-good if and only if:
\begin{itemize}
    \item There exists a unique cluster in the percolation configuration $r$ which touches the $2d$ faces of the box $\Lambda$. We denote this cluster by $\mathcal{C}(\Lambda).$
    \item The diameter of all the other clusters is smaller than $\mathrm{diam}(\Lambda) / 100.$
\end{itemize}
\end{definition}

\begin{figure}[!ht] 
\centering
\includegraphics[width=8cm]{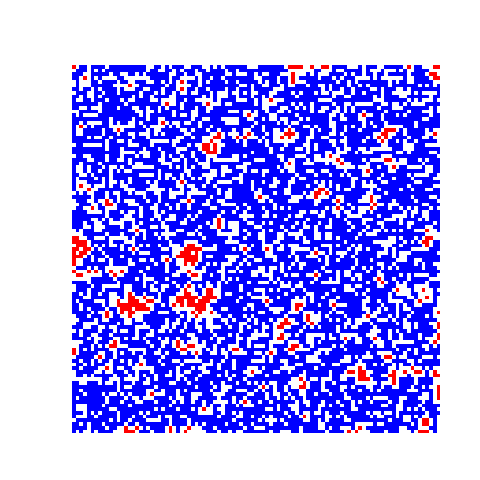}
\includegraphics[width=12cm]{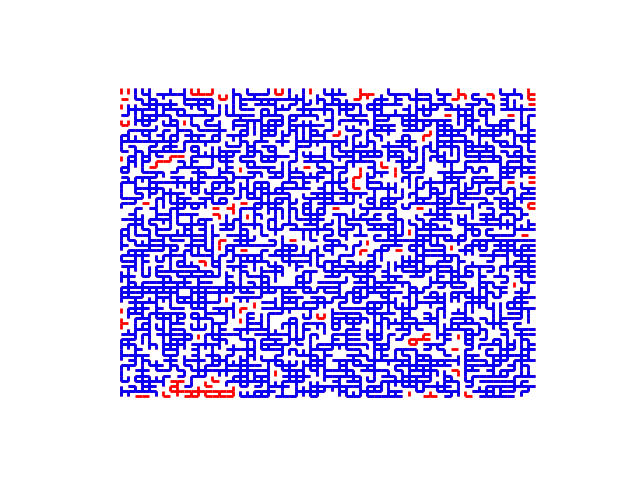}
\caption{An example of a pre-good box for site (top) and edge (bottom) percolation. In both cases, the cluster $\mathcal{C}(\Lambda)$ is drawn in blue. \label{goodboxes.caption}}
\end{figure}

Figure~\ref{goodboxes.caption} gives a graphical representation of a pre-good box. We next record the result of~\cite{penrose-pisztora-1996, pisztora-percolation} which asserts that, for a supercritical site percolation, the probability of a box to be pre-good is exponentially close to $1$ (in the sidelength of the box).

\begin{proposition}[\cite{penrose-pisztora-1996, pisztora-percolation}] \label{proppre-good}
For any $p > p_{c, \mathrm{site}}(d)$, there exist a constant $C := C(p , d) < \infty$ and an exponent $c := c(d , p) > 0$ such that for any box $\Lambda \subseteq \mathbb{Z}^d$ of sidelength $L$,
\begin{equation*}
    \mathbb{P}_p \left[\,  \Lambda \mbox{ is pre-good} \,  \right] \geq 1 - C \exp (-c L).
\end{equation*}
\end{proposition}

In order to implement renormalization arguments, it is convenient that good boxes satisfy the following property: if two boxes $\Lambda \subseteq \mathbb{Z}^d$ and $\Lambda' \subseteq \mathbb{Z}^d$ have the same sidelength, are both good and next to each other (i.e., the distance between their centers is equal to their common sidelength), then the two clusters $\mathcal{C}(\Lambda)$ and $\mathcal{C}(\Lambda')$ are connected. This property is not necessarily satisfied with the definition of pre-good boxes stated above. We thus enhance this definition as follows.

\begin{definition}[Good box] \label{def.goodbox}
A box $\Lambda \subseteq \mathbb{Z}^d$ of sidelength $L$ is said to be \emph{good} if and only if:
\begin{itemize}
    \item[(i)] The box $\Lambda$ is pre-good;
    \item[(ii)] Any box $\Lambda'$ whose sidelength is between $L/10$ and $L/2$ and such that $\Lambda \cap \Lambda' \neq \emptyset$ is also pre-good.
\end{itemize}
\end{definition}

\begin{remark} \label{remark2.17}
Let us make a few remarks about the previous definition:
\begin{itemize}
\item This definition ensures that, if two good boxes $\Lambda$ and $\Lambda'$ are next to each other and have the same sidelength, then the clusters $\mathcal{C}(\Lambda)$ and $\mathcal{C}(\Lambda')$ are connected.
\item For any box $\Lambda$ of sidelength $L$, the number of boxes of $\Lambda'$ which satisfy Condition (ii) is polynomial in $L$. Applying Proposition~\ref{proppre-good} and a union bound, we obtain the estimate, for any $p > p_{c, \mathrm{site}}(d)$,
\begin{equation} \label{eqpdecaygoodbox}
    \mathbb{P}_p \left[  \Lambda \mbox{ is good}  \right] \geq 1 - C \exp (-c L).
\end{equation}
\item The event ``$\Lambda$ is a good box" is not $\mathcal{F}(\Lambda)$-measurable (contrary to the event ``$\Lambda$ is a pre-good box"), but it is $\mathcal{F}(2\Lambda)$-measurable.
\item All the definitions and results stated in this section applies to the Bernoulli edge percolation. The only difference is that the site percolation threshold $p_{c, \mathrm{site}}(d)$ in Proposition~\ref{proppre-good} has to be replaced by the edge percolation threshold $p_{c, \mathrm{edge}}(d)$.
\end{itemize}
\end{remark}

\subsection{The extended lattice $\mathbb{Z}^d_n$}

In order to implement a renormalization argument to prove
Theorem~\ref{proofKTsitegeneral}, 
we need to define an extension of the lattice $\Zd$ (essentially obtained by adding vertices on each edge, see Figure~\ref{ZdetZdn}). We also introduce a definition of a percolation configuration and an $XY$ model on the extended lattice, and finally state (and prove in Appendix~\ref{AppendixA}), that the $XY$ model defined on the extended lattice exhibits the same phase transitions as the $XY$ model on $\Zd$.

\subsubsection{General definition} \label{sec:generaldefinition}

\begin{definition}[Extended lattice $\mathbb{Z}^d_n$] \label{def.Zdn}
Given an integer $n \in \mathbb{N}$, we define the extended lattice $\mathbb{Z}^d_n$ to be the graph $\Zd$ to which $n$ vertices are added on each edge. This graph is represented in Figure~\ref{ZdetZdn}.
\end{definition}

We extend all the definitions of Section~\ref{sectiongeneraldef} to the extended lattice $\Zd_n$. In particular, we define a box of $\Zd_n$ to be a box of $\Zd$ to which $n$ vertices have been added on each edge. Figure~\ref{ZdetZdn} shows a box on~$\Z^2$ next to a box on $\Z^2_3$. We denote by $\Lambda^n_R \subseteq \Zd_n$ the box $\Lambda_R \subseteq \Zd$ to which $n$ vertices have been added on each edge.

\begin{figure} 
\begin{center}
\includegraphics[width=10cm]{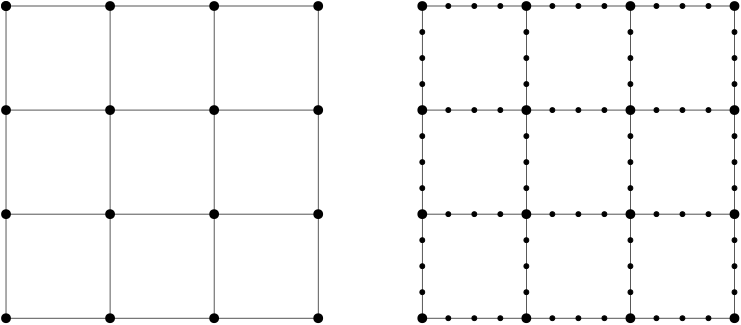}
\caption{The graph $\Zd$ and the extended graph $\Zd_n$ (with $n = 3$).} \label{ZdetZdn}
\end{center}
\end{figure}

\begin{remark} \label{remark2.7}
    Let us make three remarks about the previous definition:
    \begin{itemize}
     \item In the rest of this article, we will identify the vertices of $\Z^d$ with the corresponding vertices of the extended lattice $\Z^d_n$. However, we emphasize that two vertices $x , y \in \mathbb{Z}^d$ which are neighbours on $\Zd$ are \emph{not} neighbour on $\Z^d_n$.
     \item We write $x \sim_n y$ to refer to pair of vertices which are neighbours in $\Zd_n$.
     \item There is no isomorphism of graphs between $\Z^d$ and $\mathbb{Z}^d_n$. However, there exists an isomorphism of graphs between $\mathbb{Z}_n^d$ and the following subgraph of $\mathbb{Z}^d$ 
    \begin{equation} \label{subgraphZd}
        \left\{ (x_1 , \ldots, x_d) \in \mathbb{Z}^d \, : \, \exists i \in \{ 1 , \ldots, d\}, \, n \, | \, x_i \right\} \subseteq \mathbb{Z}^d.
    \end{equation}
     \end{itemize}
\end{remark}

\begin{figure} 
\begin{center}
\includegraphics[width=10cm]{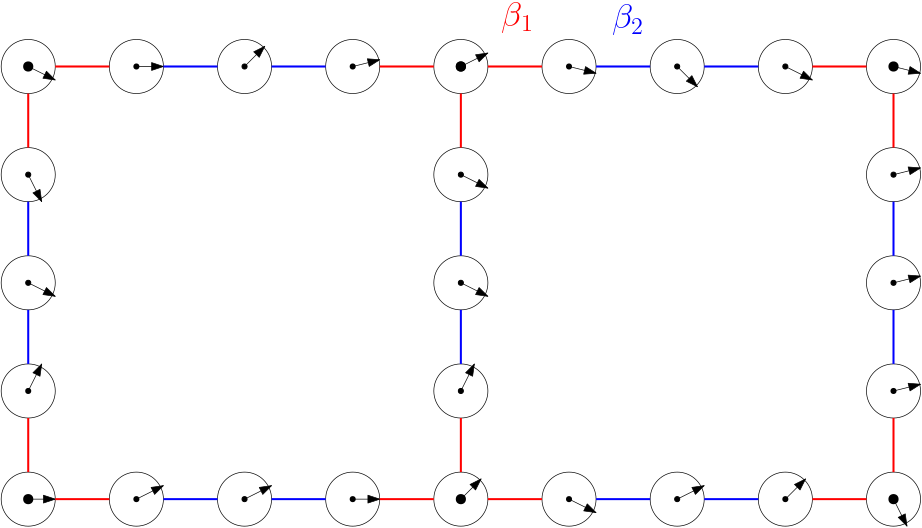}
\caption{A realization of the $XY$ model on the extended lattice $\Zd_n$ (with $d = 2$ and $n = 3$) with heterogeneous temperatures. The two inverse temperatures $\beta_1$ and $\beta_2$ are displayed in red and blue respectively} \label{XYextended}
\end{center}
\end{figure}

\subsubsection{Phase transitions for the $XY$ model on the extended lattice $\Zd_n$}
In this section, we prove the existence of the BKT transition in $d = 2$, and long-range order in dimension $d \geq 3$ for a version of the $XY$ model defined on the extended lattice $\Z^d_n$. Specifically, the model we will use involves two temperatures and is introduced below.

\begin{definition}[$XY$ model on the extended lattice $\Z^d_n$ with heterogeneous temperatures] \label{def.heterogeneousXY}
    Fix an integer $n \in \N$. For any finite set $\Lambda \subseteq \Zd_n$ and any pair of inverse temperatures $\beta_1 , \beta_2 > 0$, we define the Hamiltonian
    \begin{equation*}
    H_{\Lambda, \beta_1 , \beta_2}(\theta) := - \beta_1 \sum_{\substack{x \sim_n y \\ x,y \in \Lambda \\ \{ x , y \} \cap \Zd \neq \emptyset}} \cos (\theta_x - \theta_y) - \beta_2 \sum_{\substack{x \sim_n y \\ x,y \in \Lambda \\ \{ x , y \} \cap \Zd = \emptyset}} \cos (\theta_x - \theta_y),
\end{equation*}
as well as the probability measure 
\begin{equation*}
        \mu_{\Lambda, \beta_1 , \beta_2}(d\theta) := \frac{1}{Z_{\Lambda, \beta_1 , \beta_2 }}\exp \left( - H_{\Lambda, \beta_1 , \beta_2}(\theta)  \right) \prod_{x \in \Lambda} d \theta_x.
\end{equation*}
\end{definition}

\begin{remark}

Let us make two remarks about the previous definition:
\begin{itemize}
    \item This specific class of models fits into the general framework of Definition~\ref{generalXY}. This observation implies that the Ginibre correlation inequality applies to this model. As a consequence, the two-point function $\left\langle \cos(\theta_x - \theta_y) \right\rangle_{\mu_{\Lambda, \beta_1 , \beta_2}}$ is increasing in the domain $\Lambda$ and in the inverse temperatures $\beta_1, \beta_2$.
    \item As it was the case for the model~\eqref{eq:XYmodel}, compactness arguments and the Ginibre inequality imply that the sequence of measures $\mu_{\Lambda, \beta_1 , \beta_2}$ converges as $\Lambda \uparrow \Zd_n$ to an infinite-volume measure which will be denoted by $\mu_{\beta_1 , \beta_2}.$ This implies that the two-point function $\left\langle \cos(\theta_x - \theta_y) \right\rangle_{\mu_{\Lambda, \beta_1 , \beta_2}}$ converges as $\Lambda \uparrow \Zd_n$ to $\left\langle \cos(\theta_x - \theta_y) \right\rangle_{\mu_{\beta_1 , \beta_2}}$.
\end{itemize}
\end{remark}

The main result of this section is stated below and its proof, which is an adaptation of the proofs of the proofs of~\cite{van2023elementary, garban2022continuous} on $\Z^d$, can be found in Appendix~\ref{AppendixA}.

\begin{proposition}[Phase transitions for the $XY$ model on the extended lattice $\Zd_n$] \label{prop.phasetransitionextended}
The following hold true:
\begin{itemize}
    \item In dimension $d = 2$, there exists an inverse temperature $\beta_{1, BKT} < \infty$ such that, for every $n \in \N$, there exists an inverse temperature $\beta_{2, BKT}(n) < \infty$ such that, for any $\beta_1 \geq \beta_{1, BKT}$ and any $\beta_2 \geq \beta_{2, BKT}(n)$,
    \begin{equation*}
        \left\langle \cos(\theta_0 - \theta_x) \right\rangle_{\mu_{\beta_1 , \beta_2 }} ~\mbox{decays polynomially fast as } |x| \to \infty. 
    \end{equation*}
    \item In dimension $d \geq 3$, there exists an inverse temperature $\beta_{1, c}(d) < \infty$ such that, for every $n \in \N$, there exists an inverse temperature $\beta_{2, c}(d , n) < \infty$ such that, for any $\beta_1 \geq \beta_{1, c}(d)$ and any $\beta_2 \geq \beta_{2, c}(n , d)$,
    \begin{equation*}
        \left\langle \cos(\theta_0 - \theta_x) \right\rangle_{\mu_{\beta_1 , \beta_2 }} ~\mbox{remains bounded away from } 0 \mbox{ as } |x| \to \infty.
    \end{equation*}
\end{itemize}
\end{proposition}

\begin{remark} \label{remark2.24}
Let us make two remarks about the previous proposition:
\begin{itemize}
    \item An important feature of the previous statement is that the temperature $\beta_1$ can be chosen independently of the integer $n$. This observation plays a fundamental role in the proof of Theorem~\ref{proofKTsitegeneral}.
    \item The following lower bounds can be deduced from the argument: in dimension $d = 2$, there exists $c := c(n) > 0$ such that
\begin{equation*}
    \langle \cos \left( \theta_0 - \theta_x \right) \rangle_{\mu_{\beta_1, \beta_2}} \geq \frac{c}{|x|}.
\end{equation*}
    In dimension $d \geq 3,$ there exists a constant $C := C(d) < \infty$ such that
    \begin{equation*}
    \langle \cos \left( \theta_0 - \theta_x \right) \rangle_{\mu_{\beta_1, \beta_2}} \geq 1 - C \sqrt{ \frac{1}{\beta_1} + \frac{n}{\beta_2}}.
\end{equation*}
\end{itemize}
\end{remark}

\subsubsection{Percolation on the extended lattice $\Zd_n$}

In this section, we extend the definitions introduced in Section~\ref{subsecpercolation} to the extended lattice $\Zd_n$.

\begin{definition}[Site percolation on the extended lattice $\Z^d_n$]
A site percolation configuration on a subset $\Lambda \subseteq \Z^d_n$ is a function $r \in \{ 0 , 1\}^\Lambda$. We say that a site is open in the percolation configuration $r$ if $r_x = 1$ and closed if $r_x = 0$.
\end{definition}

\begin{figure}
\begin{center}
\includegraphics[width=5cm]{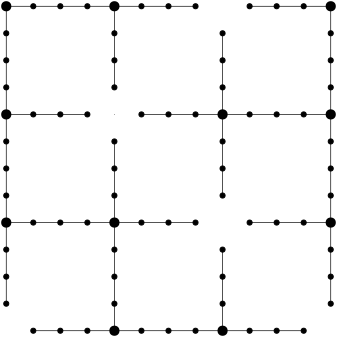}
\caption{A realization of a site percolation on the extended lattice $\Zd_n$ (with $d = 2$ and $n = 3$) sampled according to the measure $\P_p$. Only the sites of $\Zd \subseteq \Zd_n$ are allowed to be closed.}
\end{center}
\end{figure}

\begin{definition}[Percolation measures $\P_{\Lambda,p}$ and $\P_{p}$ on the extended lattice $\Zd_n$] \label{DEf.bernoullipercextend}
Given a subset $\Lambda \subseteq \Zd_n$, denote by $\P_{\Lambda,p}$ the independent probability measure on  $\left\{ 0 ,1  \right\}^\Lambda$ characterized by the identities
\begin{equation*}
    \P_{\Lambda,p} ( \{ r_x = 1 \} ) = p ~\mbox{for}~ x \in \Lambda \cap \Zd ~~\mbox{and}~~ \P_{\Lambda,p} ( \{ r_x = 1 \} ) = 1  ~\mbox{for}~ x \in \Lambda \setminus \Zd.
\end{equation*}
We simply denote by $\P_p := \P_{\Zd_n , p}$ and by $\mathbb{E}_p$ the expectation with respect to the measures $\P_{p}$ or $\P_{\Lambda,p}.$
\end{definition}

\begin{remark} Let us make two remarks about the previous definition:
    \begin{itemize}
    \item This notation is consistent with the one introduced in Section~\ref{subsecpercolation}, as, for each subset $\Lambda \subseteq \Zd_n$, the marginal of the measure $\P_{\Lambda,p}$ on the set $\{0,1\}^{\Lambda \cap \Zd}$ is equal to the i.i.d. Bernoulli site percolation measure of probability $p$.
    \item With this definition, all the sites of $\Zd_n \setminus \Zd$ are open almost surely. As a consequence, percolation configurations sampled according to $\P_p$ on the lattice $\Zd$ and on the extended lattice $\Zd_n$ have essentially the same properties.
    \end{itemize}
\end{remark}

We also state the version of Lemma~\ref{Lemma1.5} for the percolation measure $\P_{\Lambda,p}$ on the extended lattice $\Zd_n$ which will be used in the proof of Theorem~\ref{proofKTsitegeneral}.

\begin{lemma} \label{Lemma1.5ext}
        Let $\Lambda \subseteq \Zd_n$ be a finite set and let $\nu \in \mathcal{P} ( \left\{ 0 , 1 \right\}^{\Lambda})$ be probability measure satisfying the three following assumptions:
        \begin{enumerate}
        \item[(i)] For any $x \in \Zd_n \setminus \Zd,$ $\nu \left( r_x = 1 \right) = 1$;
        \item[(ii)] The marginal of $\nu$ on the set $\{0 , 1\}^{\Lambda \cap \Zd}$ is strictly positive;
        \item[(iii)] For any vertex $x \in \Lambda \cap \Zd$ and any configuration $r_1 \in \Lambda \setminus \{ x \}$ which is equal to $1$ on $\Lambda\setminus \Zd$, one has the inequality
        \begin{equation*}
            \nu \left( r_x = 1 \, | \, r = r_1 ~ \mbox{in} ~ \Lambda \setminus \{ x \} \right) \leq p.
        \end{equation*}
        \end{enumerate}
        Then $\P_{\Lambda,p}$ stochastically dominates $\nu$.
\end{lemma}

\begin{proof}
    The assumptions on the measure $\nu$ implies that one can apply Lemma~\ref{Lemma1.5} to the marginal of $\nu$ on the set $\{0,1\}^{\Lambda \cap \Zd}$. This marginal is thus stochastically dominated by the measure $\P_{\Lambda \cap \Zd, p}$. The definition of $\P_{\Lambda,p}$ (and in particular, the assumption that $\P_{\Lambda,p} ( \{ r_x = 1 \} ) = 1$ for $x \in \Lambda \setminus \Zd$) implies that the measure $\nu$ is stochastically dominated by $\P_{\Lambda,p}$.
\end{proof}

\section{Wells' inequality} \label{sectinoWellsineq}

This section is devoted to the proof of Wells' inequality for the $XY$ model. The inequality was originally established by Wells~\cite{Wellsthesis}, the proof of the result can be found in~\cite[Appendix]{bricmont1981periodic} in the case of the spin-1/2 Ising model. The proof below is an adaptation of the argument of~\cite{bricmont1981periodic} to the case of the $XY$ model (as mentioned in~\cite[Appendix]{bricmont1981periodic}). In the following proposition, we use the notation introduced in Definition~\ref{def.finite-volHamiltonian}.

\begin{proposition}[Wells' inequality for rotator models~\cite{Wellsthesis, bricmont1981periodic}] \label{prop.Wells}
    Let $\Lambda \subseteq \Zd$ be a finite set, $(\kappa_x)_{x \in \Lambda} \in \mathcal{P}(\R)^\Lambda$ be a collection of probability measures compactly supported in $[0, \infty)$ and $(a_x)_{x \in \Lambda} \in [0, \infty)^{\Lambda}$ be a collection of real numbers satisfying, for any pair of integers $m , n \in \N$,
    \begin{equation} \label{eq:143909}
        \int_{0}^\infty  \left( r + a_x \right)^m  \left( r - a_x \right)^n  d \kappa_x(r) \geq 0.      
    \end{equation}
    Then, for any pair of functions $A , m \in \mathbb{Z}^\Lambda$,
    \begin{equation} \label{eq:Wellsineq}
         \left\langle r_A \cos ( m \theta ) \right\rangle_{\mu_{\Lambda, J, \delta_{a}}} \leq \left\langle r_A \cos ( m \theta ) \right\rangle_{\mu_{\Lambda, J, \kappa}},
    \end{equation}
    where the notation $\delta_a$ on the left-hand side refers to the collection of measures $\left(\delta_{a_x}\right)_{x \in \Lambda}.$
\end{proposition}

\begin{remark} 
    Let us make a few remark about Wells' inequality:
    \begin{itemize}
        \item For the application to nearest-neighbour $XY$  model, we may relax the condition that $\kappa$ is compactly supported to the condition that $\kappa$ has sub-Gaussian tails, i.e. $\int_{\R_+} e^{\lambda t^2} d\kappa_x(t) <\infty$ for all $\lambda \geq 0$, $x \in \Lambda$. This observation has an interesting consequence (due to~\cite{dunlop1985correlation}) and it can be used to show the existence of a BKT phase transition for the two-component $\Phi^4$ model on $\Z^2$. We refer to Appendix~\ref{a.phi4} for more details in this direction.
    \end{itemize} 
\end{remark}

\begin{proof}
    We consider the measure $\mu_{\Lambda, J, \kappa}$ introduced in Definition~\ref{def.finite-volHamiltonian}, and, to simplify the notation, denote by
    \begin{equation*}
       H_{\Lambda, J}(r ,\theta) :=  -\sum_{\substack{m \in \mathcal{M} \\ A \in \mathcal{A}}} J(m , A) r_A \cos ( m \theta).
    \end{equation*}
    We follow the arguments of~\cite{Wellsthesis, bricmont1981periodic}, use duplicated variables and proceed by successive reductions. It is enough in order to prove the inequality~\eqref{eq:Wellsineq} to show that
    \begin{equation} \label{eq:duplicatedvariables}
        \int \left( r_A \cos ( m \theta ) -  r_A' \cos ( m \theta') \right) \exp \left( - H_{\Lambda, J}(r , \theta) - H_{\Lambda, J}(r' , \theta') \right) \prod_{x \in \Lambda} d\theta_x d\theta'_x d \kappa_x(r_x) d \delta_{a_x} (r'_x) \geq 0. 
    \end{equation}
    By expanding the exponential, the inequality~\eqref{eq:duplicatedvariables} can be obtained as a consequence of the following result: for any integer $k \in \N$ and any collections of functions $A_1, \ldots, A_k \in \Z^{\Lambda}$, $m_1, \ldots, m_k \in \Z^{\Lambda}$, and any sequence of plus and minus signs 
    \begin{equation} \label{ineq:1747}
        \int     \prod_{i = 1}^k \left( r_{A_i} \cos ( m_i \theta ) \pm  r_{A_i}' \cos ( m_i \theta') \right)  \prod_{x \in \Lambda} d\theta_x d\theta'_x d \kappa_x(r_x) d \delta_{a_x} (r'_x) \geq 0.
    \end{equation}       
    We next use the identities
    \begin{multline*} 
        \left( r_{A_i} \cos ( m_i \theta ) -  r_{A_i}' \cos ( m_i \theta') \right) \\ = \frac{1}{2} \left( r_{A_i} -  r_{A_i}' \right) \left(  \cos ( m_i \theta ) + \cos ( m_i \theta') \right)  + \frac{1}{2} \left( r_{A_i} +  r_{A_i}' \right) \left(  \cos ( m_i \theta )  - \cos ( m_i \theta') \right)
    \end{multline*}
    and
    \begin{multline*} 
        \left( r_{A_i} \cos ( m_i \theta ) +  r_{A_i}' \cos ( m_i \theta') \right) \\ = \frac{1}{2} \left( r_{A_i}  +  r_{A_i}' \right) \left(  \cos ( m_i \theta ) + \cos ( m_i \theta') \right)  + \frac{1}{2} \left( r_{A_i}  -  r_{A_i}' \right) \left(  \cos ( m_i \theta )  - \cos ( m_i \theta') \right).
    \end{multline*}
    From the two previous displays and the Ginibre inequality stated in Theorem~\ref{theoremGinibreineq}, we see that the inequality~\eqref{ineq:1747} can be obtained as a consequence of the following inequality: for any collection of functions $A_1, \ldots, A_k \in \Z^{\Lambda}$ and any sequence of plus and minus signs,
    \begin{equation*} 
        \int \prod_{i = 1}^k  \left( r_{A_i} \pm r_{A_i}' \right) \prod_{x \in \Lambda}  d \kappa_x(r_x) d \delta_{a_x} (r'_x) \geq 0.
    \end{equation*}
    Using (a second time) the identities, for any $a , b , c  ,d \in \R$,
    \begin{equation} \label{elementaryidentity}
        ab + cd = \frac{1}{2} (a + c ) (b + d) + \frac{1}{2} (a - c) (b - d) ~~\mbox{and}~~ab - cd = \frac{1}{2} (a - c ) (b + d) + \frac{1}{2} (a + c) (b - d),
    \end{equation}
    we see that the inequality is implied by the assumption~\eqref{eq:143909}.
    \end{proof}

    The following statement shows that the assumption~\eqref{eq:143909} is satisfied for a broad class of measures.

    \begin{proposition}[\cite{Wellsthesis, bricmont1981periodic}]
         For any probability measure $\kappa \in \mathcal{P}(\R)$ which is not equal to the Dirac $\delta_0$ and whose moments of all order are finite, there exists a constant $a(\kappa) > 0$ such that the assumption~\eqref{eq:143909} is satisfied. In the specific case $\kappa_{\bar p} := (1-\bar p) \delta_0 + \bar p \delta_1$ for $\bar p \in (0,1]$, we may choose the value $a(\kappa_{\bar p})= \min \left( \bar p , \frac 12 \right)$.
    \end{proposition}

    \begin{proof}
    Without loss of generality, we may assume that $n$ is odd. Since the probability measure $\kappa$ is not equal to the Dirac mass $\delta_0$, then there exist $\delta, \varepsilon \in (0,1)$ such that
    \begin{equation*}
        \kappa \left( [\delta , \infty) \right) > \varepsilon.
    \end{equation*}
    Selecting $a = a(\kappa) := \varepsilon \delta / (\varepsilon + 1) $, and noting that the function $r \mapsto (r+a)^m (r-a)^n$ is always larger than $- 2^m a^{m+n}$ and that it is increasing on the interval $[\delta , \infty)$ (since $a\leq\delta$), we can write
    \begin{align} \label{eq:19091905}
        \int_{0}^\infty  \left( r + a \right)^m  \left( r - a \right)^n  d \kappa(r) 
        & = \int_{0}^{\delta} \left( r + a \right)^m  \left( r - a \right)^n d \kappa(r)  + \int_\delta^\infty \left( r + a \right)^m  \left( r - a \right)^n d \kappa(r) \\
        & \geq -  2^m a^{m+n} + (\delta + a)^m (\delta - a)^n  \kappa \left( [\delta , \infty] \right)  \notag \\
        & \geq - 2^m a^{m+n} + (\delta + a)^m (\delta - a)^n  \varepsilon. \notag
    \end{align}
    Using the definition of $\delta$, we see that $\delta - a = \frac{a}{\varepsilon}$ and $\delta + a \geq 2 a$. We thus obtain (using the assumption that $n$ is odd)
    \begin{equation*}
        \int_{0}^\infty  \left( r + a \right)^m  \left( r - a \right)^n  d \kappa(r) \geq - 2^m a^{m+n} + 2^m a^{m+n}  \varepsilon^{1-n} \geq 0.
    \end{equation*}
    
    We finally show that, in the case of the measure $\kappa=(1-\bar p) \delta_0 + \bar p \delta_1$ with $\bar p \in (0,1]$, we may choose the value
 \begin{align}\label{e.KEY}
a = \min \left( \bar p , \frac 12 \right).
\end{align}
 The computation~\eqref{eq:19091905} becomes
    \begin{equation*}
        \int_{0}^\infty  \left( r + a \right)^m  \left( r - a \right)^n  d \kappa(r) = (-1)^n (1-\bar p) a^{m+n} + \bar p (1 + a)^m (1 - a)^n.
    \end{equation*}
    It is thus sufficient to verify that, with the value $a = \min \left( \bar p , \frac 12 \right)$, for any pair of integers $m , n \in \N$ with $n$ odd,
    \begin{equation*} 
         \bar p (1 + a)^m (1 - a)^n \geq  (1-\bar p) a^{m+n}.
    \end{equation*}
    This is a consequence of the inequalities $1 + a \geq a$, $1 - a \geq a$ and $\bar p(1-a) \geq (1-\bar p) a$.
\end{proof}

\section{Phase transitions for the $XY$ model on a high density Bernoulli site percolation cluster} \label{Section4.1}

This section is devoted to the proof of Theorem~\ref{mainthm}.
It is structured as follows. In Section~\ref{subsection4.1}, we make use of Wells' inequality, the stochastic domination criterion stated in Lemma~\ref{Lemma1.5} and Ginibre inequality to prove an inequality involving the two-point function of the $XY$ model on a highly supercritical percolation cluster and the two-point function of the $XY$ model with annealed disorder (see Proposition~\ref{prop4.1}). Section~\ref{Section5} contains the proof of Theorem~\ref{mainthm}
(making use of Proposition~\ref{prop4.1}).

\subsection{Domination for the $XY$ model on a percolation cluster} \label{subsection4.1}

In this section, we recall the definitions of the measures $\mu_{\Lambda, \beta, r}$ introduced in~\eqref{eq:18221905} as well as the ones of the Bernoulli site percolation measures $\P_p$ and $\P_{\Lambda,p}$ introduced in Section~\ref{subsecpercolation}. For any $\bar p\in (0,1]$, we consider the collection of probability measures $(\kappa_x)_{x \in \Z^d}$ defined by $\kappa_x  = (1-\bar p) \delta_0 + \bar p \delta_1$ for all $x \in \Zd$ and use Wells' inequality to relate the two-point function (or more generally the expectation of the random variable $\cos (m \theta)$ with $m \in \Z^\Lambda$) under the measure 
\begin{equation*}
        \mu_{\Lambda, \beta, \kappa(\bar p)}(dr d\theta) := \frac{1}{Z_{\Lambda, \beta, \kappa(\bar p)}}\exp \left( \beta \sum_{x \sim y} r_x r_y \cos (\theta_x - \theta_y) \right) \P_{\Lambda,\bar p}( d r ) \prod_{x \in \Lambda}d \theta_x,
\end{equation*}
to the one of the $XY$ model in a random environment given by a Bernoulli site percolation. Note that, in the display above, we have used the identity $\P_{\Lambda,\bar p} = \prod_{x \in \Lambda} \kappa_x$, which holds because all the measures $\kappa_x$ are equal to the measure $(1-\bar p) \delta_0 + \bar p \delta_1$. 

\begin{proposition} \label{prop4.1}
Fix $\beta \in (0 , \infty)$, $\bar p \in (0 , 1)$ and set $p_0 =p_0(\bar p,\beta) := \frac{\bar p}{\bar p + (1-\bar p)  \exp \left( - 2d \beta \right)} \in (0 , 1)$. Then, for any finite subset $\Lambda \subseteq \Zd$ and any function $m \in \Z^\Lambda$,
\begin{equation*}
       \left\langle \cos(m \theta) \right\rangle_{\mu_{\Lambda, \beta, \kappa(\bar p)}} \leq \mathbb{E}_{p_0} \left[ \left\langle \cos(m \theta) \right\rangle_{\mu_{\Lambda , \beta, r}} \right].
\end{equation*}
\end{proposition}

\begin{proof}
    We first recall the notation for the partition function of the $XY$ model in a Bernoulli site percolation: for any $r \in \{0 , 1\}^\Lambda$,
    \begin{equation} \label{def.Zlambdabetar}
        Z_{\Lambda, \beta, r} := \int \exp \left( \beta \sum_{x \sim y} r_x r_y \cos (\theta_x - \theta_y) \right) \prod_{x \in \Lambda}d \theta_x.
    \end{equation}
    We decompose the expectation $\left\langle \cos(m \theta) \right\rangle_{\mu_{\Lambda , \beta, \kappa(\bar p)}}$ according to the following computation
    \begin{align*}
        \left\langle \cos(m \theta) \right\rangle_{\mu_{\Lambda , \beta, \kappa(\bar p)}} & = \frac{1}{Z_{\Lambda, \beta, \kappa(\bar p)}} \sum_{r \in \{0 , 1 \}^{\Lambda}} \P_{\Lambda,\bar p}( \{ r \}) \int \cos(m \theta) \exp \left(  \beta \sum_{x \sim y} r_x r_y \cos (\theta_x - \theta_y)  \right) \prod_{x \in \Lambda}d \theta_x \\
        & = \frac{1}{Z_{\Lambda, \beta, \kappa(\bar p)}} \sum_{r \in \{0 , 1 \}^{\Lambda}} \P_{\Lambda,\bar p}(\{ r \}) Z_{\Lambda, \beta, r}  \left\langle \cos(m \theta) \right\rangle_{\mu_{\Lambda, \beta,r}}.
    \end{align*}
    We then let $\nu'_{\Lambda, \beta, \bar p}$ be the probability measure defined on the space of site percolation configurations~$ \{0 , 1\}^\Lambda$ according to the formula, for any $r \in \{ 0 ,1\}^{\Lambda}$,
    \begin{equation} \label{measurenu.eq}
        \nu'_{\Lambda, \beta, \bar p} (r) := \frac{ \P_{\Lambda,\bar p}( \{ r \})Z_{\Lambda, \beta, r}}{Z_{\Lambda, \beta, \kappa(\bar p)}}.
    \end{equation}
    Using this notation, we may write
    \begin{equation*}
        \left\langle \cos(m \theta) \right\rangle_{\mu_{\Lambda , \beta, \kappa(\bar p)}} = \sum_{r \in \{0 , 1 \}^{\Lambda}}  \left\langle \cos(m \theta) \right\rangle_{\mu_{\Lambda, \beta, r}} \nu'_{\Lambda, \beta,\bar p} ( \{ r \}) = \mathbb{E}_{\nu'_{\Lambda, \beta, \bar p}} \left[ \left\langle \cos(m \theta) \right\rangle_{\Lambda, \beta, r} \right].
    \end{equation*}
    We next claim that, with the value of $p_0$ introduced in the statement of the proposition, one has the stochastic domination
    \begin{equation} \label{ineq:stocdominqtion}
        \nu'_{\Lambda, \beta, \bar p} \preceq \P_{\Lambda, p_0}.
    \end{equation}
    The inequality~\eqref{ineq:stocdominqtion} is sufficient to complete the proof of Proposition~\ref{prop4.1}: indeed the Ginibre inequality (in the form of the inequality~\eqref{eq:Ginibreineq}) implies that the function $r \mapsto \left\langle \cos(m \theta) \right\rangle_{\mu_{\Lambda, \beta , r}}$ is increasing. Combining this observation with the stochastic domination~\eqref{ineq:stocdominqtion} implies that
    \begin{equation*}
        \left\langle \cos(m \theta) \right\rangle_{\mu_{\Lambda , \beta, \kappa(\bar p)}} =   \mathbb{E}_{\nu'_{\Lambda, \beta,\bar p}} \left[ \left\langle \cos(m \theta) \right\rangle_{\mu_{\Lambda, \beta, r}} \right] \leq \mathbb{E}_{p_0} \left[ \left\langle \cos(m \theta) \right\rangle_{\mu_{\Lambda , \beta, r }} \right].
    \end{equation*}
    There remains to prove~\eqref{ineq:stocdominqtion}. Using Lemma~\ref{Lemma1.5} (and noting that the two measures $\nu'_{\Lambda, \beta,\bar p}$ and $\P_{\Lambda, p_0}$ are strictly positive), it is sufficient to show that, for any vertex $x \in \Lambda$ and any percolation configuration $r_1 \in \{0,1\}^{\Lambda \setminus \{ x \}}$,
    \begin{equation} \label{eqnunu'DC}
        \nu'_{\Lambda, \beta,\bar p} \left( r_x = 1 \, | \, r = r_1 ~ \mbox{in} ~ \Lambda \setminus \{ x \} \right) \leq p_0.
    \end{equation}
    To prove the inequality~\eqref{eqnunu'DC}, we fix $r_1 \in \{0,1\}^{\Lambda \setminus \{ x \}}$ and let $r_1^+ \in \left\{ 0,1 \right\}^{\Lambda}$ be the percolation configuration defined by the identities $r_1^+ = r_1$ in $\Lambda \setminus \{ x \}$ and $r_{1,x}^+ = 1$. Similarly, we define the percolation configuration $r_1^- \in \left\{ 0,1 \right\}^{\Lambda}$ by the identities $r_1^- = r_1$ in $\Lambda \setminus \{ x \}$ and $r_{1,x}^- = 0$. Using the definition of the measure $\nu'_{\Lambda, \beta,\bar p} $, we have the identity
    \begin{equation} \label{eq:1635}
        \nu'_{\Lambda, \beta,\bar p} \left( r_x = 1 \, | \, r = r_1 ~ \mbox{in} ~ \Lambda \setminus \{ x \} \right) = \frac{ \P_{\Lambda,\bar p}( \{ r_1^+ \} )Z_{\Lambda, \beta, r_1^+}}{ \P_{\Lambda,\bar p}(\{ r_1^+ \})Z_{\Lambda, \beta, r_1^+} + \P_{\Lambda,\bar p}( \{ r_1^- \})Z_{\Lambda, \beta, r_1^-} }.
    \end{equation}
    Using that $\P_{\Lambda,\bar p}$ is the i.i.d. Bernoulli site percolation measure, we have the identity 
    \begin{equation} \label{eq:1636}
    \P_{\Lambda,\bar p}( \{ r_1^- \} ) = \frac{1-\bar p}{\bar p}  \P_{\Lambda,\bar p}( \{ r_1^+ \}).
    \end{equation}
    Using the definition of the partition function $Z_{\Lambda, \beta , r }$ introduced in~\eqref{def.Zlambdabetar}, we see that
    \begin{align} \label{eq:1637}
         Z_{\Lambda, \beta, r_1^+} & = \int \exp \left( \beta \sum_{y \sim y'} r_{1,y}^+ r_{1,y'}^+ \cos (\theta_y - \theta_{y'}) \right) \prod_{y \in \Lambda}d \theta_y \\
         & = \int \exp \left( \beta \sum_{\substack{x' \in \Lambda \\ x' \sim x}} r_{1, x'} \cos (\theta_x - \theta_{x'}) \right) \exp \left( \beta \sum_{y \sim y'} r_{1,y}^- r_{1,y'}^- \cos (\theta_y - \theta_{y'}) \right)\prod_{y \in \Lambda}d \theta_y  \notag \\
         & \leq \exp \left( 2d \beta \right)  \int \exp \left( \beta \sum_{y \sim y'} r_{1,y}^- r_{1,y'}^- \cos (\theta_y - \theta_y) \right)\prod_{y \in \Lambda}d \theta_y \notag \\
         & = \exp \left( 2d \beta \right) Z_{\Lambda, \beta, r_1^-}. \notag
    \end{align}
    Combining the inequalities and identities~\eqref{eq:1635},~\eqref{eq:1636} and~\eqref{eq:1637}, we obtain
    \begin{equation*}
        \nu'_{\Lambda, \beta,\bar p} \left( r_x = 1 \, | \, r = r_1 ~ \mbox{in} ~ \Lambda \setminus \{ x \} \right) \leq \frac{1}{1 + \frac{1-\bar p}{\bar p}  \exp \left( - 2d \beta \right)} = p_0.
    \end{equation*}
   The proof of the inequality~\eqref{eqnunu'DC} is complete.
\end{proof}

\subsection{Stochastic domination and  Ginibre inequality} \label{Section5}

This section is devoted to the proof of Theorem~\ref{mainthm} based on the result of Proposition~\ref{prop4.1} and the Ginibre inequality.

\begin{proof}[Proof of Theorem~\ref{mainthm}]
    In dimension $d = 2$, fix $\beta = 4 \beta_{BKT}$, $\bar p:=\frac 1 2$ and $p_0 = p_0(\tfrac 1 2, 4 \beta_{BKT}) =  \frac 1 {1+ e^{-16 \beta_{BKT}} }$. 
    Using Wells' inequality with the choice~\eqref{e.KEY}, i.e.  $a := \min \left( \frac12 , \bar p \right) = \tfrac12$, we have
    \begin{equation} \label{ineqKT}
        a^2 \beta =  \beta_{BKT}.
    \end{equation}
    We fix a (large) finite set $\Lambda \subseteq \Z^2$, two vertices $x , y \in \Lambda$ and apply Proposition~\ref{prop4.1} with the function $m = \indc_{x} - \indc_{y}$ to obtain the inequality,
    \begin{equation*}
        \left\langle \cos(\theta_x - \theta_y) \right\rangle_{\mu_{\Lambda , \beta, \kappa(1/2)}} \leq \mathbb{E}_{p_0} \left[ \left\langle \cos(\theta_x - \theta_y) \right\rangle_{\mu_{\Lambda , \beta, r }} \right].
    \end{equation*}
    Applying Wells' inequality with the function $m = \indc_{x} - \indc_y$ , we deduce that
    \begin{equation*}
        \left\langle \cos(\theta_x - \theta_y) \right\rangle_{\mu_{\Lambda, a^2 \beta}} \leq \mathbb{E}_{p_0} \left[ \left\langle \cos(\theta_x - \theta_y) \right\rangle_{\mu_{\Lambda , \beta, r} } \right].
    \end{equation*}
    Using Ginibre inequality, we may take the limit in the previous inequality along an increasing sequence of finite subsets $\Lambda \uparrow \Zd$ (using the monotone convergence theorem for the right-hand side). We obtain
    \begin{equation*}
        \left\langle \cos(\theta_x - \theta_y) \right\rangle_{\mu_{a^2 \beta}} \leq \mathbb{E}_{p_0} \left[ \left\langle \cos(\theta_x - \theta_y) \right\rangle_{\mu_{\beta, r }} \right].
    \end{equation*}
    Using that the two-point function of the $XY$ model at the critical inverse temperature $\beta_{KT}$ decays polynomially fast (see~\cite[Theorem 1]{van2023elementary}), the identity~\eqref{ineqKT} implies that the two-point function $x \mapsto \mathbb{E}_{p_0} [ \left\langle \cos(\theta_0 - \theta_x) \right\rangle_{\mu_{\beta, r}} ]$ decays at most polynomially fast in $|x|$. 

    In order to derive an upper bound on the two-point function (and show that it decays at least polynomially fast), we first use the Ginibre correlation inequality to show that, for any percolation configuration $r \in \{0 , 1\}^{\Zd}$ (i.e., we increase all the values of $r$ to $1$),
    \begin{equation} \label{ineq:upperboundquenched}
        \left\langle \cos(\theta_x - \theta_y) \right\rangle_{\mu_{\beta, r }}  \leq \left\langle \cos(\theta_x - \theta_y) \right\rangle_{\mu_{\beta}}.
    \end{equation}
    Note that, since this inequality holds for any percolation configuration, it also holds when an expectation is taken on the left-hand side. We then apply the result of McBryan and Spencer~\cite{mcbryan1977decay} which shows that, in two dimensions and for any inverse temperature $\beta$, the right-hand side decays at least polynomially fast. 
    
    Combining the results of the two previous paragraphs shows the polynomial decay of the expected two-point function when $\beta=4\beta_{BKT}$ and $p=p_0=\frac 1 {1+e^{-16 \beta_{BKT}}}$, and thus completes the proof of Theorem~\ref{mainthm} in this case. Now, using stochastic monotonicity together with Ginibre inequality (and the results of~\cite{mcbryan1977decay} for the upper bound), this is clearly extended to all $\beta\geq 4 \beta_{BKT}$ and $p\geq p_0$.

    The same argument can be applied in dimension $3$ and higher by replacing the value $\beta_{BKT}$ by the critical inverse temperature $\beta_c(d)$. The only difference in this case is that it is expected that long-range order does not hold at the critical point $\beta_c(d)$, this is why strict inequalities are used when $d\geq 3$ in Theorem~\ref{mainthm}. We argue as follows: for any $\eps>0$, consider $\beta = 4(\beta_c(d) +\eps)$. Using the same analysis as in the case $d=2$, together with Wells' inequality, we obtain long-range order for the $XY$ model on site $p$-percolation, with $p=p_0(\tfrac 1 2, \beta) = \frac 1 {1+e^{-8d \beta_c - 8\eps}}$. Using the same monotonies, and letting $\eps\to 0$, we thus obtain long-range order for any $\beta>4 \beta_c(d)$ and any $p>  \frac 1 {1+ e^{-8d \beta_c(d)}}$. 
\end{proof}

\begin{remark}\label{}
The reader may wonder why we do not optimise further our choice of $\bar p$ as a function of $\beta \geq 4 \beta_{BKT}$ in order to obtain better bounds on the density $p$ as a function of $\beta$. Indeed, the present proof is only using $\bar p:= \tfrac 1 2$ which seems  rather arbitrary. It turns out, somewhat surprisingly, that for any given $\beta>4\beta_{BKT}$, if instead of using stochastic domination and Ginibre, one optimises over a suitable choice of $\bar p \leq \tfrac 12$, then the bounds one obtains on the percolation densities $p$ are not better and even degenerate as $\beta \to \infty$. 
\end{remark}

\section{Phase transitions for the $XY$ model on a supercritical percolation cluster} \label{section4.2}

This section is devoted to the proof of Theorem~\ref{proofKTsitegeneral}. We will only prove lower bounds as a polynomial upper bound on the expectation of the two-point function in two dimensions is obtained by combining the inequality~\eqref{ineq:upperboundquenched} together with the result of~\cite{mcbryan1977decay} (exactly as in the proof of Theorem~\ref{mainthm}). The proof is based on Wells' inequality, the Ginibre inequality and a renormalization argument. We first prove, in Section~\ref{section5.1}, a more general version of the stochastic domination inequality of Section~\ref{Section4.1} involving the extended lattice $\Zd_n$ and the $XY$ model with heterogeneous temperatures introduced in Definition~\ref{def.heterogeneousXY}. The result is stated in Proposition~\ref{prop4.1extended} below, and its proof is essentially identical to the one presented in Section~\ref{Section4.1} (but more technical due to the nature of the objects involved). 
In Section~\ref{subsection5.2}, we combine the results of Section~\ref{section5.1} with a renormalization argument for the supercritical percolation cluster (using the results collected in Section~\ref{sectionrenormalization}), the Ginibre inequality and Proposition~\ref{prop.phasetransitionextended} (establishing the existence of phase transitions for the $XY$ model with heterogeneous temperatures on the extended lattice) and complete the proof of Theorem~\ref{proofKTsitegeneral}. Finally, Section~\ref{subsection5.3} is devoted to the proof of Corollary~\ref{c.AlmostSure}

\subsection{Domination for the $XY$ model on a percolation cluster of the extended lattice $\Zd_n$} \label{section5.1}

In this section, we recall the definition of the extended lattice $\Zd_n$ stated in Definition~\ref{def.Zdn}, fix an integer $n \in \mathbb{N}$, a probability $\bar p \in (0 , 1)$, and define the collection of probability measures $\kappa= \kappa(\bar p) = (\kappa_x)_{x \in \Z^d_n}$ according to the formula (using the notation of Remark~\ref{remark2.7} to embed the vertices of $\Zd$ in $\mathbb{Z}^d_n$)
\begin{equation}\label{def.measurekappa}
    \kappa_x : = 
    \left\{ \begin{aligned}
        (1-\bar p) \delta_0 + \bar p \delta_1 & ~\mbox{if}~ x \in \Zd, \\
         \delta_1 & ~\mbox{if}~ x \in \Z^d_n \setminus \Zd.
    \end{aligned} \right.
\end{equation}

We next introduce the versions of the $XY$ model with heterogenous temperatures on a percolation cluster and with annealed disorder which will be used in the statement of Proposition~\ref{prop4.1extended} below.
Given a finite set $\Lambda \subseteq \Z^d_n$ and two inverse temperatures $\beta_1 , \beta_2 > 0$, we define the Hamiltonian: for $r \in  \{ 0 , 1 \}^\Lambda$ and $\theta \in [0 , \infty)^\Lambda$,
\begin{equation*}
    H_{\Lambda, \beta_1 , \beta_2}(r , \theta) := - \beta_1 \sum_{\substack{x \sim_n y \\x , y \in \Lambda \\ \{ x , y \} \cap \Zd \neq \emptyset}} r_x r_y \cos (\theta_x - \theta_y) - \beta_2 \sum_{\substack{x \sim_n y \\ x,y \in \Lambda \\ \{ x , y \} \cap \Zd = \emptyset}} r_x r_y \cos (\theta_x - \theta_y) 
\end{equation*}
as well as the two probability measures:
\begin{itemize}
    \item The $XY$ model with heterogenous temperatures on a percolation cluster: for a fixed percolation configuration $r \in \{ 0 , 1 \}^\Lambda$,
\begin{equation} \label{mubeta1bata2ndef}
         \mu_{\Lambda, \beta_1 , \beta_2 , r}(d\theta) := \frac{1}{Z_{\Lambda, \beta_1 , \beta_2 , r}}\exp \left( - H_{\Lambda, \beta_1 , \beta_2}(r , \theta)  \right) \prod_{x \in \Lambda} d \theta_x.
\end{equation}
For later purposes, we note that, as it was for various models introduced before, compactness arguments and the Ginibre correlation inequality imply that, for any percolation configuration $r \in \{ 0 , 1 \}^{\Zd}$, the sequence of measures $\mu_{\Lambda, \beta_1 , \beta_2, r}$ converges as $\Lambda \uparrow \Zd_n$ to an infinite-volume measure which will be denoted by $\mu_{\beta_1 , \beta_2, r}.$
    \item The $XY$ model with heterogenous temperatures and annealed disorder:
    \begin{equation*}
        \mu_{\Lambda, \beta_1 , \beta_2 , \kappa}(dr d\theta) := \frac{1}{Z_{\Lambda, \beta_1 , \beta_2 , \kappa}}\exp \left( - H_{\Lambda, \beta_1 , \beta_2}(r , \theta)  \right) \prod_{x \in \Lambda} d \kappa_x(r_x) d \theta_x.
\end{equation*}
\end{itemize}
The following proposition shows that the two-point function under the measure $\mu_{\Lambda, \beta, \kappa}$ is smaller than the one of an $XY$ model in a sufficiently supercritical Bernoulli site percolation (using the measure $\P_{\Lambda,p}$ on the extended lattice $\Z^d_n$ introduced in Definition~\ref{DEf.bernoullipercextend}).

\begin{proposition} \label{prop4.1extended}
Fix $\beta_1 , \beta_2 \in (0 , \infty)$, $\bar p \in (0 , 1)$ and set $p_0 := \frac{\bar p}{\bar p + (1-\bar p)  \exp \left( - 2d \beta_1 \right)} \in (0 , 1)$. Then, for any finite subset $\Lambda \subseteq \Zd_n$ and any function $m \in \Z^\Lambda$,
\begin{equation*}
       \left\langle \cos(m \theta) \right\rangle_{\mu_{\Lambda, \beta_1 , \beta_2 , \kappa}} \leq \mathbb{E}_{p_0} \left[ \left\langle \cos(m \theta) \right\rangle_{\mu_{\Lambda , \beta_1, \beta_2, r}} \right]\,,
\end{equation*}
where $\kappa=\kappa(\bar p)$ is the measure defined in~\eqref{def.measurekappa}.
\end{proposition}

\begin{remark}
    A crucial feature of this proposition is that the probability $p_0$ depends only on the inverse temperature $\beta_1$. Together with Remark~\ref{remark2.24}, this property plays a fundamental role in the proof of Theorem~\ref{proofKTsitegeneral}
and is in fact the reason to introduce the $XY$ model with heterogenous temperatures.
\end{remark}

\begin{proof}
    We first recall the notation for the partition function of the $XY$ model in a Bernoulli site percolation configuration: for any $r \in \{0 , 1\}^\Lambda$,
    \begin{equation} \label{def.Zlambdabetarbis}
        Z_{\Lambda, \beta_1 , \beta_2, r} := \int \exp \left(  - H_{\Lambda, \beta_1 , \beta_2}(r , \theta) \right) \prod_{x \in \Lambda}d \theta_x.
    \end{equation}
    Decomposing the expectation $\left\langle \cos(m \theta) \right\rangle_{\mu_{\Lambda , \beta_1 , \beta_2, \kappa}}$ as in the proof of Proposition~\ref{prop4.1}, we have the identity
    \begin{equation*}
        \left\langle \cos(m \theta) \right\rangle_{\mu_{\Lambda , \beta_1 , \beta_2, \kappa}} = \frac{1}{Z_{\Lambda, \beta_1 , \beta_2, \kappa}} \sum_{r \in \{0 , 1 \}^{\Lambda}} \P_{\Lambda, \bar p}(\{ r \}) Z_{\Lambda, \beta_1 , \beta_2, r}  \left\langle \cos(m \theta) \right\rangle_{\mu_{\Lambda, \beta_1 , \beta_2,r}}.
    \end{equation*}
    We then let $\nu'_{\Lambda,\beta_1 , \beta_2, \bar p}$ be the probability measure defined on the space of site percolation configurations $ \{0 , 1\}^\Lambda$ according to the formula, for any $r \in \{ 0 ,1\}^{\Lambda}$,
    \begin{equation*}
        \nu'_{\Lambda, \beta_1 , \beta_2 , \bar p} (\{ r \}) := \frac{ \P_{\Lambda, \bar p}(\{ r \})Z_{\Lambda, \beta_1 , \beta_2, \kappa}}{Z_{\Lambda, \nu}},
    \end{equation*}
    so that 
    \begin{equation*}
        \left\langle \cos(m \theta) \right\rangle_{\mu_{\Lambda , \beta_1 , \beta_2, \kappa}} = \sum_{r \in \{0 , 1 \}^{\Lambda}}  \left\langle \cos(m \theta) \right\rangle_{\mu_{\Lambda, \beta_1 , \beta_2, r}} \nu'_{\Lambda, \beta_1 , \beta_2, \bar p} (\{ r \}) = \mathbb{E}_{\nu'_{\Lambda, \beta_1 , \beta_2, \bar p}} \left[ \left\langle \cos(m \theta) \right\rangle_{\Lambda, \beta_1 , \beta_2, r} \right].
    \end{equation*}
    We next claim that the following stochastic domination holds
    \begin{equation} \label{ineq:stocdominqtionext}
        \nu'_{\Lambda, \beta_1 , \beta_2, \bar p} \preceq \P_{\Lambda, p_0}.
    \end{equation}
    This inequality is enough to conclude. Indeed, combining the Ginibre inequality (in the form of the inequality~\eqref{eq:Ginibreineq}) with the stochastic domination~\eqref{ineq:stocdominqtionext}, we obtain the inequality
    \begin{equation*}
        \left\langle \cos(m \theta) \right\rangle_{\mu_{\Lambda ,\beta_1 , \beta_2, \kappa}} =   \mathbb{E}_{\nu'_{\Lambda, \beta_1 , \beta_2, \bar p}} \left[ \left\langle \cos(m \theta) \right\rangle_{\mu_{\Lambda, \beta_1 , \beta_2, r}} \right] \leq \mathbb{E}_{p_0} \left[ \left\langle \cos(m \theta) \right\rangle_{\mu_{\Lambda , \beta_1 , \beta_2, r }} \right],
    \end{equation*}
    which completes the proof of Proposition~\ref{prop4.1extended}.  

     There remains to prove~\eqref{ineq:stocdominqtionext}. Using Lemma~\ref{Lemma1.5ext} (and noting that the two measures $\nu'_{\Lambda, \beta_1 , \beta_2, \bar p}$ and $\P_{\Lambda, p_0}$ satisfy the first two assumptions of the lemma), it is sufficient to show that, for any vertex $x \in \Lambda \cap \Zd$ and any percolation configuration $r_1 \in \{0,1\}^{\Lambda \setminus \{ x \}}$ such that $r_1 = 1$ on the set $\Lambda \setminus \Zd$,
    \begin{equation} \label{eqnunu'DCbis}
        \nu'_{\Lambda, \beta_1 , \beta_2 ,\bar p} \left( r_x = 1 \, | \, r = r_1 ~ \mbox{in} ~ \Lambda \setminus \{ x \} \right) \leq p_0.
    \end{equation}
    To prove the inequality~\eqref{eqnunu'DCbis}, let $r_1^+ \in \left\{ 0,1 \right\}^{\Lambda}$ be the percolation configuration defined by the identities $r_1^+ = r_1$ in $\Lambda \setminus \{ x \}$ and $r_{1,x}^+ = 1$. Similarly, we define the percolation configuration $r_1^- \in \left\{ 0,1 \right\}^{\Lambda}$ by the identities $r_1^- = r_1$ in $\Lambda \setminus \{ x \}$ and $r_{1,x}^- = 0$. Using the definition of the measure $\nu'_{\Lambda, \beta_1 , \beta_2,\bar p} $, we have the identity
    \begin{equation} \label{eq:1635bis}
        \nu'_{\Lambda, \beta_1 , \beta_2,\bar p} \left( r_x = 1 \, | \, r = r_1 ~ \mbox{in} ~ \Lambda \setminus \{ x \} \right) = \frac{ \P_{\Lambda,\bar p}( \{ r_1^+ \})Z_{\Lambda, \beta_1 , \beta_2, r_1^+}}{ \P_{\Lambda, \bar p}( \{ r_1^+ \})Z_{\Lambda, \beta_1 , \beta_2, r_1^+} + \P_{\Lambda,\bar p}( \{ r_1^- \})Z_{\Lambda, \beta_1 , \beta_2, r_1^-} }.
    \end{equation}
    Using that $\P_{\Lambda,\bar p}$ is the i.i.d. Bernoulli site percolation measure, we have the identity 
    \begin{equation} \label{eq:1636bis}
    \P_{\Lambda,\bar p}( \{ r_1^- \}) = \frac{1-\bar p}{\bar p}  \P_{\Lambda,\bar p}( \{ r_1^+ \}).
    \end{equation}
    Using that the percolation configuration $r_1^+$ and $r_1^-$ are only different at the vertex $x \in \Lambda \cap \Zd$, we have the upper bound, for any spin configuration $\theta \in [0, 2 \pi)^{\Lambda}$,
    \begin{equation*}
        H_{\Lambda, \beta_1 , \beta_2}(r_1^+ , \theta) \geq H_{\Lambda, \beta_1 , \beta_2}(r_1^- , \theta) - 2 d \beta_1.
    \end{equation*}
    Using the definition of the partition function $Z_{\Lambda, \beta_1 , \beta_2 , r }$ introduced in~\eqref{def.Zlambdabetarbis} and the previous upper bound, we obtain
    \begin{align*}
         Z_{\Lambda, \beta_1, \beta_2 , r_1^+} & = \int \exp \left(  - H_{\Lambda, \beta_1 , \beta_2}(r_1^+ , \theta) \right) \prod_{x \in \Lambda}d \theta_x \\
         & \leq \exp \left( 2d \beta_1 \right) \int \exp \left(  - H_{\Lambda, \beta_1 , \beta_2}(r_1^- , \theta) \right) \prod_{x \in \Lambda}d \theta_x \\
         & =  \exp \left( 2d \beta_1 \right) Z_{\Lambda, \beta_1 , \beta_2, r_1^-}. 
    \end{align*}
    Combining the previous inequality with the identities~\eqref{eq:1635bis},~\eqref{eq:1636bis}, we obtain
    \begin{equation*}
        \nu'_{\Lambda, \beta_1 , \beta_2 ,\bar p} \left( r_x = 1 \, | \, r = r_1 ~ \mbox{in} ~ \Lambda \setminus \{ x \} \right) \leq \frac{\bar p}{\bar p + (1-\bar p)  \exp \left( - 2d \beta_1 \right)} = p_0.
    \end{equation*}
   The proof of the inequality~\eqref{eqnunu'DCbis} is complete.
\end{proof}

\subsection{Proof of Theorem~\ref{proofKTsitegeneral}.} \label{subsection5.2}

This section is devoted to the proof of Theorem~\ref{proofKTsitegeneral}.
We split it into three subsections. In Section~\ref{subsec5.2.1}, we set up the renormalization argument by introducing a series of preliminary definitions and notations. Section~\ref{subsec5.2.2} is the core of the argument, we make use of the Ginibre correlation inequality to show that the two-point function of the $XY$ model on a supercritical percolation cluster is bounded from below by the one of a $XY$ model on the extended lattice on a \emph{highly} supercritical percolation cluster. We then lower bound this two-point function in Section~\ref{subsec5.2.3}. The proof of this latter inequality is similar to the one of Theorem~\ref{mainthm} 
(in Section~\ref{Section4.1}) and makes use of the stochastic domination estimate established in Section~\ref{section5.1}, Wells' inequality and Proposition~\ref{prop.phasetransitionextended}.

We shall only give the details for the case of site percolation below.  Indeed the proof for {\em edge} percolation is almost identical: the only difference is in Section~\ref{subsec5.2.1} below where, in order to set up the renormalization argument, we need to use Proposition~\ref{proppre-good} (and the inequality~\eqref{eqpdecaygoodbox}) for edge percolation instead of site percolation. We thus focus only on site percolation and leave the details to the reader.

\subsubsection{Preliminary definitions: the site percolation on the renormalized graph} \label{subsec5.2.1}

We fix a dimension $d \geq 2$ and a probability $p > p_{c , \mathrm{site}}(d)$, a vertex $x \in \Zd$ and introduce a few parameters and definitions. We recall the statement of Proposition~\ref{prop.phasetransitionextended} and set $\beta_1 := 2^{2d + 2} \beta_{1 , BKT} = 64 \beta_{1 , BKT}$ in dimension $d = 2$ and $\beta_1 := 2^{2d + 2} \beta_{1 , c}(d)$ in dimension $d \geq 3$. We then set
    \begin{equation} \label{def.p0thm1.3}
            p_0 := \frac{1}{1 + \exp \left( - 2 d \beta_{1} \right)}.
    \end{equation}
    (This is because as in Section \ref{Section4.1}, we will choose $\bar p:=\tfrac 1 2$).
    For each integer $L \in \N$, we partition $\Zd$ into boxes of sidelength $L$ according to the identity
    \begin{equation*}
        \Zd := \bigcup_{z \in (2L +1) \Zd} (z + \Lambda_L).
    \end{equation*}
    Using this partition, we may define a renormalized graph as follows: we denote by $\mathcal{G}_L$ the graph whose vertices are the boxes of the form $(z + \Lambda_L)$ with $z \in (2L +1) \Zd$, and the edges are the pairs of boxes $\{ z + \Lambda_L , z' + \Lambda_L \}$ which have adjacent faces, or equivalently such that $|z - z'|_1 = 2L+1$ (see Figure~\ref{figure.partition}). We note that there exists an isomorphism of graph between $\mathcal{G}_L$ and $\Zd$.

    For each integer $L \in \N$, we select two vertices $x_{0,L}, x_{1 , L}$ such that
    \begin{equation} \label{requirementx0x1}
        x_{0,L} \in \Lambda_{L/2}, ~~x_{1 , L} \in \bigcup_{z \in (2L +1) \Zd} (z + \Lambda_{L/2}) ~~\mbox{and}~~x = x_{1 , L} - x_{0 , L}.
    \end{equation}
    Using the translation invariance of Bernoulli site percolation, we have the identity
    \begin{equation} \label{eq:identitytransinv}
        \mathbb{E}_p [  \left\langle \cos(\theta_0 - \theta_x) \right\rangle_{\mu_{\beta, r}}] = \mathbb{E}_p [  \left\langle \cos(\theta_{x_{0 , L}} - \theta_{x_{1 , L}}) \right\rangle_{\mu_{\beta, r}}].
    \end{equation}
    We will from now on study the two-point function on the right-hand side: it enjoys the convenient property that, contrary to the vertex $x$, none of the points $x_{0 , L}$ and $x_{1 , L}$ can be close to the boundary of one of the boxes $(z + \Lambda_L)$ with $z \in (2L +1) \Zd$.

    \begin{figure}
\begin{center}
\includegraphics[width=7cm]{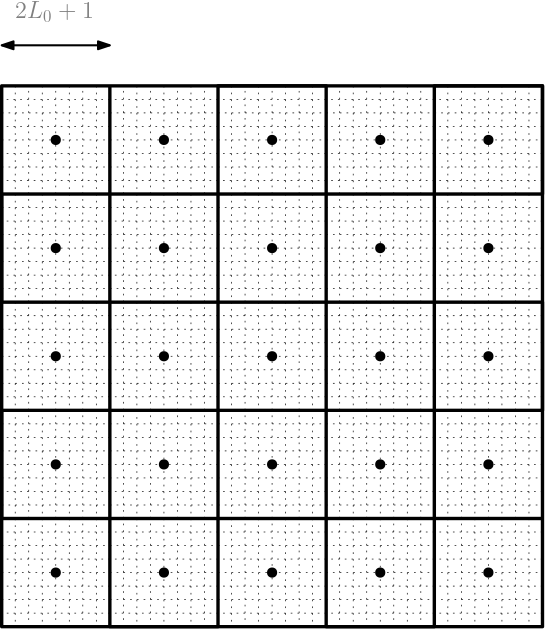} \hspace{10mm}
\includegraphics[width=7cm]{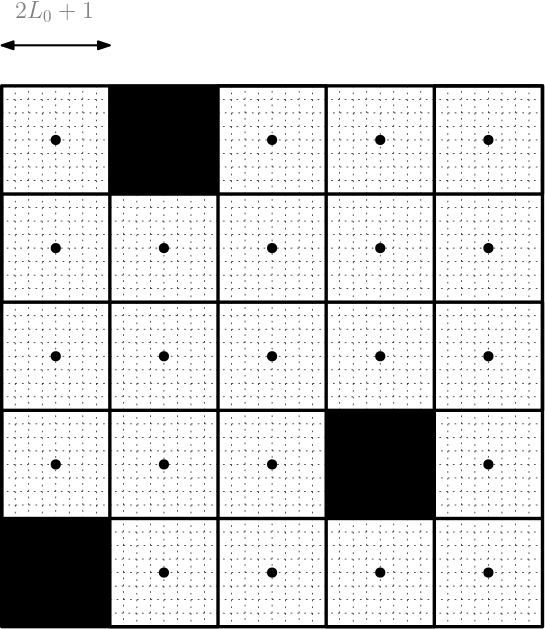}
\caption{The partition of $\Zd$ into boxes of the form $(z + \Lambda_{L_0})$ with $z \in (2 L_0 + 1) \Zd$ is drawn on the left. A realization of site percolation on the renormalized graph $\mathcal{G}_L$ is depicted on the right: the black cells are the closed sites of the percolation configuration, the other cells are the open sites.} \label{figure.partition}
\end{center}
\end{figure}

Equipped with these definitions, we define a translation invariant, $1$-dependent site percolation on the renormalized graph $\mathcal{G}_L$ as follows. We consider a realization of an i.i.d. Bernoulli site percolation of probability $p$ on the lattice $\Zd$ conditioned on the event
    \begin{multline*}
        E_{L, x} := \left\{ x_{0 , L} \mbox{ is connected to the boundary of the box } x_{0 , L} + \Lambda_{L/2} \right\} \\ 
        \bigcap \left\{ x_{1 , L} \mbox{ is connected to the boundary of the box } x_{1 , L} + \Lambda_{L/2} \right\}.
    \end{multline*}
    Note that, by the FKG inequality, the probability of the event $E_{L,x}$ is larger than $\theta(p)^2$, where $\theta(p) \in (0,1]$ denotes the density of the infinite cluster. We then declare a box $(z + \Lambda_L)$ open if and only if it is good in the sense of Definition~\ref{def.goodbox}. This procedure provides a realization of a site percolation on the graph $\mathcal{G}_L$ (see Figure~\ref{figure.partition}). We collect below two properties of this site percolation which are a consequence of the results stated in Remark~\ref{remark2.17}:
    \begin{itemize}
        \item The site percolation on the renormalized graph $\mathcal{G}_L$ is $1$-dependent. This is a consequence of the two following observations: the event ``the box $(z + \Lambda_L)$ is good" is $\mathcal{F}(z + \Lambda_{2L})$-measurable, and the event $E_{L, x}$ is the intersection of two local events which are respectively $\mathcal{F}(x_{0 , L} + \Lambda_{L/2})$-measurable and $\mathcal{F}(x_{1,L} + \Lambda_{L/2})$-measurable.
        \item Conditionally on the event $E_{L,x}$, the probability for a box of the form $(z + \Lambda_L)$ to be good is larger than $1- C \exp(- c L)$ (for some constants $C < \infty$ and $c > 0$ depending only on the probability $p$ and the dimension~$d$). This is a consequence of the inequality~\eqref{eqpdecaygoodbox} and the lower bound $\P_p [E_{L,x}] \geq \theta(p)^2$.
    \end{itemize}

    By combining the two previous observations with Theorem~\ref{th.Liggett} (due to \cite{liggett1997domination}) we deduce that there exists an integer $L_0= L_0(p, d) \in \N$ depending only on the probability $p$ and the dimension $d$ such that the site percolation defined on the renormalized graph $\mathcal{G}_{L_0}$ stochastically dominates an i.i.d. site percolation of parameter $p_0$ (defined in~\eqref{def.p0thm1.3}). We fix this integer $L_0$ for the rest of the proof and introduce three additional notations:
    \begin{itemize}
        \item We set $x_0 := x_{0,L_0}$ and $x_1 := x_{1,L_0}$.
        \item For each $y \in \Zd$, we denote by $[y]$ the unique vertex of $\Zd$ such that $y \in (2L_0 + 1) [y] + \Lambda_{L_0}$. We note that $[x_0] = 0$.
        \item We set $n := 2 \left| \Lambda_{L_0} \right| = 2 (2L_0+1)^d.$
        \item Using Proposition~\ref{prop.phasetransitionextended}, we consider the inverse temperature $\beta_2 := 2^{2d} \beta_{2, BKT}(n) = 16 \beta_{2, BKT}(n)$ in dimension $d = 2$ and $\beta_2 := 2^{2d} \beta_{2, c}(n, d)$ in dimension $d \geq 3$. We then set 
    \begin{equation} \label{def.parameterbeta}
        \beta := \max  \left( \beta_{1}, \beta_{2} \right).
    \end{equation}
    \end{itemize}
\begin{figure}
\begin{center}
\includegraphics[width=7cm]{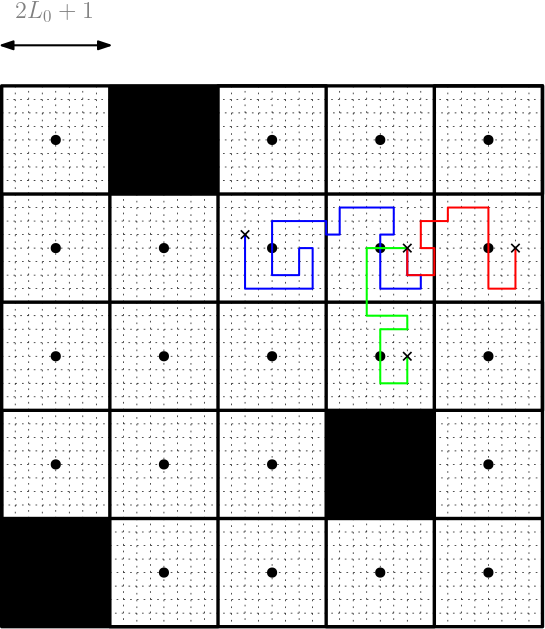}
\caption{In each good box of the form $(z + \Lambda_{L_0})$, a vertex $c_z$ is selected. Four of these vertices are represented by the crosses in the picture. Paths are then selected to connect these vertices: three of these paths are drawn in the picture using different colors.} \label{figure.paths}
\end{center}
\end{figure}

\subsubsection{Renormalization argument for the two-point function} \label{subsec5.2.2}
This section is the core of the proof of Theorem~\ref{proofKTsitegeneral}. Equipped with the definitions and notations of the previous section, we will prove the following inequality
    \begin{equation} \label{eq:1143}
        \E_{p_0} \left[ \langle \cos( \theta_0 - \theta_{[x_1]} ) \rangle_{\mu_{\beta/2^{2d} , \beta /2^{2d} , r}} \right] \leq \E_{p} \left[ \langle \cos( \theta_{x_0} - \theta_{x_1} ) \rangle_{\mu_{\beta , r}} \, | \, E_{L_0,x} \right],
    \end{equation}
    where, on the right-hand side, we used the definition of the $XY$ model with heterogeneous temperatures on a percolation cluster on the extended lattice $\Zd_n$ introduced in~\eqref{mubeta1bata2ndef}. (N.B. Recall the values $p_0$, $n := 2 (2L_0+1)^d$ and $L_0=L_0(p , d)$ are defined in the previous section).

    The main feature of the inequality~\eqref{eq:1143} is that, while the probability $p$ on the right-hand side is only assumed to be larger than the critical probability $p_{c, \mathrm{site}}(d)$, the probability $p_0$ on the left-hand side is very close to $1$. In particular $p_0$ as defined in~\eqref{def.p0thm1.3} is sufficiently close to $1$ that we can apply the stochastic domination inequality of Section~\ref{section5.1} and Wells' inequality to the term on the left-hand side. This will be the purpose of Section~\ref{subsec5.2.3} and we for now focus on the proof of~\eqref{eq:1143}.

\begin{figure}
\begin{center}
\includegraphics[width=6cm]{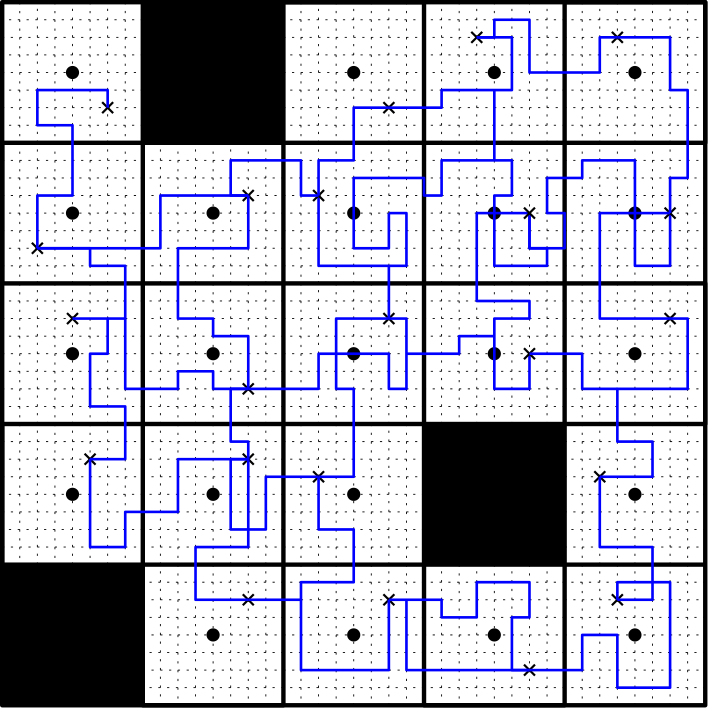} \hspace{15mm}
\includegraphics[width=8cm]{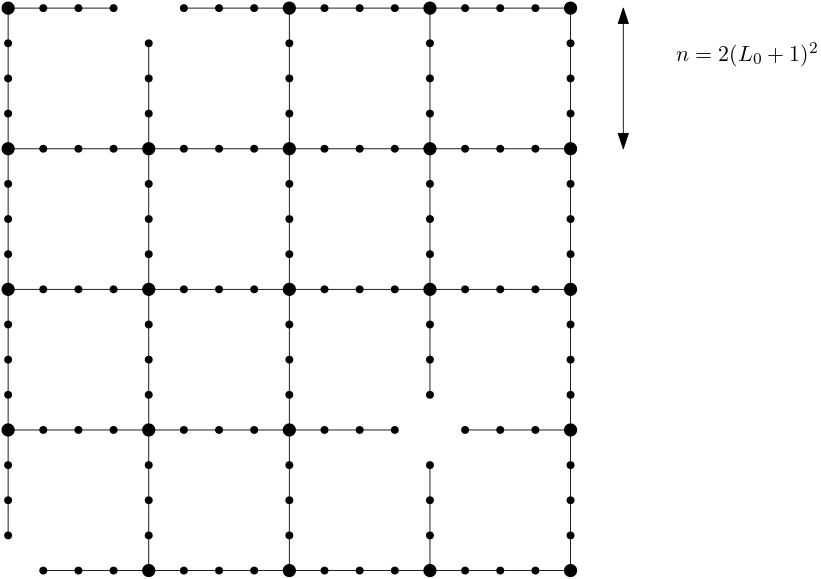}
\caption{The picture on the left represents the percolation configuration $r_1$ constructed by taking the union $\mathcal{L}$ of the paths connecting the selected vertices (represented by a cross) in each of the good boxes. The picture on the right represents the percolation configuration $r_2$ on the extended lattice $\Zd_n$.} \label{figure.constructionL}
\end{center}
\end{figure}

\medskip

To prove the inequality~\eqref{eq:1143}, we will work in finite volume. To this end, we fix a (large) integer $R \in \N$ and denote by $\Lambda := \Lambda_{(2L_0 + 1)R} \subseteq \Zd$ and $\Lambda^n  := \Lambda_{R}^n \subseteq \Zd_n$. These two boxes are represented in Figures~\ref{figure.partition} and~\ref{figure.constructionL} with the value $R = 2$. We choose $R$ sufficiently large so that $x ,x_0 , x_1 \in \Lambda$. We will prove the finite-volume inequality
\begin{equation} \label{eq:1143finitevol}
        \E_{p_0} \left[ \langle \cos( \theta_0 - \theta_{[x_1]} ) \rangle_{\mu_{\Lambda^n,  \beta/2^{2d} , \beta /2^{2d} , r}} \right] \leq \E_{p} \left[ \langle \cos( \theta_{x_0} - \theta_{x_1} ) \rangle_{\mu_{\Lambda , \beta , r}} | E_{L_0 , x} \right],
\end{equation}
The inequality~\eqref{eq:1143} is then deduced from~\eqref{eq:1143finitevol} by taking the limit $R \to \infty$ (and applying the monotone convergence theorem) in both sides of the inequality~\eqref{eq:1143finitevol}.

\medskip

To prove~\eqref{eq:1143finitevol}, we consider a realization $(r_y)_{y \in \Zd}$ of an i.i.d. site percolation percolation of probability $p$ conditioned on the event $E_{L,x}$ in the box $\Lambda$. In each box of the form $(z + \Lambda_{L_0})$ with $z \in (2L_0 + 1) \Zd \cap \Lambda$ which is good, we select a vertex $c_z$ such that $c_z \in \mathcal{C}(z + \Lambda_{L_0})$. We impose two restrictions on this selection: if $\Lambda_{L_0}$ is a good box then $c_0 = x_0$ and if $[x_1] + \Lambda_{L_0}$ is a good box, then $c_{[x_1]} = x_1$. The other vertices may be selected arbitrarily. Note that the definition of the event $E_{L,x}$ and of a good box together with the properties $x_0 \in \Lambda_{L_0/2}$ and $x_1 \in [x_1] + \Lambda_{L_0/2}$ ensure that $x_0 \in \mathcal{C}(\Lambda_{L_0})$ and $x_1 \in \mathcal{C}([x_1] + \Lambda_{L_0})$.\footnote{At this stage, the condition~\eqref{requirementx0x1} is used as this property could fail if $x_0$ or $x_1$ were close to the boundary of the boxes $\Lambda_{L_0}$ and $[x_1] + \Lambda_{L_0}.$}

\medskip

If two boxes of the form $(z + \Lambda_{L_0})$ and $(z' + \Lambda_{L_0})$ are both good and neighbours in $\mathcal{G}_L$ (i.e., $|z - z'|_1 = 2L_0 + 1$), we then consider an open path $\ell_{zz'}$ which connects the two vertices $c_z$ and $c_{z'}$  and lies in the union $(z + \Lambda_{L_0}) \cup (z' + \Lambda_{L_0})$ (see Figure~\ref{figure.paths}). The existence of the path $\ell_{zz'}$ is guaranteed by the definition of good boxes and Remark~\ref{remark2.17}. Its length is upper bounded by the value $n = 2 \left| \Lambda_{L_0} \right|$ (that is, by the number of vertices in the union $(z + \Lambda_{L_0}) \cup (z' + \Lambda_{L_0})$). We denote by $\mathcal{L}$ the collection of all the paths of the form $\ell_{zz'}$, i.e.,
    \begin{multline*}
        \mathcal{L} := \left\{ y \in \Zd \, : \, \exists z , z' \in (2L_0+1) \Zd, ~\mbox{with} \, |z - z'|_1 = (2L_0+1) \mbox{ such that} \right. \\ \left.  \mbox{the boxes}~(z+ \Lambda_{L_0}) ~\mbox{and}~ (z' + \Lambda_{L_0}) ~\mbox{are good and } y \in \ell_{zz'}  \right\}.
    \end{multline*}
The set $\mathcal{L}$ is drawn in blue on the left-hand side of Figure~\ref{figure.constructionL}. We next define two percolation configurations, one on $\Zd$ and one on the extended lattice $\Zd_n$, as follows:
    \begin{itemize}
        \item The percolation configuration $r_{1 } \in \{ 0 , 1\}^{\Zd}$ defined by the identity $r_{1,y} = 1$ if $y \in \mathcal{L}$ and $r_{1,y} = 0$ otherwise.
        \item The percolation configuration $r_{2 } \in \{ 0 , 1\}^{\Zd_n}$ on the extended lattice $\Zd_n$ defined such that $r_{2,y} = 1$ if $y \in \Zd_n \setminus \Zd$, and $r_{2,y} = 1$ if $y \in \Zd$ and the box $(2L_0 + 1) y + \Lambda_{L_0}$ is good.
    \end{itemize}
    We refer to Figure~\ref{figure.constructionL} for a visual description of the percolation configurations $r_1$ and $r_2$. We next claim that the following inequality holds: for any realization of the percolation configuration $r \in E_{L_0,x}$,
    \begin{equation} \label{eq:1655}
        \langle \cos( \theta_0 - \theta_{[x_1]} ) \rangle_{\mu_{\Lambda^n, \beta /2^{2d}, \beta /2^{2d} , r_2}} \leq \langle \cos( \theta_{x_0} - \theta_{x_1} ) \rangle_{\mu_{\Lambda, \beta , r_1}}.
    \end{equation}
    \begin{figure}
    \centering
    \includegraphics[width=9cm]{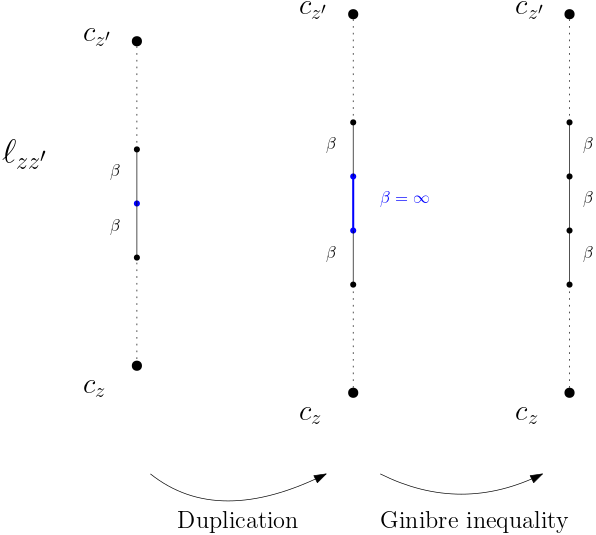}
    \caption{Extending the length of a path using the Ginibre inequality. The vertex in blue on the left is duplicated and an edge with inverse temperature equal to infinity is added. The inverse temperature is then reduced to the value $\beta$.} \label{fig.guidance1}
\end{figure}
    To prove the inequality~\eqref{eq:1655}, we transform the percolation configuration $r_1$ into (a rescaled version of) the percolation configuration $r_2$ by performing successive operations relying on the Ginibre inequality. We refer to Figures~\ref{fig.guidance1},~\ref{fig.guidance2} and~\ref{fig.guidance3} for guidance. We first fix three vertices $z , z', z'' \in (2L_0 + 1) \Zd \cap \Lambda$ such that $|z - z'|_1 = |z - z''|_1 = 2L_0 + 1$, fix a realization of the percolation configuration $r_1$, and perform the following three operations:
    \begin{itemize}
        \item \textit{Increasing the length of a path:} If the boxes $(z + \Lambda_{L_0})$ and $(z' + \Lambda_{L_0})$ are good, we consider the path $\ell_{zz'}$. By construction, the length of this path is smaller than $n$. If it is strictly smaller than $n$, then we increase it as follows: we consider a vertex of the path $\ell_{zz'}$, duplicate it and add an edge with an inverse temperature equal to infinity between the two vertices. This does not change the distribution of the spins. We then reduce the inverse temperature on this new edge from infinity to $\beta$. This operation is represented in Figure~\ref{fig.guidance1}. It increases the length of the path $\ell_{zz'}$ by $1$ and, by the Ginibre inequality, reduces the value of the two-point function. We next iterate this procedure until the path $\ell_{zz'}$ has length $n$.
        \item \textit{Separating two intersecting paths:} If the boxes $(z + \Lambda_{L_0})$, $(z' + \Lambda_{L_0})$ and $(z'' + \Lambda_{L_0})$ are good and the paths $\ell_{zz'}$ and $\ell_{zz''}$ have a nonempty intersection, we split the two paths as follows: we duplicate each vertex and edge in the intersection, divide the inverse temperature by $2$, and add new edges whose inverse temperatures are set to infinity between the vertices and the duplicated vertices (see Figure~\ref{fig.guidance2}). This does not affect the distribution of the spins. We then reduce the inverse temperature of the new edges from infinity to $0$. This operation is represented in Figure~\ref{fig.guidance2}. It reduces the value of the two-point function (by the Ginibre inequality), and splits the paths $\ell_{zz'}$ and $\ell_{zz''}$ into disjoint paths.
        \item \textit{Adding a line of spins next to bad boxes:} If the box $(z + \Lambda_{L_0})$ is good and the box $(z' + \Lambda_{L_0})$ is bad, then we add a line of spins as follows: we split the vertex $c_z$ into a line of $n$ vertices and set the inverse temperature of each edge of the line to infinity. We then reduce the inverse temperature on the new edges from infinity to $\beta$ (see Figure~\ref{fig.guidance3}). We remark that this operation does not affect the law of the original spins, it is added here to match with the formalism, and in particular, with the definition of percolation configurations on the extended graph.
    \end{itemize}
    \begin{figure}
    \centering
    \includegraphics[width=12cm]{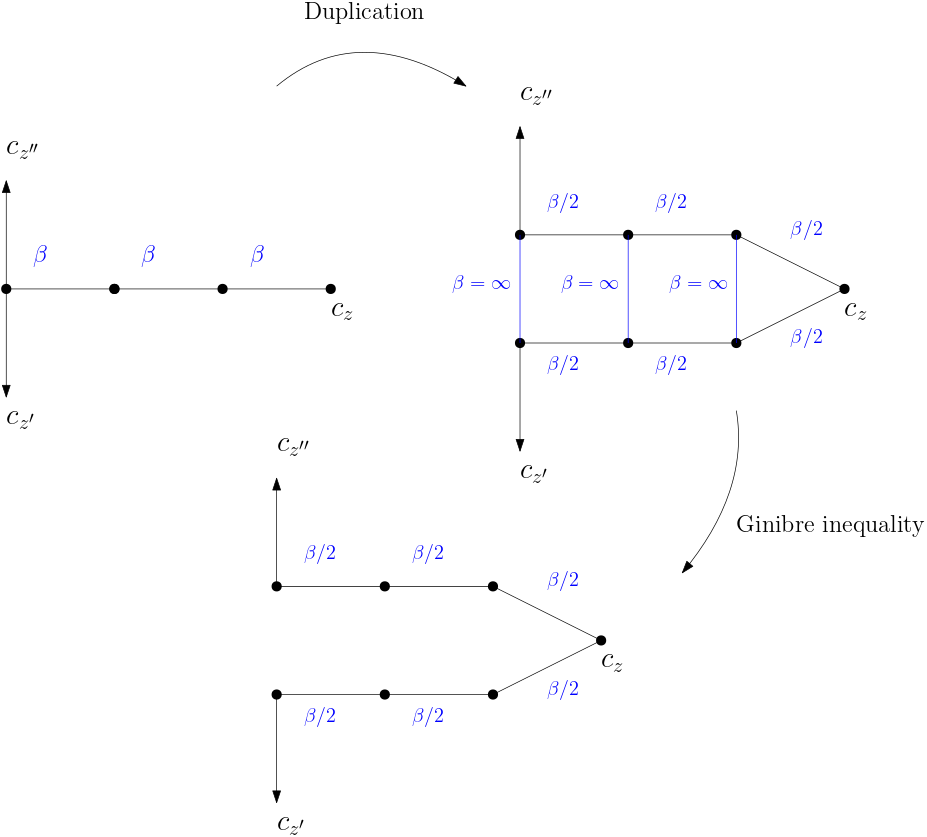}
    \caption{Splitting two intersecting paths using the Ginibre inequality. If two paths $\ell_{zz'}$ and $\ell_{zz''}$ have non empty intersection, we duplicate the vertices in the intersection and add new edges with inverse temperature equal to infinity (so as not to modify the distribution of the spins). We then reduce the value of the inverse temperature from infinity to $0$ to make the two paths disjoint.} \label{fig.guidance2}
    \end{figure}
    \begin{figure}
    \centering
    \includegraphics[width=10cm]{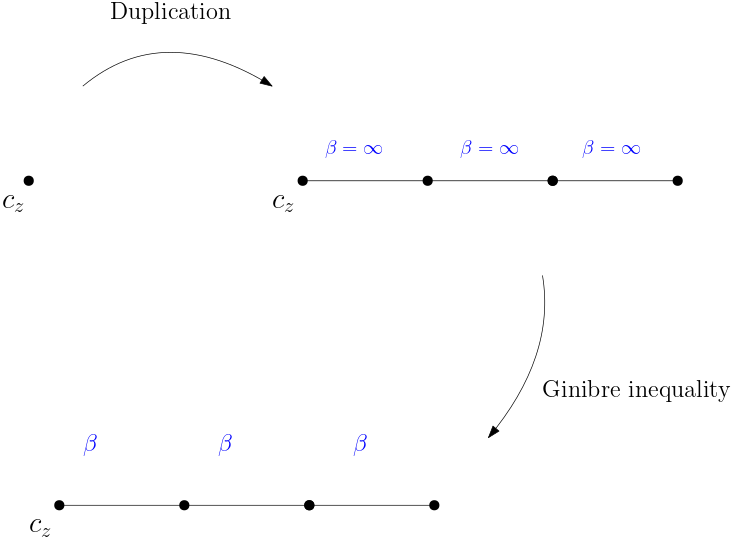}
    \caption{Adding a line of vertices using the Ginibre inequality. If a good box has a bad neighbour, we add a line of spins using the Ginibre inequality.} \label{fig.guidance3}
    \end{figure}
    Applying these three operations to all the triplets of neighbouring vertices $z , z', z'' \in (2L_0 + 1) \Zd$ transforms the $XY$ model defined on percolation configuration $r_1$ into an $XY$ model defined on the percolation configuration $r_2$: the vertex $c_z$ in the percolation configuration $r_1$ is mapped to the vertex $[z]  \in \Zd \subseteq \Zd_n$ in the percolation configuration $r_2$, and the path $\ell_{zz'}$ connecting the two vertices $c_z$ and $c_{z'}$ is mapped to the straight line connecting the vertices $[z]$ and $[z']$ in the percolation configuration $r_2$. 

    \medskip
    
    Additionally, the value of the two-point function is reduced by performing these operations (by the Ginibre inequality) and the inverse temperatures on the edges of the $XY$ model on the percolation configuration $r_2$ are always larger than $\beta /2^{2d}$ (as the second operation cannot be applied more than $2d$ times on the same path). The inequality~\eqref{eq:1655} follows from these two observations and the Ginibre inequality.

    \medskip

    We then complete the proof of the inequality~\eqref{eq:1143finitevol} using the inequality~\eqref{eq:1655}. By the definition of the percolation configuration $r_1$, we see that the inequality $r_{1 , y} \leq r_y$ (trivially) holds for any $y \in \Zd$. Combining this observation with the Ginibre inequality yields the bound
    \begin{equation*}
        \langle \cos( \theta_{x_0} - \theta_{x_1}  ) \rangle_{\mu_{\Lambda, \beta , r_1}} \leq \langle \cos( \theta_{x_0} - \theta_{x_1} ) \rangle_{\mu_{\Lambda, \beta , r}}.
    \end{equation*}
    Taking the conditional expectation on the event $E_{L_0, x}$ implies
    \begin{equation*}
        \E_{p} \left[ \langle \cos( \theta_{x_0} - \theta_{x_1}  ) \rangle_{\mu_{\Lambda,\beta , r_1}} \, | \, E_{L_0, x} \right] \leq \E_{p} \left[  \langle \cos( \theta_{x_0} - \theta_{x_1}  ) \rangle_{\mu_{\Lambda,\beta , r}} \, | \, E_{L_0, x} \right].
    \end{equation*}
    We finally observe that, since the law of the site percolation on the graph $\mathcal{G}_L$ stochastically dominates an i.i.d. site percolation of parameter $p_0$ (as established in Section~\ref{subsec5.2.1}), the law of the percolation configuration $r_{2}$ stochastically dominates the law of $\P_{p_0}$ (on the extended lattice $\Zd_n$). Combining this observation with the Ginibre inequality yields the bound
    \begin{equation*}
        \E_{p_0} \left[ \langle \cos( \theta_0 - \theta_{[x_1]} ) \rangle_{\mu_{\Lambda^n,\beta /2^{2d} , \beta /2^{2d} , r}} \right] \leq \E_{p} \left[ \langle \cos( \theta_0 - \theta_{[x_1]} ) \rangle_{\mu_{\Lambda^n,\beta /2^{2d} , \beta /2^{2d} , r_2}} \, | \, E_{L_0, x} \right].
    \end{equation*}
    A combination of the two previous displays with~\eqref{eq:1655} completes the proof of the inequality~\eqref{eq:1143finitevol}.

\subsubsection{Completing the proof of Theorem~\ref{proofKTsitegeneral}} \label{subsec5.2.3}

    In this final section, we show how to obtain Theorem~\ref{proofKTsitegeneral} from the inequality~\eqref{eq:1143}. We first simplify the conditional expectation on the right-hand side of~\eqref{eq:1143}. Using the inequality~\eqref{eq:Ginibreineq}, which implies that the two-point function $\langle \cos( \theta_{x_0} - \theta_{x_1} ) \rangle_{\mu_{\beta , r}}$ is always nonnegative, and the lower bound $\P \left[ E_{L_0,x} \right] \geq \theta(p)^2$, we have that 
    \begin{equation} \label{eq:4.17.5}
        \E_{p} \left[ \langle \cos( \theta_{x_0} - \theta_{x_1} ) \rangle_{\mu_{\beta , r}} \, | \, E_{L_0,x} \right] \leq  \theta(p)^{-2} \E_{p} \left[ \langle \cos( \theta_{x_0} - \theta_{x_1} ) \rangle_{\mu_{\beta , r}} \right].
    \end{equation}
    We then fix a large, finite domain $\Lambda \subseteq \Zd_n$ use the definition of the inverse temperature $\beta$ (which implies $\beta \geq \beta_1$ and $\beta \geq \beta_2$) and the Ginibre inequality to obtain
    \begin{equation*}
        \E_{p_0} \left[ \langle \cos( \theta_0 - \theta_{[x_1]} ) \rangle_{\mu_{\Lambda, \beta_1/2^{2d} , \beta_2 /2^{2d} , r}} \right] \leq \E_{p_0} \left[ \langle \cos( \theta_0 - \theta_{[x_1]} ) \rangle_{\mu_{\Lambda,\beta/2^{2d} , \beta /2^{2d} , r}} \right].
    \end{equation*}
    We apply Proposition~\ref{prop4.1extended} to obtain the bound
    \begin{equation*}
         \left\langle \cos( \theta_0 - \theta_{[x_1]} ) \right\rangle_{\mu_{\Lambda, \beta_{1}/2^{2d} , \beta_{2} /2^{2d}  , \kappa}} \leq \E_{p_0} \left[ \langle \cos( \theta_0 - \theta_{[x_1]} ) \rangle_{\mu_{\Lambda, \beta_{1}/2^{2d} , \beta_{2} /2^{2d} , r}} \right] ,
    \end{equation*}
    where the collection of measures $\kappa$ refers to the one introduced in~\eqref{def.measurekappa} with the value $\bar p = 1/2$. Applying Wells' inequality (Proposition~\ref{prop.Wells}) with the value $a = 1/2 (= \min(1/2, \bar p))$, we obtain
    \begin{equation*}
         \left\langle \cos( \theta_0 - \theta_{[x_1]} ) \right\rangle_{\mu_{\Lambda,  \beta_{1}/2^{2d + 2} , \beta_{2} /2^{2d}}} \leq \left\langle \cos( \theta_0 - \theta_{[x_1]} ) \right\rangle_{\mu_{\Lambda, \beta_{1}/2^{2d} , \beta_{2} /2^{2d} , \kappa}}.
    \end{equation*}
    Combining the two previous inequalities, we obtain that, for any finite set $\Lambda \subseteq \Zd_n$,
    \begin{equation} \label{eq:13599}
        \left\langle \cos( \theta_0 - \theta_{[x_1]} ) \right\rangle_{\mu_{\Lambda,  \beta_{1}/2^{2d + 2} , \beta_{2} / 2^{2d}}} \leq \E_{p_0} \left[ \langle \cos( \theta_0 - \theta_{[x_1]} ) \rangle_{\mu_{\Lambda, \beta_1/2^{2d}, \beta_2/2^{2d}, r}} \right].
    \end{equation}
    As both sides of the inequality~\eqref{eq:13599} are increasing in the set $\Lambda$, we may take the limit $\Lambda \uparrow \Zd_n$ (and apply the monotone convergence theorem for term on the right-hand side) to obtain
    \begin{equation*} \label{eq:1359}
        \left\langle \cos( \theta_0 - \theta_{[x_1]} ) \right\rangle_{\mu_{\beta_{1}/2^{2d + 2} , \beta_{2}/2^{2d} }} \leq \E_{p_0} \left[ \langle \cos( \theta_0 - \theta_{[x_1]} ) \rangle_{\mu_{\beta_1/2^{2d}, \beta_2/2^{2d}, r}} \right].
    \end{equation*}
    Combining~\eqref{eq:1143},~\eqref{eq:4.17.5} and~\eqref{eq:13599} with the identity~\eqref{eq:identitytransinv}, we have obtained
    \begin{align*}
        \left\langle \cos( \theta_0 - \theta_{[x_1]} ) \right\rangle_{\mu_{\beta_{1}/2^{2d + 2} , \beta_{2}/2^{2d}}} & \leq \theta(p)^{-2} \E_{p} \left[ \langle \cos( \theta_{x_0} - \theta_{x_1} ) \rangle_{\mu_{\beta , r}} \right] \\
        & = \theta(p)^{-2} \E_{p} \left[ \langle \cos( \theta_{0} - \theta_{x} ) \rangle_{\mu_{\beta , r}} \right].
    \end{align*}
    Since, by the definitions of the inverse temperatures $\beta_1, \beta_2$ and Proposition~\ref{prop.phasetransitionextended}, the term on the right-hand side decays polynomially fast in $|[x_1]|$ in dimension $d = 2$ and remains bounded away from $0$ in dimension $d \geq 3$, the same properties hold for the term on the right-hand side. This completes the proof of Theorem~\ref{proofKTsitegeneral}.

\subsection{Proof of Corollary~\ref{c.AlmostSure}} \label{subsection5.3}
This section is devoted to the proof of Corollary~\ref{c.AlmostSure}, which is obtained by combining the result of Theorem~\ref{proofKTsitegeneral} with the pointwise ergodic theorem. In order to apply the ergodic theorem, we will need to translate percolation configurations and we thus introduce the following notation: for any pair of vertices $x , y \in \Zd$ and any percolation configuration $r \in \{ 0 , 1 \}^{\Zd}$, we define $(\tau_x r)_y = r_{y + x}.$

\begin{proof}[Proof of Corollary~\ref{c.AlmostSure}]
    Fix an integer $m \geq 1$. For a percolation configuration $r \in \{ 0 , 1 \}^{\Zd}$, we define the random variable
    \begin{equation*}
        X(r)  := \sum_{y\in \Lambda_m}  \langle \cos(\theta_0 - \theta_y) \rangle_{\mu_{\beta, r}}.
    \end{equation*}
    We note that the following identity holds: for any $x \in \Zd$,
    \begin{equation*}
        X( \tau_x r)  := \sum_{y\in x + \Lambda_m}  \langle \cos(\theta_x - \theta_y) \rangle_{\mu_{\beta, r}}.
    \end{equation*}
    Using the ergodicity of Bernoulli percolation and applying the ergodic theorem (see the multiparameter version of~\cite[Theorem VIII.6.9]{dunford1988linear}), we have, for any $p \in [0 , 1]$,
    \begin{equation*}
        \lim_{R \to \infty} \frac{1}{\left| \Lambda_R \right|}\sum_{x \in \Lambda_R}  X( \tau_x r) = \E_p\left[ X \right] ~~\P_p-\mbox{almost-surely}.
    \end{equation*}
    A combination of the two previous displays implies the identity
    \begin{equation*}
            \lim_{R \to \infty}  \frac{1}{\left| \Lambda_R \right|}\sum_{x \in \Lambda_R}   \sum_{y\in x + \Lambda_m}  \langle \cos(\theta_x - \theta_y) \rangle_{\mu_{\beta, r}} = \E_p\left[ X \right] ~~\P_p-\mbox{almost-surely}.
    \end{equation*}
    Theorem~\ref{proofKTsitegeneral} and the first item of Remark~\ref{remark1.4} provide a lower bound on the expectation of $\E_p\left[ X \right]$ for any $p > p_{c , \mathrm{site}}(d)$. Combining this observation with the previous display (and taking an intersection of events of probability $1$ over all the integers $m \geq 1$) completes the proof of Corollary~\ref{c.AlmostSure}.
\end{proof}

\section{Phase transitions for the $XY$ model on Poisson-Voronoi}

The purpose of this section is to prove Theorem \ref{th.Vor}.
The proof will follow closely our proof of Theorem~\ref{proofKTsitegeneral} in the above Section \ref{section4.2}. In particular, it will rely on the same renormalization argument. 

\smallskip

One standard difficulty one faces when dealing with statistical physics on a Poisson-Voronoi graph is the fact the shape of the  Voronoi cell $\Vor(x)$ of a point $x$ in the Poisson Point Process $P$, may potentially depend on the behaviour of the Poisson Point process $P$ at very remote distance from $x$. (Recall the definitions of the Voronoi cell $\Vor(x)$ from~\eqref{e.Vorx}).  For example if in $d=2$,  $P\cap B(0,R) = (\Z\times \{0\} ) \cap B(0,R)$, then the Voronoi cell attached to the origin will be a long vertical rectangle until it will eventually feel other points of $P$.

We will need the following very useful Lemma which quantifies the fact that the shape of Voronoi cells in a region $A$ only depends locally on the Poisson Point process $P$ when $P$ is sufficiently {\em dense} around $A$. 
Such a Lemma first appeared in  Lemma 3.2. in \cite{bollobas2006critical}. We rely below an a version highly inspired from Lemma 1.1 in \cite{tassion2016crossing}.  In what follows, we will denote with a slight abuse of notation $\Lambda_R:=[-R,R]^d \subset \R^d$ (the slight abuse of notations here is that in the previous sections, the same notations was referring to boxes in $\Z^d$). 

\begin{lemma}[See similar statements in \cite{bollobas2006critical,tassion2016crossing}]\label{l.Poisson}
Let $P$ be a Poisson point process in $\R^d$ (with intensity 1, say).
For any $R>0$ and any $z\in \R^d$, let $\calE_R(z)$ be the event that $z+\Lambda_{5R} \subset \bigcup_{x\in P} B(x,\frac R {10})$. (In other words, any point in $z+\Lambda_{5R}$ has at least one point in $P$ at distance less than $R/10$.
Notice also that the event $\calE_R(z)$ is measurable w.r.t $P\cap (z+\Lambda_{6R})$). 

In what follows, we will consider the case $z=0$ and call the event $\calE_R(0)$ simply by $\calE_R$ (for other points, the same applies by translation invariance). 
If the event $\calE_R$ is satisfied then the following hold:
\bi
\item[i)] The shape of the Voronoi cell $\Vor(x,P)$ of each point $x\in P\cap \Lambda_{4R}$ is measurable w.r.t $P\cap \Lambda_{6R}$. 
\item[ii)] For any pair of points $x\neq y$ in  $P \cap \Lambda_{3R}$ (whose cardinality is larger than $\Omega(R^d)$ since we assume that the event $\calE_R$ holds), we can find a finite  path $l_{x,y}:= \{ x_0=x,x_1,x_2,\ldots, x_M=y \}$ so that for any $i\leq M$, $x_i\in P \cap \Lambda_{4R}$ and for each $i<M$, $x_i \sim x_{i+1}$ in the Poisson-Voronoi graph, i.e. 
\begin{align*}\label{}
\Vor(x_i,P) \cap \Vor(x_{i+1},P) \neq \emptyset\,.
\end{align*}
If several such paths exist, we may fix one among the shortest possible such paths according to any deterministic prescribed rule and in a way which is measurable w.r.t $P \cap \Lambda_{6R}$.

By construction, the length $M=M_{x,y}$ of such paths is bounded above by $\#\{ P \cap \Lambda_{4R} \}$. 
\ei
\end{lemma}

\ni
{\em Sketch of proof.} (We refer to the proof of Lemma 1.1 in \cite{tassion2016crossing}, where the setting is very close to ours for the details). 

For item $i)$, we assume $\calE_R$ holds. Notice first that for each $x\in P\cap \Lambda_{4R}$, the radius of the Voronoi cell $\Vor(x,P)$ measured from $x$ is necessarily smaller than $R/10$. Indeed suppose we find a point $\tilde z\in \Vor(x,P)$ at (Euclidean) distance larger than $R/10$ from $x$. Since $\Vor(x,P)$ is convex, we can find a point $z\in \Vor(x,P)$ whose Euclidean distance to $x$ is in the open interval $(\frac R {10}, R)$. 
In particular this point $z$ is inside $\Lambda_{5R}$. Since the event $\calE_R$ is assumed to hold, this means there exists $y\in P$ at distance less than $R/10$ from $z$ which gives a contradiction. 

Therefore to compute the shape of $\Vor(x,P)$, it is enough to compute for any specific point $z$ in the $R/10$ neighbourhood of $x$, whether this point $z$ is in $\Vor(x,P)$ or not. Notice for this that the event $z\in \Vor(x,P)$ is measurable w.r.t $P\cap \bar B(z,|z-x|)$ (the closed Euclidean ball of radius $|z-x|$).

This shows that when $\calE_R$ is satisfied, the shape of $\Vor(x,P)$ is measurable w.r.t $P\cap \bar B(x, \frac {R} 5)$ and this ends the proof if $i)$ since $\calE_R$ is itself measurable w.r.t $P\cap \Lambda_{6R}$. 
\medskip

For item $ii)$. Pick any pair of such points $x\neq y \in P \cap \Lambda_{3R}$ (which do exist on the event $\calE_R$). Let $L_{x,y}$ be the segment $\{(1-t)x +t y, t\in [0,1]\}$ joining $x$ and $y$. For any $1\leq t \leq 1$, we will denote by $u_t$ the point $(1-t) x + t y$. By convexity, this segment stays inside $\Lambda_{3R}$. We will find a connecting path $x_0,\ldots, x_M$ by moving continuously along $L_{x,y}$. Let $x_0:=x$ as follows. Define $t_1:=\sup \{t \in [0,1], u_t = (1-t)x+ ty \in \Vor(x,P)$. This means $u_{t_1}$ it on $\p \Vor(x,P)$. It is an easy and classical fact that the probability that such a line $L_{x,y}$ intersects a "corner" of the Voronoi tiling is zero. Similarly the intersection 
\begin{align*}\label{}
L_{x,y} \cap \bigcup_{w\in P} \p \Vor(w,P) 
\end{align*}
is a.s. a finite set of isolated points. 
This implies that a.s. $t_1>0$ and $u_{t_1}$ belongs to exactly two Voronoi cells. We let $x_1\in P$ be the center of the newly visited Voronoi cell and define $t_2:= \sup\{ t\in[t_1,1], u_t \in \Vor(x_1,P) \}$. We have again that a.s. $t_2>t_1$ and $u_{t_2}$ belongs to exactly two Voronoi cells. This defines our third point $x_3\in P$ and so on and so forth. This algorithm terminates in finite time and produces a connecting path $x=x_0\sim x_1 \ldots \sim x_M=y$. 
Exactly as in item $i)$ each cell $\Vor(x_i,P)$ is of diameter less than $R/10$ and intersects the line $L_{x,y}\subset \Lambda_{3R}$. This ensures that the centers $x_i$ belong to $P\cap \Lambda_{4R}$ as desired. \qed

\medskip

In order to control the $XY$ model on the renormalized graph, we will also need to have a good control on the length of the minimal path $l_{x,y}$ connecting $x$ to $y$. As stated in the above Lemma, this path is of length smaller than $\#\{ P \cap \Lambda_{4R}\}$. This is why we introduce for any $z\in \R^d$ the event:
\begin{align}\label{e.eventF}
\calF_{R,H}(z):= \{ \# \{P \cap (z+ \Lambda_R)\} \leq H \,  R^d  \}\,.
\end{align}
 
 Since $P$ is a Poisson Point process of intensity 1 and since $|\Lambda_R|=(2R)^d$, there exists $H$ sufficiently large so that for sufficiently large $R$,  
 \begin{align*}\label{}
\Pb{\calF_{R,H}(z)^c}& \leq  e^{- R^d}\,.
\end{align*}

Furthermore, exactly as in Lemma 1.1 from \cite{tassion2016crossing}, using a suitable finite covering of $\Lambda_{4R}$ by open balls of radius $R/10$, we have that for some constant $0<c<1$, 
\begin{align*}\label{}
\Pb{\calE_R^c} \leq e^{-c R^d}
\end{align*}

As such, if $R$ is chosen large enough, we have for any $z\in \R^d$
\begin{align}\label{e.good}
\Pb{\calE_{R}(z) \cap \calF_{R,H}(z)} \geq 1 - 2 e^{-c R^d}\,.
\end{align}

Finally, when the strength of the interaction between $i,j \in P$ depends on the shapes of $\Vor(i,P)$ and $\Vor(j,P)$ (as it is the case for the interaction strengths $F_2$ and $F_3$ in~\eqref{e.choices}), we need to get some control on the lack of uniform ellipticity of such an interaction. To achieve this, for any $z\in \R^d$ and any $R>0$, let us introduce the following random variable:
\begin{align*}\label{}
 K(z,R):= \inf_{\substack{i\neq j \in P\cap \Lambda_{4R}(z) \\ \Vor(i,P) \cap \Vor(j,P) \neq \emptyset}} \min \Big\{ 
f(|\Vor(i,P)|) f(|\Vor(j,P)|),  \lambda^{d-1}(\Vor(i,P) \cap  \Vor(j,P))
\Big\}\,, 
\end{align*}
where $f: \R_+ \to \R_+$ is the strictly increasing function in the definition of $F_2$ in~\eqref{e.choices}.

We make the following observations:
\bnum
\item On the event $\calE_R(z)$, the random variable $K(z,R)$ is measurable w.r.t $P \cap \Lambda(z,6R)$. 
\item Conditioned on the event $\calE_R(z)$, the probability that $K(z,R)=0$ is zero. This is because $f(v)>0$ for any $v>0$ and because generically when two Voronoi cells intersect, they intersect along a non-degenerate face of co-dimension 1 (otherwise, this would involve finding $d+2$ points in $P$ which would belong to a same sphere in $\R^d$). 
\enum

Using this observation, we have that for any choice of strictly increasing $f: \R_+ \to \R_+$, and for any scale $R>0$ large enough, there exists a small enough $\delta_R= \delta_R(d)>0$ s.t. 
\begin{align}\label{e.good2}
\Pb{\calE_{R}(z) \cap \calF_{R,H}(z) \cap \{K(z,R) \geq \delta_R(d) \}} \geq 1 - 4 e^{-c R^d}\,,
\end{align}
where $c>0$ is the same as in~\eqref{e.good}. 

From the above estimate, we shall also need some fixed large enough scale $R_0 \geq 10$ so that 
\begin{align}\label{e.good3}
\Pb{\calE_{R_0}(z) \cap \calF_{R_0,H}(z) \cap \{K(z,R_0) \geq \delta_{R_0}(d) \}} \geq \tfrac 9 {10}\,. 
\end{align}
(This fixed scale $R_0$ will be used below to have a control on the inverse temperature $\beta_1$ which shall not depend on the way to tune our renormalization argument, a key issue in the previous Section \ref{section4.2}).

\smallskip

We now have all the necessary ingredients to run a renormalization argument similar to our analysis for supercritical percolation clusters for any $p>p_{c, \mathrm{site}}(d)$ in Section \ref{section4.2}. 
\medskip

\ni
{\em Proof of Theorem \ref{th.Vor}.}

We will follow the same strategy as for the proof of Theorem \ref{proofKTsitegeneral} in Section \ref{section4.2}
and shall only highlight the main modifications. 
\bnum
\item First, we introduce a tiling of $\R^d$ very similar to the tiling used on $\Z^d$ in Section \ref{section4.2}. We shall use here (using the same notations with a slight abuse of notations) the graph $\mathcal{G}_L$ made of the boxes $z+\Lambda_L$, with 
\begin{align*}\label{}
z\in (2L) \Z^d\,.
\end{align*}

\item As in Section \ref{section4.2}, we set \footnote{For a technical reason (explained below), we do not choose $A_d := 2d$ as in Section~\ref{section4.2} but replace it by a larger constant which depends only on the dimension $d$.}
\begin{align}\label{e.beta1R0}
\begin{cases}
& \beta_1  := 2^{A_d+1}\delta_{R_0}^{-1}(2)\,  \beta_{1, BKT}  \text{   in dimension $d = 2$} \\
&  \beta_1 := 2^{A_d+1} \,\delta_{R_0}^{-1}(d)\, \beta_{1, c}(d)  \text{ in dimension $d \geq 3$}
\end{cases}
\end{align}
We also set  
    \begin{equation} \label{def.p0twice}
            p_0 := \max\left( \frac{1}{1 + \exp \left( - 2 d \beta_{1} \right)} , \tfrac 9 {10} \right).
    \end{equation}

\item We say that the box $z+\Lambda_L$ is ``good'' if the following four events are satisfied, first we ask $\calE_L(z)$, $\calF_{L,H}(z)$ and $\{K(z,L) \geq \delta_L \}$ to be satisfied. (N.B. in the case of the strength 1 nearest-neighbour interaction, i.e. $F_1$ in~\eqref{e.choices}, we only need the first two events). 
The last event we need is the following one:
\begin{align}\label{e.Jr}
J(z,L):= \{ \exists w \in z+ \Lambda_{L/4}, \text{ s.t. }  \calE_{R_0}(w) \cap \calF_{R_0,H}(w) \cap \{K(w,R_0) \geq \delta_{R_0}(d) \}  \text{  holds } \} \,.
\end{align}
Since by~\eqref{e.good3}, the probability of $ \calE_{R_0}(w) \cap \calF_{R_0,H}(w) \cap \{K(w,R_0) \geq \delta_{R_0}(d) \}$ is greater than $9/10$ and since two points $w,w'$ at distance $12R_0$ from each other, these events are independent, we readily obtain the existence of a $\tilde c>0$ so that for any $z\in \R^d$, 
\begin{align*}\label{}
\Pb{J(z,L)} \geq 1 - e^{-\tilde c L^d}\,.
\end{align*}
Combining~\eqref{e.good2} with the above estimate, we get for any $z\in \R^d$,
\begin{align}\label{e.verygood}
\Pb{z+\Lambda_L \text{ is ``good'' }} = \Pb{\calE_L(z)\cap \calF_{L,H}(z) \cap \{K(z,L) \geq \delta_L \} \cap J(z,L) } \geq  1 - 4 e^{-c L^d} - e^{-\tilde c L^d}\,.
\end{align}

\item For any $x\in \R^d$ at distance, say at least $L$ from the origin, choose as in Section \ref{section4.2} points $x_{0,L}$ and $x_{1,L}$ in $\R^d$ such that
   \begin{equation} \label{requirement2}
        x_{0,L} \in \Lambda_{L/2}, ~~x_{1 , L} \in \bigcup_{z \in (2L) \Zd} (z + \Lambda_{L/2}) ~~\mbox{and}~~x = x_{1 , L} - x_{0 , L}.
    \end{equation}

Similarly as in Section \ref{section4.2}, let  
\begin{align*}\label{}
E_{L,x} := 
\{K(x_{0,L},R_0) \geq \delta_{R_0}(d) \}  \cap \{K(x_{1,L},R_0) \geq \delta_{R_0}(d) \}\,.
\end{align*}
Since both of these events have probability greater than $\tfrac 9 {10}$, this event is of probability greater than~$8/10$. (N.B. $8/10$ plays role of $\theta(p)^2$ in the proof of Section \ref{section4.2}). 
The event $E_{L,x}$ ensures that the closest points in $P$ of $x_{0,L}$ and $x_{1,L}$, i.e. $\hat{x_{0,L}}$ and $\hat{x_{1,L}}$, are at distance of order $R_0$ from $x_{0,L}$ and $x_{1,L}$ and also that the geometry around these points is not too degenerate (the coupling constants around $\hat{x_{0,L}}$ and $\hat{x_{1,L}}$ are lower bounded by $\delta_{R_0}$). 
If we denote by $z_{1,L}$ the point in $(2L) \Z^d$ such that $x_{1,L}\in z_{1,L}+\Lambda_{L/2}$, we further stress that the probability of the event 
\begin{align*}\label{}
F_{L,x}:= E_{L,x}  \cap \{ \Lambda_L \text{ is ``good'' } \} \cap \{ z_{1,L} + \Lambda_L \text{ is ``good'' } \}
\end{align*}
is greater than $\frac 7 {10}$ if $L$ is chosen large enough.

\item Conditioned on the event $F_{L,x}$, by construction the field of good boxes in the covering graph $\mathcal{G}_L$ is a $10$-dependent (inhomogeneous) random field.  

Notice further that using the estimate~\eqref{e.verygood} together with the fact that the event $F_{L,x}$ is local with $\Pb{F_{L,x}}\geq \tfrac 7 {10}$, we obtain that for $L$ sufficiently large, conditioned on $F_{L,x}$, the probability of each box of the grid $\calG_L$ to be good is at least $1- e^{-\hat c L^d}$ for some $\hat c>0$. 
Therefore, recalling our choice of $p_0$ from~\eqref{def.p0twice} and using the stochastic domination bound in Theorem \ref{th.Liggett} (due to~\cite{liggett1997domination}), we may pick a sufficientlly large scale $L_0\in \N$ so that, conditioned on $F_{x,L_0}$,  the field of good boxes on the renormalized graph $\mathcal{G}_{L_0}$ stochastically dominates an i.i.d site percolation of parameter $p_0$.

\item For the good boxes $\Lambda_{L_0}$ and $z_{1,L_0}+ \Lambda_{L_0}$, we choose the points $\hat{x_{0,L_0}}$ and $\hat{x_{1,L_0}}$ in $P$ as centers. For any other good boxes $z+\Lambda_{L_0}$ in $\mathcal{G}_{L_0}$, we choose in an arbitrary way a center $\hat z$ in $z+\Lambda_{\frac 1 4 L_0}$ which is such that the event 
\begin{align*}\label{}
\calE_{R_0}(z) \cap \calF_{R_0,H}(z) \cap \{K(z,R_0) \geq \delta_{R_0}(d) \}  \,\,\, \text{  holds.}
\end{align*}
(N.B. These centers exist in good boxes thanks to the event $J(z,L_0)$ from~\eqref{e.Jr}). 
\item Using the definition of the event $\calF(R,H)$ in~\eqref{e.eventF} as well as Lemma \ref{l.Poisson}, we may connect the centers $\hat z$ and $\hat z'$ of two neighbouring good boxes with a path of Voronoi cells touching each other inside $(z + \Lambda_{4L_0}) \cup (z' + \Lambda_{4L_0})$ using less than $2H L_0^d$ cells.\footnote{This property is the reason for the introduction of the constant $A_d$. In the percolation setup, a path connecting two neighbouring good boxes lies in the union $(z + \Lambda_L) \cup (z' + \Lambda_L)$, which is a smaller set.}   

We thus set $n :=2\, H \, L_0^d$ as in the case of $XY$ model on top of supercritical percolation. 

Also thanks to the events $J(z,L_0)$ and $J(z',L_0)$, the coupling constants around each of the above centers are lower bounded by $\delta_{R_0}(d)$. 

\item The rest of the proof is identical except that when using Proposition~\ref{prop.phasetransitionextended}, we need to consider instead the following inverse temperatures (recall our choices of $\beta_1$ have already been fixed in~\eqref{e.beta1R0}). 
\begin{align*}\label{}
\begin{cases}
& \beta_2  := 2^{A_d} \, \delta_{L_0}^{-1}(d)\, \beta_{2, BKT}(n) \text{   in dimension $d = 2$} \\
&  \beta_2 := 2^{A_d} \, \delta_{L_0}^{-1}(d)\,  \beta_{2, c}(n, d)  \text{ in dimension $d \geq 3$}.
\end{cases}
\end{align*}

\enum

\qed

\appendix

\section{Phase transitions for the $XY$ model on the extended lattice $\mathbb{Z}^d_n$} \label{AppendixA}

In this section, we prove Proposition~\ref{propApp.phasetransitionextended} which establishes the existence of a BKT phase transition in two dimensions and long-range order in dimensions three and higher for the $XY$ model on the extended lattice. The proofs are adaptations of the corresponding proofs for the $XY$ model, and especially of the article of van Engelenburg and Lis~\cite{van2023elementary} for the BKT phase transition in two dimensions and of the article of the second author and Spencer~\cite{garban2022continuous} for the long-range order in dimensions three and higher.

\subsection{The BKT phase transition on the extended lattice $\Zd_n$}

In this section, we establish the existence of a BKT phase transition for the $XY$ model on the extended lattice. The statement is recalled below. The proof is a minor adaptation of the argument of~\cite{van2023elementary}.

\begin{proposition}[BKT phase transition for the $XY$ model on the extended lattice $\Z^2_n$] \label{propApp.phasetransitionextended}
In dimension $d = 2$, there exists an inverse temperature $\beta_{1, BKT} < \infty$ such that, for every $n \in \N$, there exists an inverse temperature $\beta_{2, BKT}(n) < \infty$ such that, for any $\beta_1 \geq \beta_{1, BKT}$ and any $\beta_2 \geq \beta_{2, BKT}(n)$,
    \begin{equation*}
        \left\langle \cos(\theta_0 - \theta_x) \right\rangle_{\mu_{\beta_1 , \beta_2 }} ~\mbox{decays polynomially fast as } |x| \to \infty. 
    \end{equation*}
\end{proposition}

\begin{remark}
As in~\cite{van2023elementary}, the argument gives the lower bound, for some constant $c := c(n) > 0$,
\begin{equation*}
    \langle \cos \left( \theta_0 - \theta_x \right) \rangle_{\mu_{\beta_1, \beta_2}} \geq \frac{c}{|x|}.
\end{equation*}

\end{remark}

We follow closely the proof of~\cite{van2023elementary} and break the argument into several steps. 
We first collect in Proposition~\ref{prop.corrineq} some inequalities for the $XY$ model.
Following~\cite[Section 6]{van2023elementary}, we observe that they imply the following dichotomy regarding the behaviour of the two-point function: it decays either exponentially or polynomially fast (see Proposition~\ref{prop.dichotomy}).

We then prove that it is possible to choose the inverse temperatures $\beta_1$ and $\beta_2$ such that the two-point function does not decay exponentially fast. To this end, we follow the strategy of van Engelenburg and Lis~\cite{van2023elementary} who proceeded in two steps: they first showed that the two-point function of the $XY$ model is larger than the two-point function of a $XY$ model defined on a certain triangulation, they then showed that it is possible to choose the inverse temperature sufficiently large so that the two-point function of the $XY$ model on the triangulation is not summable. A combination of these two results implies that the two-point function cannot decay exponentially fast.

Section~\ref{sectionineqXY} collects the inequalities for the $XY$ model used in the proof as well as the dichotomy result for the decay of the two-point function. The first step of the proof of~\cite{van2023elementary} is then reproduced in Section~\ref{sectiontriangulation}, the second step can be found in Section~\ref{sectionheightfunction}.

\subsubsection{Inequalities for the two-point function of the $XY$ model} \label{sectionineqXY}

We first collect some well-known inequalities about the two-point function of the $XY$ model. In the following statement, we consider the $XY$ model introduced in Definition~\ref{def.triangheterogeneousXY}. We refer to~\cite[Appendix A]{van2023elementary} and the references below for a proof. These inequalities are not restricted to the two-dimensional case and are stated for subsets of $\Zd_n$ in arbitrary dimension. 

\begin{proposition}[Inequalities for the two-point function of the $XY$ model] \label{prop.corrineq}
For any dimension $d \geq 2$, any integer $n \in \N$ and any finite subset $\Lambda \subseteq \Zd_n$, the following inequalities hold:
\begin{itemize}
    \item For any $x , y , z \in \Lambda$,
        \begin{equation} \label{eq:twopointbound}
            \langle \cos( \theta_x - \theta_y) \rangle_{\mu_{\Lambda, \beta_1 , \beta_2}} \geq \langle \cos( \theta_x - \theta_z) \rangle_{\mu_{\Lambda, \beta_1 , \beta_2}}  \langle \cos( \theta_z - \theta_y) \rangle_{\mu_{\Lambda, \beta_1 , \beta_2}} 
        \end{equation}
    \item Lieb-Rivasseau inequality~\cite{lieb1980refinement, rivasseau1980lieb}: Let $x , y \in \Lambda$ be two distinct vertices and let $H \subseteq \Lambda$ be a finite subset of $\Lambda$ containing $x$ and not containing $y$. Then
        \begin{equation*}
            \langle \cos( \theta_x - \theta_y) \rangle_{\mu_{\Lambda, \beta_1 , \beta_2}} \leq \sum_{z \in \partial H} \langle \cos( \theta_x - \theta_z) \rangle_{\mu_{H, \beta_1 , \beta_2}}  \langle \cos( \theta_z - \theta_y) \rangle_{\mu_{\Lambda, \beta_1 , \beta_2}} 
        \end{equation*}
    \item Messager-Miracle-Sole inequality~\cite{messager1977correlation}: For a given hyperplane $P \subseteq \R^d$ and a vertex $x \in \Z^d$, we denote $P(x)$ the reflection of $x$ across the plane. Assume that $\Lambda \subseteq \Zd_n$ is symmetric under reflections across a plane $P$. Let $x , y \in \Lambda$ be two vertices which are on the same side of the plane $P$. Then
        \begin{equation*}
            \langle \cos( \theta_x - \theta_y) \rangle_{\mu_{\Lambda, \beta_1 , \beta_2}} \geq \langle \cos( \theta_x - \theta_{P(y)}) \rangle_{\mu_{\Lambda, \beta_1 , \beta_2}}.
        \end{equation*}
\end{itemize}
\end{proposition}

\begin{remark} \label{remarkA4}
    Let us make two remarks about the previous definitions:
        \begin{itemize}
            \item By taking the limit $\Lambda \uparrow \Zd_n$, all these inequalities hold with the measure $\mu_{\beta_1 , \beta_2}$ instead of $\mu_{\Lambda, \beta_1 , \beta_2}$
            \item Successive applications of the Messager-Miracle-Sole inequality using hyperplanes passing through vertices or mid-edges and ``diagonal" hyperplanes yield the following useful inequality (for which we refer to~\cite[Lemma 21]{van2023duality} for the $XY$ model in two dimensions and to~\cite[Section 5.1]{aizenman2021marginal} in the case of the Ising model): there exists a constant $C := C(d) < \infty$ such that for any pair of vertices $x , y \in \Zd \subseteq \Zd_n$ with $|y| \geq C |x|$, 
            \begin{equation} \label{MMSimpliescomparison}
                \langle \cos( \theta_0 - \theta_y) \rangle_{\mu_{ \beta_1 , \beta_2}} \leq \langle \cos( \theta_0 - \theta_x) \rangle_{\mu_{\beta_1 , \beta_2}}.
            \end{equation}
            \item These inequalities are in fact valid in a much more general framework: they are valid for the $XY$ model on any graph $G$ and for any collection of ferromagnetic coupling constants (with some suitable symmetry assumptions for the Messager-Miracle-Sole inequality). They are stated and proved in~\cite[Appendix A]{van2023elementary} in the case of general graph when all the coupling constants are equal to an inverse temperature $\beta$ (except for the Messager-Miracle-Sole inequality which is stated for subgraphs of $\Z^2$), but their proof could be extended to more general collections of coupling constants with the same arguments.
        \end{itemize}    
\end{remark}

Following~\cite[Section 6]{van2023elementary}, these inequalities lead to the following dichotomy for the two-point function under the measure $\mu_{\beta_1, \beta_2}$

\begin{proposition}[Dichotomy for the decay of the two-point function] \label{prop.dichotomy}
    For any integer $n \in \N$ and any pair of inverse temperatures $\beta_1 , \beta_2$, there exists a constant $c := c(n, \beta_1 , \beta_2) > 0$ such that:
    \begin{itemize}
        \item[(i)] Either $\langle \cos \left( \theta_0 - \theta_x \right) \rangle_{\mu_{\beta_1, \beta_2}} \leq e^{-c |x|}$ for each $x \in \Z^2_n$,
        \item[(ii)] Or $\langle \cos \left( \theta_0 - \theta_x \right) \rangle_{\mu_{\beta_1, \beta_2}} \geq \frac{c}{|x|}$ for each $x \in \Z^2_n \setminus \{ 0 \}$.
    \end{itemize}
\end{proposition}

\begin{proof}
Since the argument is essentially a rewrite of~\cite[Section 6]{van2023elementary}, we only present a detailed sketch of the proof. For any integer $R \in \N$, we introduce the quantity (recalling the definition of the box $\Lambda_R^n \subseteq \Z^2_n$ introduced in Section~\ref{sec:generaldefinition})
\begin{equation*}
    \varphi_{R} := \sum_{y \in \partial \Lambda_R^n} \langle \cos \left( \theta_0 - \theta_y \right) \rangle_{\mu_{\Lambda_R^n , \beta_1 , \beta_2}}.
\end{equation*}
Using this definition, we may formulate the following dichotomy: either $ \inf_{R \in \N} \varphi_{R} < 1$ or $ \inf_{R \in \N} \varphi_R \geq 1$. We will show that the first case corresponds to situation (i) and that the second case corresponds to the situation (ii).

Let us first assume that there exists an integer $R \in \N$ such that $\varphi_{R} < 1$. Applying the Lieb-Rivasseau inequality with the set $H = \Lambda_R^n$, we have, for any $x \in \Z^2_n \setminus \Lambda_R^n$,
\begin{align*}
    \langle \cos( \theta_0 - \theta_x) \rangle_{\mu_{\beta_1 , \beta_2}} & \leq \sum_{y \in \partial \Lambda_R^n} \langle \cos( \theta_0 - \theta_y) \rangle_{\mu_{\Lambda_R^n , \beta_1 , \beta_2}}  \langle \cos( \theta_y - \theta_x) \rangle_{\mu_{\beta_1 , \beta_2}} \\
    & \leq \varphi_{R} \times \sup_{y \in \partial \Lambda_R^n} \langle \cos( \theta_y - \theta_x) \rangle_{\mu_{\beta_1 , \beta_2}} \\
    & \leq \varphi_{R} \times \sup_{y \in  \Lambda_R^n \cap \Z} \langle \cos( \theta_y - \theta_x) \rangle_{\mu_{\beta_1 , \beta_2}},
 \end{align*}
 where we used in the third inequality that the vertices of $\partial \Lambda_R^n$ are included in $\Z^2 \subseteq \Z^2_n$ (and thus $\partial \Lambda_R^n \subseteq \Lambda_R^n \cap \Z^2$). 
This inequality may the be iterated so as to obtain (N.B. we use here that the graph $\Z^2_n$ is invariant under $\Z^2$ translations), for any integer $k \in \N$ and any $x \in \Z^2_n \setminus \Lambda_{kR}^n$,
\begin{equation*}
    \langle \cos( \theta_0 - \theta_x) \rangle_{\mu_{\beta_1 , \beta_2}} \leq \varphi_{R}^k  \sup_{y \in  \Lambda_{kR}^n \cap \Z} \langle \cos( \theta_y - \theta_x) \rangle_{\mu_{\beta_1 , \beta_2}}.
\end{equation*}
Since the two-point function is always smaller than $1$, the supremum on the right-hand side is smaller than $1$ and we deduce that, for any integer $k \in \N$ and any $x \in \Z^2_n \setminus \Lambda_{kR}^n$,
\begin{equation*}
    \langle \cos( \theta_0 - \theta_x) \rangle_{\mu_{\beta_1 , \beta_2}} \leq \varphi_{R}^k.
\end{equation*}
Using that $\varphi_{R} < 1$, we obtain the exponential decay of the two-point function. We are thus in the situation (i).

If $ \inf_{R \in \N} \varphi_R \geq 1$, then we can use the Messager-Miracle-Sole inequality as well as the rotation symmetry of the system to obtain the lower bound, for any $x \in \Z^2$,
\begin{equation} \label{eq:16.30}
    \langle \cos(\theta_0 - \theta_x) \rangle_{\mu_{\beta_1 , \beta_2}} \geq \frac{c}{|x|}.
\end{equation}
The argument is identical to the one written in~\cite[Section 6]{van2023elementary}, and we refer to this article for the details.

We finally prove that the lower bound holds for all $x \in \Z^2_n$ (and not only $x \in \Z^2$). To this end, for each $x \in \Z^2_n \setminus \Z^2$, we denote by $(x)$ the vertex in $\Z^2$ which is the closest to $x$ for the graph distance on $\Z^2_n$ (breaking ties using the lexicographical order). Applying the inequality~\eqref{eq:twopointbound}, we deduce that
\begin{equation*}
     \langle \cos(\theta_0 - \theta_x) \rangle_{\mu_{\beta_1 , \beta_2}} \geq \langle \cos(\theta_0 - \theta_{(x)}) \rangle_{\mu_{\beta_1 , \beta_2}} \langle \cos(\theta_{(x)} - \theta_x) \rangle_{\mu_{\beta_1 , \beta_2}}
\end{equation*}
The first term is bounded from below by $c/|(x)|$ and thus by $c'/|x|$ for some possibly smaller constant $c'$. Using translation invariance, we have $\langle \cos(\theta_{ (x)} - \theta_{x}) \rangle_{\mu_{\beta_1 , \beta_2}} = \langle \cos(\theta_{0} - \theta_{x - (x)}) \rangle_{\mu_{\beta_1 , \beta_2}}$. Since $x - (x) \in \Z^2_n$ is at distance less than $n$ from $0$ in the graph $\Z^2_n$ (for the graph distance), this second term can be easily bounded from below by a positive real number depending only on $n, \beta_1$ and $\beta_2$.
\end{proof}

\subsubsection{Transforming the graph into a triangulation} \label{sectiontriangulation}

We start by defining the triangulation and the $XY$ model on the triangulation which will be used in the rest of the argument.

\begin{definition}[Extended triangulation $\Gamma_n$]
Given an integer $n \in \N$, we define the extended triangulation $\Gamma_n$ to be the extended graph $\Z_n^2$ to which the diagonals (going from top left to bottom right of each square) have been added. We then add $(2n +1)$ vertices on each diagonal. We refer to Figure~\ref{extendedtriang} for a representation of the triangulation $\Gamma_n$. We will write $x \sim_n y$ to mean that $x , y \in \Gamma_n$ are neighbours (since this graph is an extension of the graph $\Z^2_n$, the notation is consistent with Remark~\ref{remark2.7}).
\end{definition}

\begin{figure} 
\begin{center}
\includegraphics[width=10cm]{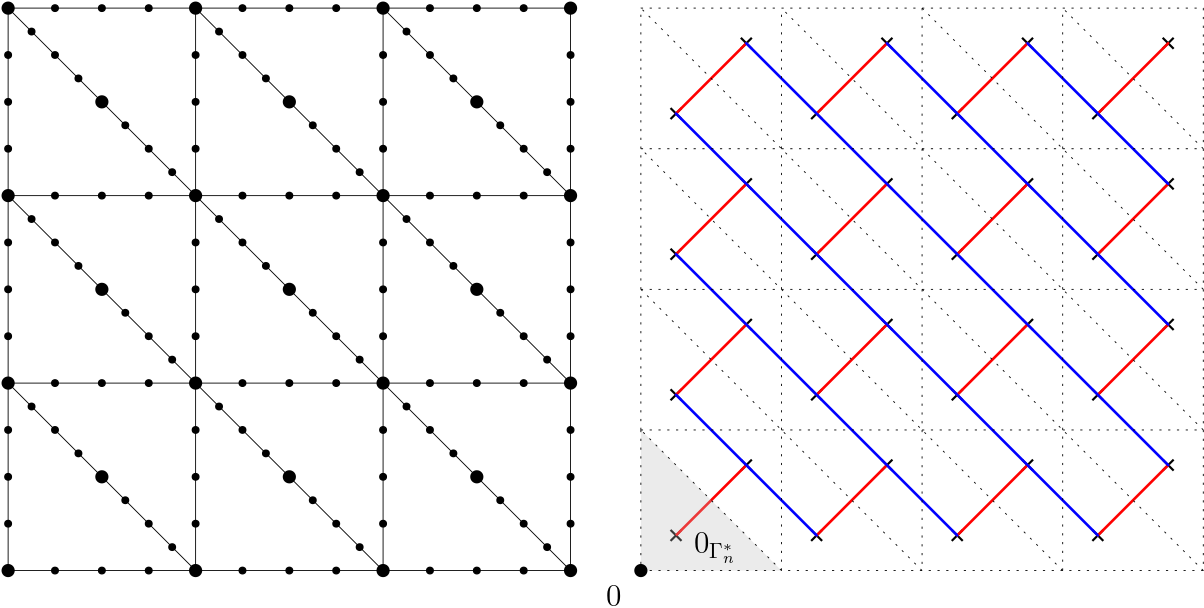}
\caption{The extended triangulation $\Gamma_n$ with $n = 3$ (left) and the dual graph $\Gamma_n^*$ (right). In the picture on the right, we represented the triangulation $\Gamma_n$ using dotted lines. The vertices of the dual graph are the faces $\Gamma_n$ and crosses are used to represent the starting and end points of the edges of $\Gamma_n^*$. The vertex $0 \in \Gamma_n$ is drawn with a black dot and the face $0_{\Gamma_n^*}$ is drawn in grey. The type 1 edges are drawn in blue and the type 2 edges are drawn in red.} \label{extendedtriang}
\end{center}
\end{figure}

We then define the $XY$ model with heterogeneous temperatures on the extended triangulation that will be used in the proof.

\begin{definition}[$XY$ model on the extended triangulation $\Gamma_n$ with heterogeneous temperature] \label{def.triangheterogeneousXY}
    Fix an integer $n \in \N$. For any finite set $\Lambda \subseteq \Gamma_n$ and any pair of inverse temperatures $\beta_1 , \beta_2 > 0$, we define the Hamiltonian
    \begin{equation*}
    H_{\Lambda, \beta_1 , \beta_2}(\theta) := - \beta_1 \sum_{\substack{x \sim_n y \\ \{ x , y \} \cap (\Zd \cup (\frac12, \frac12) + \Zd) \neq \emptyset}} \cos (\theta_x - \theta_y) - \beta_2 \sum_{\substack{x \sim_n y \\ \{ x , y \} \cap (\Zd \cup (\frac12, \frac12) + \Zd) = \emptyset}} \cos (\theta_x - \theta_y),
\end{equation*}
as well as the probability measure 
\begin{equation*}
        \mu_{\Lambda, \beta_1 , \beta_2}(d\theta) := \frac{1}{Z_{\Lambda, \beta_1 , \beta_2 }}\exp \left( - H_{\Lambda, \beta_1 , \beta_2}(\theta)  \right) \prod_{x \in \Lambda} d \theta_x.
\end{equation*}
We refer to Figure~\ref{XYtriang} for a representation of the model.
\end{definition}

\begin{figure} 
\begin{center}
\includegraphics[width=6.5cm]{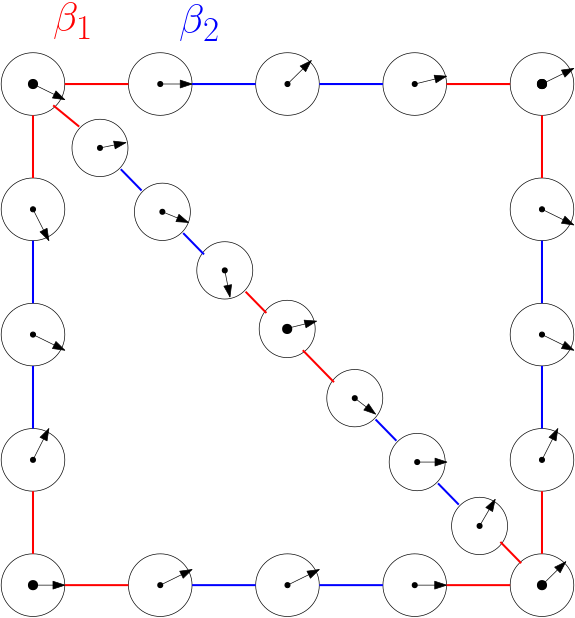}
\caption{A realization of the $XY$ model with heterogeneous temperatures on the extended triangulation with $n = 3$.} \label{XYtriang}
\end{center}
\end{figure}

We extend the notation and convention of Section~\ref{sec2.1.3.def} to the triangulation $\Gamma_n$.

\begin{remark} 
The Ginibre correlation inequality applies to this model. Combining this observation with a compactness argument, we obtain that the sequence of measures $\mu_{\Lambda, \beta_1 , \beta_2}$ converges as $\Lambda \uparrow \Gamma_n$ to an infinite-volume measure which will be denoted by $\mu_{\Gamma_n, \beta_1 , \beta_2}.$
\end{remark}

The following lemma shows that the value of the two-point function under the measure $\mu_{\beta_1 , \beta_2}$ is larger than the one under the measure $\mu_{\Gamma_n, \beta_1 /2 , \beta_2/2}$.

\begin{lemma} \label{LemmaA6}
    For any pair of inverse temperatures $\beta_1 , \beta_2$ and any pair of vertices $x , y \in \Z^2_n$, one has the inequality
    \begin{equation} \label{ineqaulityextedtoestendtriang}
        \langle \cos( \theta_x - \theta_y) \rangle_{\mu_{\beta_1 , \beta_2}} \geq \langle \cos( \theta_x - \theta_y) \rangle_{\mu_{ \Gamma_n, \beta_1 /2 , \beta_2 /2}}
    \end{equation}
\end{lemma}

\begin{proof}
The proof is based on successive applications of the Ginibre inequality and we essentially replicate the argument of~\cite[Proof of Theorem 15 for the square lattice]{van2023elementary}. We refer to Figure~\ref{transformintotrinag} for guidance. For each square of the extended lattice, we consider the left and bottom segments. We then duplicate each vertex and  each edge on these two lines, divide the temperature by $2$ and add an edge with a coupling constant equal to infinity between the vertices and the duplicated vertices. This does not affect the distribution of the spins. We then reduce the inverse temperature on these new edges from infinity to $0$.

This operation maps the extended graph $\Z^2_n$ to the triangulation $\Gamma_n$ and the measure $\mu_{\beta_1 , \beta_2}$ to the measure $\mu_{\Gamma_n, \beta_1 , \beta_2}$. The Ginibre inequality then implies the inequality~\eqref{ineqaulityextedtoestendtriang}.

\begin{figure} 
\begin{center}
\includegraphics[width=10cm]{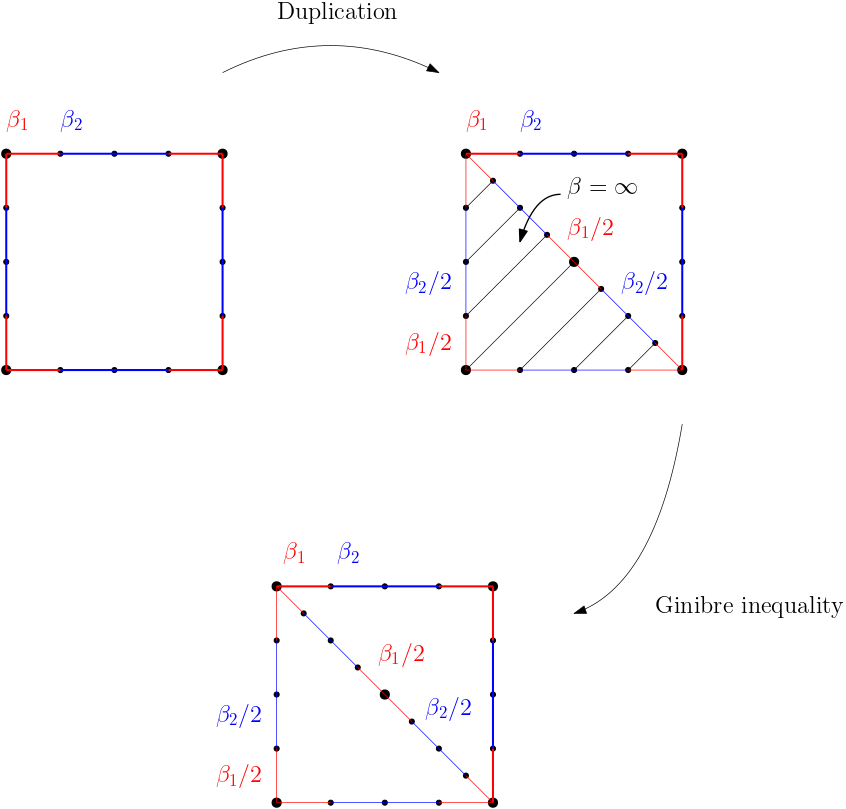}
\caption{Transformation into a triangulation. For each square of the graph $\Z^2_n$, we duplicate each vertex of the bottom and left side of the square, we then reduce the value of the inverse temperature from infinity to $0$. These two operations are then applied to each square of $\Z^2_n$ to obtain the triangulation $\Gamma_n$.} \label{transformintotrinag}
\end{center}
\end{figure}
\end{proof}

\subsubsection{Dual height function and delocalization} \label{sectionheightfunction}

In this section, we establish that the two-point function of the $XY$ model on the extended triangulation $\Gamma_n$ defined in the previous section is not summable. The proof follows the one of van Engelenburg and Lis~\cite{van2023elementary} which proceeds via a duality argument. They first define the dual graph to the triangulation $\Gamma_n$ as well as a dual height function on this graph, and show that, if this height function is delocalized (as formally defined in Proposition~\ref{prop.A11}) then the two-point function of the $XY$ model (on the extended triangulation $\Gamma_n$) is not summable. They then apply the results of Lammers~\cite{lammers2022height} and Lammers-Ott~\cite{lammers2023delocalisation} to show that if the inverse temperature is chosen sufficiently large then the dual height function is delocalized. Their argument is reproduced in this section with minor modifications.

\begin{definition}[Dual graph]
We denote by $\Gamma_n^*$ the dual graph of $\Gamma_n$ whose vertices are the faces of the graph $\Gamma_n$. Two faces of the graph $\Gamma_n^*$ are connected by an edge if they share at least one (in fact either $(n+1)$ or $(2n+2)$) edge in the graph $\Gamma_n$ (see Figure~\ref{extendedtriang}). We let $0_{\Gamma_n^*} \in \Gamma_n^*$ be the top right face incident to the vertex $0 \in \Gamma_n$ (see Figure~\ref{extendedtriang}).

We classify the edges of the dual graph $\Gamma_n^*$ in two categories:
    \begin{itemize}
        \item An edge $\{ u,v\}$ of the dual graph $\Gamma_n^*$ (with $u , v \in \Gamma_n^*$) is called a \emph{type 1 edge} if the two faces $u$ and $v$ of $\Gamma_n$ share $(n+1)$ edges (or equivalently the side of a square of $\Z^2_n$). These edges are drawn in blue on the right-hand side of Figure~\ref{extendedtriang}.
        \item An edge $\{ u,v\}$ of the dual graph $\Gamma_n^*$ (with $u , v \in \Gamma_n^*$) is called a \emph{type 2 edge} if the two faces $u$ and $v$ of $\Gamma_n$ share $(2n+2)$ edges (or equivalently the diagonal of a square of $\Z^2_n$). These edges are drawn in red on the right-hand side of Figure~\ref{extendedtriang}.
    \end{itemize}
\end{definition}

We next define the height function which is dual to the $XY$ model introduced in Definition~\ref{def.triangheterogeneousXY}.

\begin{definition}[Dual height-function] \label{dualheightfct}
    For any finite subset $G \subseteq \Gamma_n^*$, we define the set of height functions according to the formula $\Omega_G := \left\{ h : G \to \Z \, : \, h = 0 ~\mbox{on}~ \partial G \right\}$. For any pair of inverse temperatures $\beta_1 , \beta_2 > 0$, we equip the space of height functions $\Omega_G$ with the probability distribution
    \begin{equation*}
        \mathbb{P}_{G, \beta_1 , \beta_2} := \frac{1}{Z_{G,  \beta_1 , \beta_2}} \exp \left( - \sum_{\{ u, v\} \in  E \left( \Gamma_n^*\right)} V_{uv}^{\beta_1 , \beta_2}(h(u) - h(v)) \right),
    \end{equation*}
    where the potentials $V_{uv}^{\beta_1 , \beta_2} : \Z \to \R$ are defined as follows:
    \begin{itemize}
        \item For any type 1 edge $\{ u, v\} \in  E \left( \Gamma_n^*\right)$ and any integer $k \in \Z$,
    \begin{align*}
            V_{uv}^{\beta_1 , \beta_2}(k) & := - 2 \ln \left( \sum_{i=0}^\infty \frac{1}{i! (i + |k|)!} \beta_1^{2i + |k|} \right)  - (n-1) \ln \left( \sum_{i=0}^\infty \frac{1}{i! (i + |k|)!} \beta_2^{2i + |k|} \right) \\
            & = - 2 \ln \left( I_k(2\beta_1) \right) - (n-1) \ln \left( I_k(2\beta_2) \right),
    \end{align*}
    where $I_k$ is the modified Bessel function.
        \item  For any type 2 edge $\{ u, v\} \in E \left( \Gamma_n^*\right)$ and any integer $k \in \Z$,
    \begin{align*}
            V_{uv}^{\beta_1 , \beta_2}(k) & := -4  \ln \left( \sum_{i=0}^\infty \frac{1}{i! (i + |k|)!} \beta_1^{2i + |k|} \right)  - (2n-2) \ln \left( \sum_{i=0}^\infty \frac{1}{i! (i + |k|)!} \beta_2^{2i + |k|} \right) \\
            & = - 4 \ln \left( I_k(2\beta_1) \right) - (2n-2) \ln \left( I_k(2\beta_2) \right).
    \end{align*}
    \end{itemize}
\end{definition}

It is known that the potentials $V_{uv}^{\beta_1 , \beta_2}$ (either type 1 or type 2) are convex (see~\cite{thiruvenkatachar1951inequalities} and~\cite[Section 2]{van2023elementary}), and thus the model defined above falls into the well-studied framework of integer-valued height functions with convex interaction potential. We next collect two results from~\cite{van2023elementary}. The first one relates the behaviour of height function to the decay the two-point function of the $XY$ model on the extended triangulation $\Gamma_n$. The second one establishes the delocalization of the height function and is based on the results of Lammers~\cite{lammers2022height} and Lammers-Ott~\cite{lammers2023delocalisation}.

\begin{proposition}[Delocalization for the height function and lower bound for the sum of correlations along the $x$-axis] \label{prop.A11}
    In the setup above, for every integer $n \in \N$ and every pair of inverse temperatures $\beta_1, \beta_2 > 0$, the following properties hold:
    \begin{itemize}
        \item Delocalization for the dual height function~\cite{lammers2022height, lammers2023delocalisation} and~\cite[Theorem 3]{van2023elementary}: if the inverse temperatures $\beta_1$ and $\beta_2$ are such that, for any edge $\{u , v\} \in E\left( \Gamma_n^* \right)$ (either type 1 or type 2), Lammers' condition is satisfied
        \begin{equation} \label{Lammerscondition}
            V_{uv}^{\beta_1 , \beta_2}( \pm 1) \leq V_{uv}^{\beta_1 , \beta_2}( 0) + \ln 2,
        \end{equation}
        then for any increasing sequence of finite sets $G_n \subseteq \Gamma_n^*$ such that $G_{n} \subseteq G_{n+1}$ and $\cup_{n \in \N} G_n = \Gamma_n^*$, and any vertex $u \in G_0$
        \begin{equation} \label{delocarandomsurface}
            \lim_{n \to \infty} \mathbb{E}_{G_n, \beta_1 , \beta_2} \left[ |h(u)| \right] = \infty.
        \end{equation}
        \item Lower bound for the sum of correlations~\cite[Proposition 9]{van2023elementary}: there exists a constant $C := C(\beta_1, \beta_2, n) < \infty$ such that 
        \begin{equation} \label{sumcorrelation}
            \sup_{\substack{G \subseteq \Gamma_n^* \\ G \, \mathrm{finite} \\ 0_{\Gamma_n^*}\in G} }\mathbb{E}_{G , \beta_1 , \beta_2} \left[ |h(0_{\Gamma_n^*})| \right] \leq C \sum_{k , k' = 0}^\infty \left\langle  \cos \left( \theta_{(-k/(n+1), 0)} - \theta_{(k'/(n+1), 0)} \right) \right\rangle_{\mu_{\Gamma_n , \beta_1 , \beta_2}},
        \end{equation}
        where the supremum is considered over all the finite subgraphs $G \subseteq \Gamma_n^*$ containing $0_{\Gamma_n^*}$.
    \end{itemize}
\end{proposition}

\begin{proof}
The first statement is a consequence of~\cite[Theorem 3]{van2023elementary} combined with the result of Lammers~\cite{lammers2022height} regarding the non-existence of translation invariant Gibbs measure for height functions. The only minor difference with~\cite[Theorem 3]{van2023elementary} is that we need to verify the absolute-value FKG inequality for the dual height function with the potentials introduced in Definition~\ref{dualheightfct} (which are slightly different from the ones used in~\cite{van2023elementary}). Nevertheless, the proof in our case is a direct consequence of the proof of~\cite[Lemma 5 and Lemma 6]{van2023elementary}, we thus omit the technical details.

The second statement is essentially a less general version~\cite[Proposition 9]{van2023elementary}. It can be deduced from their result by choosing $\Gamma$ to be the graph $\Gamma_n^*$, $\epsilon = 1$ and the cut $L$ to be the horizontal line $\{ (k/(n+1) , 0) \, : \, k \in \Z\} \subseteq \Gamma_n.$ We note that there is a minor difference between~\cite[Proposition 9]{van2023elementary} and the situation of Proposition~\ref{prop.A11} as~\cite[Proposition 9]{van2023elementary} assumes that all the edges have the same inverse temperature $\beta$. Nevertheless, a notational modification of the arguments of~\cite[Proposition 9]{van2023elementary} would yield the result for a collection of arbitrary non-negative coupling constants. We thus omit the technical details.
\end{proof}

We have now collected all the necessary ingredients to complete the proof of Proposition~\ref{propApp.phasetransitionextended}.

\begin{proof}[Proof of Proposition~\ref{propApp.phasetransitionextended}]
    We essentially rewrite the argument of~\cite[Proof of Theorem 15 for the square lattice]{van2023elementary}. We first show that we may choose two inverse temperatures $\beta_1$ and $\beta_2$ such that Lammers' condition~\eqref{Lammerscondition} is satisfied. Using Definition~\ref{dualheightfct}, we see that it is sufficient to select $\beta_1$ and $\beta_2$ such that
    \begin{equation} \label{conditionbeta1beta2}
        \frac{I_0(2 \beta_1)}{I_1(2 \beta_1)} \leq 2^{\frac{1}{8}} ~~\mbox{and}~~ \frac{I_0(2 \beta_2)}{I_1(2 \beta_2)} \leq e^{\frac{1}{4(n-1)}}.
    \end{equation}
    Using that the ratio $I_0(\beta)/I_1(\beta)$ converges to $1$ as $\beta$ tends to infinity, we may select a (universal) inverse temperature $\beta_1 < \infty$ and an inverse temperature $\beta_2 < \infty$ (depending on the integer $n$) such that the conditions~\eqref{conditionbeta1beta2} are satisfied.

    Once the inverse temperatures $\beta_1$ and $\beta_2$ are selected so that the inequalities~\eqref{Lammerscondition} hold, we may combine~\eqref{delocarandomsurface} with the inequality~\eqref{sumcorrelation} to see that the sum of correlations on the right-hand side of~\eqref{sumcorrelation} must be infinite. Combining this result with the dichotomy stated in Proposition~\ref{prop.dichotomy} and Lemma~\ref{LemmaA6}, we see that the two-point function cannot decay exponentially fast (otherwise the right-hand side of~\eqref{sumcorrelation} would be finite), and thus the two-point function $ x \mapsto \langle \cos \left( \theta_0 - \theta_x \right) \rangle_{\mu_{\beta_1, \beta_2}}$ decays polynomially fast. The proof of Proposition~\ref{propApp.phasetransitionextended} is complete.
\end{proof}

\subsection{Long-range order on the extended lattice $\Zd_n$}

In this section, we establish the existence of a order/disorder phase transition for the $XY$ model on the extended lattice. The statement is recalled below. The proof is a minor adaptation of the recent proof of the long range-order for the $XY$ model of the second author and Spencer~\cite{garban2022continuous}. We first establish in Proposition~\ref{propA16} a form of long-range order for a special class of disordered version of the $XY$ model described by the \emph{Nishimori disorder} (which was originally introduced by Nishimori~\cite{nishimori1981internal} in the setting of the Ising model). We then show that it implies the result for the non-disordered $XY$ model using a correlation inequality of Messager, Miracle-Sole, and Pfister~\cite{MMP} (see Proposition~\ref{propMMSPineq}), and the Messager-Miracle-Sole inequality~\cite{messager1977correlation} (stated in Proposition~\ref{prop.corrineq}). 

\begin{proposition}[Long-range order for the $XY$ model on the extended lattice $\Zd_n$] \label{propAppLRO.phasetransitionextended}
In dimension $d \geq 3$, there exists an inverse temperature $\beta_{1, c}(d) < \infty$ such that, for every $n \in \N$, there exists an inverse temperature $\beta_{2, c}(d , n) < \infty$ such that, for any $\beta_1 \geq \beta_{1, c}(d)$ and any $\beta_2 \geq \beta_{2, c}(n , d)$
    \begin{equation*}
        \left\langle \cos(\theta_0 - \theta_x) \right\rangle_{\mu_{\beta_1 , \beta_2 }} ~\mbox{remains bounded away from } 0 \mbox{ as } |x| \to \infty.
    \end{equation*}
\end{proposition}

\begin{remark}
The following quantitative estimate can be deduced from the argument: there exists a constant $C := C(d) < \infty$ such that, for any $x \in \Zd_n$,
\begin{equation*}
    \langle \cos \left( \theta_0 - \theta_x \right) \rangle_{\mu_{\beta_1, \beta_2}} \geq 1 - C \sqrt{ \frac{1}{\beta_1} + \frac{n}{\beta_2}}.
\end{equation*}

\end{remark}

\subsubsection{The extended $XY$ model in a Nishimori disorder}

Following~\cite{garban2022continuous}, we introduce the Nishimori disorder, and then introduce the $XY$ model in a Nishimori disorder. 

\begin{definition}[Nishimori disorder]
Fix $n \in \N$ and $\Lambda \subseteq \Z^d_n$ and a pair of inverse temperatures $\beta_1 , \beta_2 >0$. We let $\omega := \left( \omega_{xy} \right)_{(x,y) \in \vec{E} \left( \Lambda \right)} \in [0, 2\pi)^{\vec{E} \left( \Lambda \right)}$ be a collection of random variables satisfying the following properties:
\begin{enumerate}
    \item[(i)] For any $(x,y) \in \vec{E} \left( \Lambda \right)$, $\omega_{xy} = - \omega_{yx}$.
    \item[(ii)] For any pair of distinct oriented edges $(x,y), (z,t) \in \vec{E} \left( \Lambda \right)$, the random variables $\omega_{xy}$ and $\omega_{zt}$ are independent.
    \item[(iii)] For any $(x,y) \in \vec{E} \left( \Lambda \right)$ such that $\{x , y \} \cap \Z \neq \emptyset$ (i.e., an edge colored in red in Figure~\ref{XYextended}), the law of $\omega_{xy}$ is supported on $[0 , 2\pi]$, and its density is given by the distribution:
    \begin{equation} \label{densityrhobeta1}
        \rho_{\beta_1}(\omega) := \frac{1}{Z_{\beta_1}} e^{\beta_1 \cos (\omega)} ~~\mbox{with}~~ Z_{\beta_1} := \int_{0}^{2\pi} e^{\beta_1 \cos (\omega)} d \omega.
    \end{equation}
    Similarly, for any $(x,y) \in \vec{E} \left( \Lambda \right)$ such that $\{x , y \} \cap \Z = \emptyset$ (i.e., an edge colored in blue in Figure~\ref{XYextended}), the law of $\omega_{xy}$ is supported on $[0 , 2\pi]$, and its density is given by the distribution:
    \begin{equation} \label{densityrhobeta2}
        \rho_{\beta_2}(\omega) := \frac{1}{Z_{\beta_2}} e^{\beta_2 \cos (\omega)} ~~\mbox{with}~~ Z_{\beta_2} := \int_{0}^{2\pi} e^{\beta_2 \cos (\omega)} d \omega.
    \end{equation}
    We denote by $\mathbb{P}^{\mathrm{Nish}}_{\beta_1, \beta_2}$ and $\mathbb{E}^{\mathrm{Nish}}_{\beta_1, \beta_2}$ the probability and expectation with respect to this quenched disorder.
\end{enumerate}
\end{definition}

We next defined the disordered version of the $XY$ model which will be used in the proof.

\begin{definition}[Disordered $XY$ model on the extended lattice] \label{def.disorderedXYmodel}
    Let $\Lambda \subseteq \Zd_n$ be a finite subset. For any pair of inverse temperatures $\beta_1 , \beta_2 > 0$ and any disorder $\omega \in (0, 2\pi)^{\vec{E}(\Lambda)}$, we define the Hamiltonian
    \begin{equation*}
    H_{\Lambda, \beta_1 , \beta_2, \omega}(\theta) := - \beta_1 \sum_{\substack{x \sim_n y \\ x,y \in \Lambda \\ \{ x , y \} \cap \Zd \neq \emptyset}} \cos (\theta_x - \theta_y - \omega_{xy}) - \beta_2 \sum_{\substack{x \sim_n y \\ x,y \in \Lambda \\ \{ x , y \} \cap \Zd = \emptyset}} \cos (\theta_x - \theta_y - \omega_{xy}),
\end{equation*}
as well as the probability measure 
\begin{equation*}
        \mu_{\Lambda, \beta_1 , \beta_2, \omega}(d\theta) := \frac{1}{Z_{\Lambda, \beta_1 , \beta_2 }}\exp \left( - H_{\Lambda, \beta_1 , \beta_2, \omega}(\theta)  \right) \prod_{x \in \Lambda} d \theta_x.
\end{equation*}
\end{definition}

We now state the main result of this section which establishes the long-range order for the $XY$ model in a Nishimori disorder. The result can be compared to~\cite[Theorem 1.3]{garban2022continuous} (and the proof is an adaptation of this result). In the following statement and below, given an integer $k \in \N$, we denote by $\vec{\mathbf{k}} := (k , \ldots, k) \in \Zd$ (the vertex of $\Zd$ whose coordinates are all equal to $k$).

\begin{proposition} \label{propA16}
    Fix a dimension $d \geq 3$, an integer $n \in \N$, an integer $k \in \N$ and a finite box $\Lambda \subseteq \Zd_n$ such that $0 \in \Lambda$ and $\vec{\mathbf{k}} \in \Lambda$. There exists a constant $C := C(d) < \infty$ such that, for any pair of inverse temperatures $\beta_1 , \beta_2 \in (0, \infty)$,
    \begin{equation} \label{ineq:1637}
        \mathbb{E}^{\mathrm{Nish}}_{\Lambda, \beta_1,\beta_2} \left[ \left\langle \cos \left( \theta_0 - \theta_{\vec{\mathbf{k}}} \right) \right\rangle_{\mu_{\Lambda, \beta_1 , \beta_2, \omega}} \right] \geq 1 - C \sqrt{ \frac{1}{\beta_1} + \frac{n}{\beta_2}}.
    \end{equation}
\end{proposition}

The proof of Proposition~\ref{propA16} relies on the two following lemmas. The first one is an identity which relies on the specific structure of the Nishimori disorder. Its proof is based on a gauge transform and is essentially identical to the proof written in~\cite[Lemma 2.1]{garban2022continuous}, we thus omit the demonstration and refer the interested reader to~\cite{garban2022continuous}. 

The second result is the so-called \emph{unpredictable path} of Benjamini, Pemantle and Peres~\cite{benjamini1998unpredictable}. In order to state the result, we consider the lattice $\Z^3$ in dimension $d=3$ and define an \emph{increasing} paths of $\Z^3$ to be an infinite path starting from $0$ and formed by sums of $(1 , 0 , 0)$, $(0 , 1 ,0 )$ and $(0,0,1)$ (see Figure~\ref{unpredictablepath}). The main result of~\cite{benjamini1998unpredictable} (stated below) asserts that it is possible to find a probability distribution $\mathfrak{M}$ on the set of increasing path such that two paths sampled independently according to $\mathfrak{M}$ have typically small intersection. We will need here a finite volume version of the result whose proof can be found in~\cite{abbe2018group}.

\begin{figure} 
\begin{center}
\includegraphics[width=10cm]{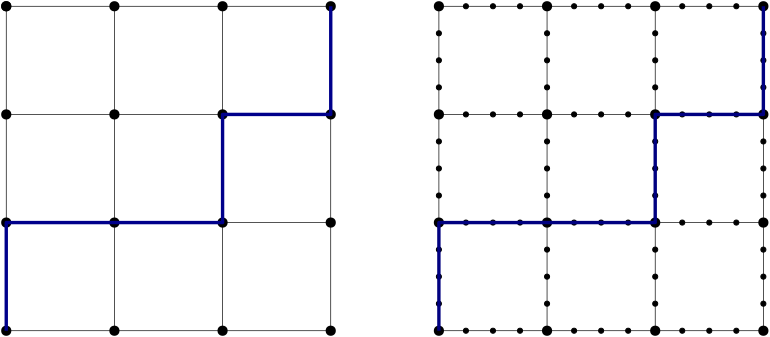}
\caption{An increasing path between $0$ and $\vec{\mathbf{k}}$ is depicted in blue on the lattice $\Zd$ on the left with the value $k = 3$ (and in dimension $d = 2$). The corresponding path on the extended lattice $\Zd_n$ is drawn on the right with the value $n=3$. On the extended lattice, the path visits exactly $(n+1) d k$ edges, $2dk$ of these edges are of the form $e = (x , y)$ with $\{x , y\}\cap  \Zd \neq \emptyset$ and $(n-1) d k$ of them are of the form $e = (x , y)$ with $\{x , y\}\cap  \Zd = \emptyset$.} \label{unpredictablepath}
\end{center}
\end{figure}

\begin{lemma} \label{lemmaA17}
    For any finite domain $\Lambda \subseteq \Zd_n$, any pair of inverse temperatures $\beta_1 , \beta_2 >0$, and any collection of smooth, periodic functions $(f_{xy})_{(x,y) \in \vec{E}(\Lambda)}$, one has the identity
    \begin{equation*}
        \mathbb{E}^{\mathrm{Nish}}_{\Lambda, \beta_1,\beta_2}  \left[ \left\langle \prod_{\substack{x , y \in \Lambda \\x \sim_n y}} f_{xy}(\theta_x - \theta_y + \omega_{xy}) \right\rangle_{\mu_{\Lambda, \beta_1,\beta_2, \omega}} \right] = \mathbb{E}^{\mathrm{Nish}}_{\Lambda, \beta_1,\beta_2}  \left[ \prod_{\substack{x , y \in \Lambda \\x \sim_n y}} f_{xy}( \omega_{xy}) \right].
    \end{equation*}
\end{lemma}

\begin{theorem}[Unpredictable path~\cite{benjamini1998unpredictable, abbe2018group}] \label{thmunpredic}
    In dimension $d = 3$, there exist universal constants $C < \infty$ and $c >0$ such that the following statements hold:
    \begin{itemize}
    \item Unpredictable paths in infinite volume~\cite{benjamini1998unpredictable}: there exists a probability distribution $\mathfrak{M}$ on the set of infinite increasing paths which satisfy the following intersection tail property: for any integer $k \in \N$,
    \begin{equation*}
        \mathfrak{M} \otimes \mathfrak{M} \left[ (\gamma_1 , \gamma_2) \mbox{ such that } \left| \gamma_1 \cap \gamma_2 \right| \geq k \right] \leq C e^{-c k}.
    \end{equation*}
    \item Unpredictable paths in finite volume~\cite{abbe2018group}: For any $k \in \N$, there exists probability distribution $\mathfrak{M}_k$ on the set of increasing paths going from $0$ to $\vec{\mathbf{k}}$ which satisfy the following intersection tail property: for any integer $k \in \N$,
    \begin{equation*}
        \mathfrak{M}_k \otimes \mathfrak{M}_k \left[ (\gamma_1 , \gamma_2) \mbox{ such that } \left| \gamma_1 \cap \gamma_2 \right| \geq k \right] \leq C e^{-c k}.
    \end{equation*}
    \end{itemize}
\end{theorem}

\begin{remark}
The analogue of this statement in dimension $d \geq 4$ is much easier as the uniform measure on increasing paths satisfies the intersection tail property. In that case, the constant $C$ and the exponent $c$ depend also on the dimension $d$.
\end{remark}

We then combine Lemma~\ref{lemmaA17} and Theorem~\ref{thmunpredic} to prove Proposition~\ref{propA16}. The proof below closely follows the one of~\cite{garban2022continuous}, which is itself inspired from~\cite{abbe2018group}.

\begin{proof}[Proof of Proposition~\ref{propA16}] 
For $d \geq 3$, we fix an integer $n \in \N$, an integer $k \in \N$ and let $\Lambda \subseteq \Zd_n$ be a box of the extended lattice such that $0  \in \Lambda$ and $\vec{\mathbf{k}} \in \Lambda$. For $\beta_1 , \beta_2 \in (0 , \infty)$, we define
\begin{align*}
    \lambda_n(\beta_1 , \beta_2) & := \E_{\rho_{\beta_1}} \left[ \cos(\omega) \right]^2 \E_{\rho_{\beta_2}} \left[ \cos(\omega) \right]^{n-1} \\
    & = \left( \frac{\int_{0}^{2\pi} \cos (\omega) e^{\beta_1 \cos (\omega)} d \omega}{\int_{0}^{2\pi} e^{\beta_1 \cos (\omega)} d \omega} \right)^{2} \times \left(\frac{\int_{0}^{2\pi} \cos (\omega) e^{\beta_2 \cos (\omega)} d \omega}{\int_{0}^{2\pi} e^{\beta_2 \cos (\omega)} d \omega} \right)^{n-1}.
\end{align*}
One can check that,
\begin{equation} \label{expansionlambdanbeta}
    \lambda_n(\beta_1 , \beta_2) = 1 - \frac{1}{\beta_1} - \frac{n-1}{2\beta_2} + o \left( \frac{1}{\beta_1} + \frac{n}{\beta_2} \right).
\end{equation}
We next fix an integer $k \in \N$ and select an increasing path $\gamma$ going from $0$ to $\vec{\mathbf{k}}$ in $\Z^d$ and note that its length is equal to $dk$. We may see $\gamma$ as a path of the extended lattice~$\Zd_n$ (see Figure~\ref{unpredictablepath}). We then apply Lemma~\ref{lemmaA17} with the function $f_{xy}(\theta) = e^{i \theta}$ if the edge $ (x , y) \in \vec{E} \left(  \Lambda \right)$ belongs to the path $\gamma$ and $f_{xy}(\theta) = 1$ otherwise. We obtain the identity
\begin{equation*}
    \mathbb{E}^{\mathrm{Nish}}_{\Lambda, \beta_1,\beta_2}  \left[ \left\langle e^{i (\theta_{\vec{\mathbf{k}}} - \theta_0)} \right\rangle_{\mu_{\Lambda, \beta_1,\beta_2, \omega}} \prod_{(x,y) \in \gamma} e^{ i \omega_{xy}} \right] = \mathbb{E}^{\mathrm{Nish}}_{\Lambda, \beta_1,\beta_2}  \left[ \prod_{(x,y) \in \gamma} e^{ i \omega_{xy}} \right].
\end{equation*}
The second term on the right-hand side can be computed using the independence of the Nishimori disorder and the following observation: the path $\gamma$ contains exactly $2d k$ edges of the form $e = (x , y)$ with $\{ x , y\} \cap \Z \neq \emptyset$ (i.e., an edge colored in red in Figure~\ref{XYextended}, the density of the disorder $\omega_{xy}$ is then given by~\eqref{densityrhobeta1}) and exactly $(n-1)dk$ edges of the form $(x , y)$ with $\{ x , y\} \cap \Z \neq \emptyset$ (i.e., an edge colored in blue in Figure~\ref{XYextended}, the density of the disorder $\omega_{xy}$ is then given by~\eqref{densityrhobeta2}). We obtain the identity
\begin{equation} \label{eq:identitypathdisroder}
    \mathbb{E}^{\mathrm{Nish}}_{\Lambda, \beta_1,\beta_2}  \left[ \prod_{(x,y) \in \gamma} e^{ i \omega_{xy}} \right] = \E_{\rho_{\beta_1}} \left[ \cos(\omega) \right]^{2dk} \E_{\rho_{\beta_2}} \left[ \cos(\omega) \right]^{(n-1)dk} =  \lambda_n (\beta_1 , \beta_2)^{dk}.
\end{equation}
Combining the two previous displays, we obtain the identity
\begin{equation} \label{eq:1053}
    \mathbb{E}^{\mathrm{Nish}}_{\Lambda, \beta_1,\beta_2}  \left[ \left\langle e^{i (\theta_{\vec{\mathbf{k}}} - \theta_0)} \right\rangle_{\mu_{\Lambda, \beta_1,\beta_2, \omega}} \prod_{(x,y) \in \gamma} e^{ i \omega_{xy}} \right] = \lambda_n (\beta_1 , \beta_2)^{dk}.
\end{equation}
Since the identity~\eqref{eq:1053} holds for every increasing path $\gamma$ going from $0$ to $\vec{\mathbf{k}}$, we may average~\eqref{eq:1053} over paths sampled according to the measure $\mathfrak{M}_k$ (using the uniform measure in dimension $d \geq 4$) to obtain the identity
\begin{equation} \label{eq:IdentityR(omega)}
\mathbb{E}^{\mathrm{Nish}}_{\Lambda, \beta_1,\beta_2}  \left[ \left\langle e^{i (\theta_{\vec{\mathbf{k}}} - \theta_0)} \right\rangle_{\mu_{\Lambda, \beta_1,\beta_2, \omega}} \mathbf{R}(\omega) \right] = 1 ~~\mbox{with}~~  \mathbf{R}(\omega) := \frac{1}{\lambda_n(\beta_1 , \beta_2)^{dk}}\E_{\mathfrak{M}_k} \left[  \prod_{(x,y) \in \gamma} e^{ i \omega_{xy}}  \right],
\end{equation}
where $\E_{\mathfrak{M}_k}$ denotes the expectation with respect to the measure $\mathfrak{M}_k$. Note that the identity~\eqref{eq:identitypathdisroder} implies that $\mathbb{E}^{\mathrm{Nish}}_{\Lambda, \beta_1,\beta_2}  \left[  \mathbf{R}(\omega) \right] = 1$. We next show, using the intersection tail property of the measure $\mathfrak{M}_k$ and the independence of the Nishimori disorder that the random variable $\mathbf{R}(\omega)$ is concentrated around the value $1$. Specifically, we prove the inequality: there exist two constants $C := C(d) < \infty$ and $\alpha := \alpha(d) > 0$ such that for any $\beta_1 , \beta_2$ satisfying $1/\beta_1 + n/ \beta_2 \leq \alpha$
\begin{equation} \label{ineq:concentration}
    \mathbb{E}^{\mathrm{Nish}}_{\Lambda, \beta_1,\beta_2} \left[ \left| \mathbf{R}(\omega) - 1\right|^2  \right] \leq C \left( \frac{1}{\beta_1} + \frac{n}{\beta_2} \right).
\end{equation}
We first show how to complete the proof of Proposition~\ref{propA16} using the inequality~\eqref{ineq:concentration}. Using the identity~\eqref{eq:IdentityR(omega)} and the Cauchy-Schwarz inequality, we have, for any $\beta_1 , \beta_2$ satisfying $1/\beta_1 + n/ \beta_2 \leq \alpha$,
\begin{align*}
    \mathbb{E}^{\mathrm{Nish}}_{\Lambda, \beta_1,\beta_2}  \left[ \left\langle \cos \left( \theta_{\vec{\mathbf{k}}} - \theta_0 \right) \right\rangle_{\mu_{\Lambda, \beta_1,\beta_2, \omega}} \right]
    & = \mathbb{E}^{\mathrm{Nish}}_{\Lambda, \beta_1,\beta_2}  \left[ \left\langle e^{i (\theta_{\vec{\mathbf{k}}} - \theta_0)} \right\rangle_{\mu_{\Lambda, \beta_1,\beta_2, \omega}} \right] \\
    & = 1 - \mathbb{E}^{\mathrm{Nish}}_{\Lambda, \beta_1,\beta_2}  \left[ \left\langle e^{i (\theta_{\vec{\mathbf{k}}} - \theta_0)} \right\rangle_{\mu_{\Lambda, \beta_1,\beta_2, \omega}} (\mathbf{R}(\omega) -1) \right] \\
    & \geq 1 -  \mathbb{E}^{\mathrm{Nish}}_{\Lambda, \beta_1,\beta_2}  \left[ \left|\mathbf{R}(\omega) -1 \right|^2 \right]^{\frac{1}{2}} \\
    & \geq 1 - C \sqrt{ \frac{1}{\beta_1} + \frac{n}{\beta_2}}.
\end{align*}
The additional assumption $1/\beta_1 + n/ \beta_2 \leq \alpha$ can be removed by noting that the left-hand side of the previous display is always nonnegative and by increasing the value of the constant $C$ if necessary (by, for instance, assuming that it is larger than $1/\sqrt{\alpha}$). This is precisely~\eqref{ineq:1637}. There only remains to prove the inequality~\eqref{ineq:concentration}. The proof written below is similar to the proof of~\cite[Lemma 2.5]{garban2022continuous} and~\cite{abbe2018group}. Using the definition of the random variable $\mathbf{R}(\omega)$ stated in~\eqref{eq:IdentityR(omega)}, we have the identity
\begin{equation*}
    \mathbb{E}^{\mathrm{Nish}}_{\Lambda, \beta_1,\beta_2} \left[ \left| \mathbf{R}(\omega) \right|^2  \right] = \E_{\mathfrak{M}_k \otimes \mathfrak{M}_k} \left[ \left(\frac{1}{\lambda_n(\beta_1 , \beta_2)}\right)^{|\gamma_1 \cap \gamma_2|} \right].
\end{equation*}
where, on the right-hand side $\gamma_1, \gamma_2$ are two paths sampled independently according to the measure $\mathfrak{M}_k$. We next introduce the function, for $t \in \R$,
\begin{equation*}
    G_k(t) := \E_{\mathfrak{M}_k \otimes \mathfrak{M}_k} \left[ e^{t |\gamma_1 \cap \gamma_2|} \right].
\end{equation*}
The function $G$ is differentiable on $\R$. Using the intersection tail property stated in Theorem~\ref{thmunpredic}, we may bound its derivative in the interval $(- \infty , \frac{c}{2})$ (where $c$ is the exponent in the statement of Theorem~\ref{thmunpredic}) uniformly over $k$. We obtain
\begin{align} \label{upperboundderivativeG}
    \sup_{t \in (- \infty , \frac{c}{2})} G_k'(t) & = \sup_{t \in (- \infty , \frac{c}{2})}  \E_{\mathfrak{M}_k \otimes \mathfrak{M}_k} \left[ |\gamma_1 \cap \gamma_2| e^{t |\gamma_1 \cap \gamma_2|} \right] \notag \\
    & = \E_{\mathfrak{M}_k \otimes \mathfrak{M}_k} \left[ |\gamma_1 \cap \gamma_2| e^{\frac{c}{2} |\gamma_1 \cap \gamma_2|} \right] \notag \\
    & \leq C \sum_{k = 1}^\infty k e^{\frac{c}{2} k}e^{- c k} \\
    & \leq C, \notag
\end{align}
where we increased the value of the constant $C := C(d) < \infty$ on the right-hand side. Using~\eqref{expansionlambdanbeta}, we select a constant $\alpha := \alpha(d) > 0$ such that 
\begin{equation*}
    \frac{1}{\beta_1} + \frac{n}{\beta_2} \leq \alpha ~~ \implies ~~ - \ln \lambda_n(\beta_1 , \beta_2) \leq \frac{2}{\beta_1} + \frac{2n}{\beta_2} \leq \frac{c}{2}.
\end{equation*}
Under the assumption $1/\beta_1 + n/ \beta_2 \leq \alpha$, we may use~\eqref{upperboundderivativeG} and write
\begin{equation*}
    \mathbb{E}^{\mathrm{Nish}}_{\Lambda, \beta_1,\beta_2} \left[ \left| \mathbf{R}(\omega) \right|^2  \right] = G_k( - \ln \lambda_n(\beta_1 , \beta_2) ) \leq 1 - C \ln \lambda_n(\beta_1 , \beta_2) \leq 1 + 2C \left( \frac{1}{\beta_1} + \frac{n}{\beta_2} \right).
\end{equation*}
This completes the proof of~\eqref{ineq:concentration}, and thus of Proposition~\ref{propA16}.
\end{proof}

\subsubsection{Long-range order for the $XY$ model on the extended lattice}

In this section, we complete the proof of Proposition~\ref{propAppLRO.phasetransitionextended} by combining of Proposition~\ref{propA16} with two correlation inequalities for the $XY$ model: the Messager-Miracle-Sole inequality stated in Proposition~\ref{prop.corrineq} and the Messager-Miracle-Sole-Pfister inequality which is stated below and can be found in~\cite[Proposition 1]{MMP}.

In the following statement, we will use the notations in Definition~\ref{def.heterogeneousXY} (for the $XY$ model on the extended lattice) and in Definition~\ref{def.disorderedXYmodel} (for the disrodered $XY$ model on the extended lattice).

\begin{proposition}[Messager-Miracle-Sole-Pfister, Proposition 1 of~\cite{MMP}] \label{propMMSPineq}
    Let $\Lambda \subseteq \Z_d^n$, fix two inverse temperatures $\beta_1, \beta_2 > 0$ and let $\omega \in (0, 2\pi)^{\vec{E}(\Lambda)}$ be a disorder. Then for any pair of vertices $x , y \in \Lambda$, one has the inequality
    \begin{equation*}
        \left\langle \cos\left( \theta_x - \theta_y \right) \right\rangle_{\mu_{\Lambda , \beta_1 , \beta_2}} \geq \left\langle \cos\left( \theta_x - \theta_y \right) \right\rangle_{\mu_{\Lambda , \beta_1 , \beta_2, \omega}}.
    \end{equation*}
\end{proposition}

\begin{remark}
We mention that, in the article~\cite{MMP}, the correlation inequality is proved in a much more general framework (similar to the one introduced in Definition~\ref{generalXY}). 
\end{remark}

\begin{proof}[Proof of Proposition~\ref{propAppLRO.phasetransitionextended}]
    Let us fix an integer $n \in \N$, a finite box $\Lambda \subseteq \Zd_n$, an integer $k \in \N$ and assume that $0 \in \Lambda$ and $\vec{\textbf{k}} \in \Lambda$. Combining the result of Proposition~\ref{propA16} with the Messager-Miracle-Sole-Pfister inequality (Proposition~\ref{propMMSPineq}), we deduce that, for any pair of inverse temperatures $\beta_1 , \beta_2 \in (0, \infty)$,
    \begin{equation*}
        \left\langle \cos\left( \theta_{\vec{\textbf{k}}} - \theta_0 \right) \right\rangle_{\mu_{\Lambda , \beta_1 , \beta_2}} \geq  \mathbb{E}^{\mathrm{Nish}}_{\Lambda, \beta_1,\beta_2} \left[ \left\langle \cos \left( \theta_{\vec{\mathbf{k}}} - \theta_0 \right) \right\rangle_{\mu_{\Lambda, \beta_1 , \beta_2, \omega}} \right] \geq 1 - C \sqrt{ \frac{1}{\beta_1} + \frac{n}{\beta_2}}.
    \end{equation*}
    Since the constant $C$ does not depend on the parameter $n$, we may set $\beta_{1,c} := 8 C^2$ and $\beta_{2,c} := 8 n C^2$.
    For any pair of inverse temperatures $\beta_1 \geq \beta_{1,c}$ and $\beta_2 \geq \beta_{2,c}$ and any integer $k \in \N$, one has the estimate
    \begin{equation*}
        \left\langle \cos\left( \theta_{\vec{\textbf{k}}} - \theta_0 \right) \right\rangle_{\mu_{\Lambda , \beta_1 , \beta_2}} \geq \frac{1}{2}.
    \end{equation*}
    Taking the limit $\Lambda \uparrow \Zd_n$, the same inequality holds with the measure $\mu_{\beta_1 , \beta_2}$ (instead of $\mu_{\Lambda , \beta_1 , \beta_2}$). By the Messager-Miracle-Sole inequality (or more specifically, the inequality~\eqref{MMSimpliescomparison} stated in Remark~\ref{remarkA4}), there exists a constant $C := C(d) < \infty$ such that for any pair of vertices $x , y \in \Zd$ with $|x| \geq C |y|$,
    \begin{equation*}
        \left\langle \cos\left( \theta_{x} - \theta_0 \right) \right\rangle_{\mu_{\beta_1 , \beta_2}} \leq \left\langle \cos\left( \theta_{y} - \theta_0 \right) \right\rangle_{\mu_{\beta_1 , \beta_2}}.
    \end{equation*}
    For any vertex $x \in \Zd$, we select an integer $k \in \N$ such that $|\vec{\textbf{k}}| \geq C |x|$ and apply the two previous inequalities to obtain
    \begin{equation*}
        \left\langle \cos\left( \theta_{x} - \theta_0 \right) \right\rangle_{\mu_{\beta_1 , \beta_2}} \geq \left\langle \cos\left( \theta_{\vec{\textbf{k}}} - \theta_0 \right) \right\rangle_{\mu_{\beta_1 , \beta_2}} \geq \frac{1}{2}.
    \end{equation*}
We complete the proof by showing that a similar lower bound holds for all $x \in \Zd_n$ (and not only $x \in \Zd$) using the same argument as in Proposition~\ref{prop.dichotomy}. For each $x \in \Zd_n \setminus \Zd$, we denote by $(x)$ the vertex in $\Zd$ which is the closest to $x$ for the graph distance on $\Zd_n$. Applying the inequality~\eqref{eq:twopointbound}, we deduce that
\begin{equation*}
     \langle \cos(\theta_0 - \theta_x) \rangle_{\mu_{\beta_1 , \beta_2}} \geq \langle \cos(\theta_0 - \theta_{(x)}) \rangle_{\mu_{\beta_1 , \beta_2}} \langle \cos(\theta_{0} - \theta_{x- (x)}) \rangle_{\mu_{\beta_1 , \beta_2}}.
\end{equation*}
The first term is bounded from below by $1/2$ and the second is bounded from below by a positive real number uniformly over $x \in \Zd_n$. The proof of Proposition~\ref{propAppLRO.phasetransitionextended} is complete.
\end{proof}

\section{Existence of a BKT phase for two-component $\Phi^4$ model}
\label{a.phi4}

In this appendix, we explain another application of Wells' inequality: it implies the existence of a BKT phase for the two-component $\Phi^4$ model on $\Z^2$ which we define below. (N.B. In $d\geq 3$, long-range order for the $N$-component $\Phi^4$ model follows from reflection positivity, see \cite{frohlich1976infrared}). 
The applicability of Wells' inequality to $P(|\Phi|^2)$ model was already suggested in \cite{dunlop1985correlation}. Since we did not find anything further in the literature than  this remark in \cite{dunlop1985correlation}, we give the details here.  This short appendix may also be useful given the recent growing interest in the  $\Phi^4$ model on $\Z^d$ (mostly the one-component case), see in particular \cite{aizenman2021marginal,gunaratnam2022random,panis2023triviality} to which we refer the reader. 
We start with a definition of the relevant objects. 

Fix $g > 0$, $h \in \R$ and define the single-site measure in $\R^2$ according to the formula
\begin{equation*}
    d \rho_{g , h} (S) := e^{- g |S|^4 - h |S|^2} dS,
\end{equation*}
where $dS$ denotes the Lebesgue measure on $\R^2$, and $|S|$ is the Euclidean norm of $S \in \R^2$.

\begin{definition}[Two-components $\Phi^4$-model]
Let $\Lambda \subseteq \Zd$ be a finite set, and let $\beta>0$. 
We  define the finite-volume Gibbs measure on the space $(\R^2)^{\Lambda}$ according to the formula
    \begin{equation} \label{eq:17141905Phi4}
        \mu_{\Lambda, \beta , g , h}^{\Phi^4} (dS) := \frac{1}{Z_{\Lambda, \beta , g , h}}\exp \left( \sum_{\substack{ x , y \in \Lambda \\ x \sim y}} \beta \, S_x \cdot S_y \right) \prod_{x \in \Lambda}  d \rho_{g , h} (S_x),
    \end{equation}
where $Z_{\Lambda, \beta , g , h}$ is the normalizing constant. We denote by $\left\langle \cdot \right\rangle_{\mu_{\Lambda, \beta , g , h}}$ the expectation with respect to the measure $\mu_{\Lambda, \beta , g , h}$.
\end{definition}

\begin{remark}
    Using the decomposition $S_x = (R_x \cos \theta_x , R_x \sin \theta_x)$ with $R_x \geq 0$ and $\theta_x \in [0 , 2\pi)$, we may rewrite the measure~\eqref{eq:17141905Phi4} as follows
    \begin{equation*}
        \mu_{\Lambda, \beta , g , h}^{\Phi^4} (dR d \theta) := \frac{1}{Z_{\Lambda, \beta , g , h}}\exp \left( \sum_{\substack{ x , y \in \Lambda \\ x \sim y}} \beta \, R_x R_y \cos(\theta_x - \theta_y) \right) \prod_{x \in \Lambda}  \left( R_x e^{- g R_x^4 - h R_x^2} d R_x \right) d \theta_x.
    \end{equation*}
    This model fits in the general framework introduced in Definition~\ref{def.finite-volHamiltonian}.
\end{remark}

Wells' inequality readily implies the following result (which we state in a finite volume form and with free boundary conditions to avoid dealing with infinite-volume limits. Those are less immediate in the present setting due to the fact spins are no longer bounded). 
\begin{theorem}[An observation from \cite{dunlop1985correlation}]\label{}
For any $g>0$, $h\in \R$, there exists $\beta_{BKT}(g,h)<\infty$ and $c := c(g , h )>0$ such that for any $\beta>\beta_{BKT}(g,h)$, any $R\geq 1$, and any $x,y \in \Lambda_R:=\{-R,\ldots,R\}^2$, one has 
\begin{align*}\label{}
  \mu_{\Lambda_{R}, \beta , g , h}^{\Phi^4} \left( S_x \cdot S_y \right) \geq \frac {c} {|x-y|}\,.
\end{align*}
\end{theorem}

\begin{proof}
Notice that the Gibbs measure~\eqref{eq:17141905Phi4} is exactly of the form of an $XY$ model with annealed disorder introduced in Definition \ref{def.finite-volHamiltonian} with the choices:
\bi
\item $J(\textbf{1}_x - \textbf{1}_y , \textbf{1}_x + \textbf{1}_y)=\beta$ if $x \sim y$ and $J(m , A) = 0$ otherwise.
\item $d \kappa(r) =  e^{-g r^4  - h r^2} 2\pi r dr $.  
\ei
In particular, Wells' inequality  (Proposition \ref{prop.Wells}) applies with a constant $a=a(\kappa)>0$ which thus depends on $g,h$ the parameters of the measure $\kappa$. (See the proof of Wells' inequality in Section \ref{sectinoWellsineq} for the positivity of $a(\kappa)$).  Also since $g>0$, notice that $\kappa$ has sub-Gaussian tails. 

This implies that for any $x,y \in \Lambda \subset \Z^2$, 
\begin{align*}\label{}
   \mu_{\Lambda, \beta , g , h}^{\Phi^4} \big(S_x \cdot S_y\big) & =     \left\langle r_x r_y \cos (\theta_x - \theta_y) \right\rangle_{\mu_{\Lambda, \beta, \kappa}} \\
 & \geq 
   \left\langle r_x r_y \cos (\theta_x - \theta_y ) \right\rangle_{\mu_{\Lambda, \beta, \delta_{a}}}  \\
 & =   a^2 
 \left\langle  \cos (\theta_x - \theta_y ) \right\rangle_{\mu_{\Lambda, a \beta}}\,,
\end{align*}
which decays as a power-law when $a \beta \geq \beta_{BKT}$ by \cite{frohlich1981kosterlitz,van2023elementary,aizenman2021depinning}. 
\end{proof}

\medskip
\noindent
\textit{Acknowledgments.}

We wish to thank Diederik van Engelenburg, Avelio Sep\'{u}lveda and Hugo Vanneuville for useful discussions. C.G. also wishes to thank Yuval Peres for a fruitful discussion regarding the EIT property. The research of C.G. is supported by the Institut Universitaire de France (IUF), the ERC grant VORTEX 101043450 and the French ANR grant ANR-21-CE40-0003.

\medskip

\subsubsection*{Data availability} No data has been used in the research described here.

\bigskip

\noindent
\textbf{Declarations:}

\bigskip

\subsubsection*{Funding:} This work was supported in part the Institut Universitaire de France (IUF), the ERC grant VORTEX 101043450 and the French ANR grant ANR-21-CE40-0003.

\medskip

\subsubsection*{Conflict of interest:} The authors have no competing interests to declare that are relevant to the content of this article.

\medskip

\subsubsection*{Financial or non-financial interests:} All authors certify that they have no affiliations with or involvement in any organization or entity with any financial interest or non-financial interest in the subject matter or materials discussed in this manuscript.

\medskip

\subsubsection*{Financial or proprietary interests:} The authors have no financial or proprietary interests to disclose.

\bibliographystyle{alpha}
\bibliography{bibliography}

\end{document}